\documentclass[preprint,12pt]{elsarticle}

\usepackage{amssymb,amsmath,amsthm}
\usepackage{lineno,hyperref,color,bm}
\usepackage{booktabs}

\usepackage{graphicx}
\usepackage{subcaption}
\usepackage{caption}
\DeclareCaptionSubType{figure}

\usepackage[ruled]{algorithm2e}

\usepackage{geometry}
\usepackage{empheq}
\newtheorem{theorem}{Theorem}[section]
\newtheorem{lemma}[theorem]{Lemma}
\newtheorem{remark}[theorem]{Remark}
\newtheorem{corollary}[theorem]{Corollary}
\newtheorem{proposition}[theorem]{Proposition}
\newtheorem{definition}[theorem]{Definition}
\newtheorem{assumption}{Assumption}[section]  

\makeatletter
\def\ps@pprintTitle{%
	\let\@oddhead\@empty
	\let\@evenhead\@empty
	\let\@oddfoot\@empty
	\let\@evenfoot\@empty
}
\makeatother

\begin{document}

\begin{frontmatter}



\title{Sparse Recovery via $\ell_p^p/\ell_q^p$ Ratio Minimization: Theory and Algorithm}

%

\author[label1]{Lang Yu} 

\author[label1]{Nanjing Huang\corref{mycorrespondingauthor}} 

\cortext[mycorrespondingauthor]{Corresponding author}
\ead{nanjinghuang@hotmail.com}
\address[label1]{School of Mathematics, Sichuan University, 610064, Chengdu, China}

\begin{abstract}
The constrained $\ell_p^p/\ell_q^p$ ratio model is scale invariant and is therefore attractive for sparse signal recovery. However, its nonconvex, nonsmooth, and fractional structure makes a unified theoretical and algorithmic analysis challenging for $0<p\le1$ and $q>1$. This paper develops a unified framework for this general model, covering deterministic exact recovery, stable recovery for sparse and compressible signals, and convergence analysis of a fractional algorithm. We first establish two deterministic sufficient conditions for exact recovery: a local optimality criterion and a null-space condition ensuring uniform recovery. For the $\ell_1/\ell_q$ subfamily, this null-space condition is further converted into high-probability sample-complexity bounds for isotropic sub-Gaussian matrix. We then study noisy recovery. Under the $k$-sparsity assumption, we improve the RIP-based stable recovery theory by relaxing the required sufficient condition and deriving sharper reconstruction-error bounds. For compressible signals, we establish RIP--ROP based error estimates whose constants are independent of the ambient dimension, improving prior bounds with explicit dimension-dependent factors \cite{zhu2026stable}. An RIP-only variant is also derived. On the algorithmic side, we propose a prox-linear Dinkelbach framework that directly handles the fractional structure of the constrained problem and prove its convergence. Numerical experiments demonstrate that suitable choices of $(p,q)$ are effective for high-dynamic-range sparse signals and coherent sensing matrices.

\end{abstract}



\begin{keyword}
Compressed Sensing \sep Sparse Recovery \sep $\ell_p^p/\ell_q^p$ Ratio minimization \sep Exact and Stable Recovery \sep Restricted Isometry Property \sep Restricted Orthogonality Property


\end{keyword}

\end{frontmatter}



\section{Introduction}\label{sec:introduction}

Compressed sensing (CS) \cite{donoho2006compressed,candes2008introduction} aims to recover a sparse or compressible signal $\bm{x}\in\mathbb{R}^{n}$ from far fewer linear measurements than the ambient dimension. In the standard observation model $\bm{b} = \bm{A}\bm{x} + \bm{e}$, where $\bm{b}\in\mathbb{R}^{m}$ is the observation vector, $\bm{A}\in\mathbb{R}^{m\times n}$ with $m\ll n$ is the sensing matrix, and $\bm{e}\in\mathbb{R}^{m}$ is a bounded noise vector. The central task is to exploit the sparsity of $\bm{x}$ and reconstruct it from this underdetermined system. The most direct formulation is $\ell_0$-minimization, but it is NP-hard in general \cite{natarajan1995sparse}. A standard convex relaxation replaces the $\ell_0$ measure by the $\ell_1$ norm, leading to basis pursuit and its noisy variants \cite{donoho2006compressed,chen2001atomic}. Under conditions such as the null space property (NSP) and the restricted isometry property (RIP), $\ell_1$-minimization enjoys exact, stable, and robust recovery guarantees \cite{candes2008introduction,candes2006stable}. Nevertheless, the $\ell_1$ norm is only a convex surrogate for sparsity, and it may produce biased estimates or suboptimal recovery, especially for coherent sensing matrices or high-dynamic-range signals \cite{fannjiang2012coherence}.

To obtain sharper approximations of the $\ell_0$ measure, many nonconvex sparsity-promoting models have been investigated, including $\ell_p$-minimization with $0<p<1$ \cite{zhang2019optimal,zeng2014l}, transformed $\ell_1$ penalties \cite{zhang2018minimization}, the minimax concave penalty (MCP)~\cite{zhang2010nearly}, capped-$\ell_1$~\cite{zhang2010analysis}, the smoothly clipped absolute deviation (SCAD) penalty~\cite{fan2001variable}, and iterative reweighted schemes \cite{lai2013improved}. These models often promote sparsity more aggressively than convex $\ell_1$-minimization, but their nonconvex and nonsmooth structures make both recovery analysis and algorithmic convergence substantially more delicate.

Beyond separable penalties, non-separable models have also been proposed, including the difference-based penalties $\ell_1-\ell_2$ \cite{yin2014ratio,lou2018fast}, $\ell_1^2-\eta\ell_2^2$ \cite{li2026l12}, $\ell_1-\alpha \ell_2$ \cite{ge2021dantzig}, $\ell_1-\beta\ell_q$ \cite{huo2023minimization}, weighted $\ell_r-\alpha\ell_1$~\cite{zhou2023rip} with $r\in(0,1)$, and the more general models $\ell_p-\alpha\ell_q$ \cite{gu2024nonconvex} with $p\in(0,1]$, $q\in[1,2]$, $\alpha\in[0,1]$, and weighted $\ell_q-\alpha\ell_p$ \cite{zhan2025sparse} with $q\in(0,1]$, $p\in[q,+\infty]$, $\alpha\in[0,1]$. A second important family is given by ratio-based models, such as $\ell_1/\ell_2$ \cite{rahimi2019scale,wang2020accelerated}, $\ell_1/\ell_{\infty}$ \cite{wang2023variant}, and $\ell_1^2/\ell_2^2$ \cite{jia2025sparse}. More general ratio-type models include the $q$-ratio CMSV \cite{zhou2019q}, the $q$-ratio sparsity measure $(\ell_1/\ell_q)^{q/(q-1)}$ with $1<q\leq\infty$ \cite{zhou2021minimization}, $\ell_p/\ell_1$ with $0<p<1$ \cite{zhou2025recovery}, SPOQ \cite{cherni2020spoq}, and $\ell_p/\ell_q$ with $0<p\leq1$ and $q>1$ \cite{zhu2026stable}. Ratio-based models are particularly attractive because they preserve the scale invariance of the $\ell_0$ measure and can reduce the amplitude bias caused by homogeneous convex relaxations. Their highly nonconvex and nonsmooth nature, however, leaves several theoretical and algorithmic issues unresolved.

Many of the preceding ratio models can be viewed as special cases or variants of $\ell_p/\ell_q$. However, existing studies have mostly focused on specific choices of $p$ and $q$ \cite{repetti2014euclid,wang2021limited,zeng2021analysis,tao2022minimization}, leaving the recovery theory for the general $\ell_p^p/\ell_q^p$ model with $0<p\leq1$ and $q>1$ far from complete. Recently, Zhu et al. \cite{zhu2026stable} investigated the following generalized norm-ratio $\ell_p^p/\ell_q^p$ minimization problem
\begin{subequations}\label{eq:lp_lq_models}
	\begin{empheq}[left=\empheqlbrace]{align}
		&\mathcal{Q}_{p,q}: \inf_{\bm{x} \in \mathbb{R}^n\setminus\{\bm{0}\}} \left\{ \frac{\|\bm{x}\|_p^p}{\|\bm{x}\|_q^p} \mid \bm{A}\bm{x} = \bm{b} \right\}, \label{eq:lp_lq} \\
		&\mathcal{Q}_{p,q}^{\delta}: \inf_{\bm{x} \in \mathbb{R}^n\setminus\{\bm{0}\}} \left\{ \frac{\|\bm{x}\|_p^p}{\|\bm{x}\|_q^p} \mid \|\bm{A}\bm{x} - \bm{b}\|_2 \leq \delta \right\}, \label{eq:lp_lq_eps}
	\end{empheq}
\end{subequations}
where $0<p\le 1$ and $q>1$. They established stable recovery guarantees for such a problem based on the restricted isometry property (RIP) and the restricted orthogonality property (ROP). The main results of \cite{zhu2026stable} can be stated as follows: (i) Theorem 1 in \cite{zhu2026stable} derives a sufficient RIP condition expressed in terms of the RIC $\delta_{2k}$ and an associated error bound under the assumption that the true signal is $k$-sparse; (ii) Corollary 1 in \cite{zhu2026stable} is a special case to the important $\ell_1/\ell_q$ model; (iii) Theorem 2 in \cite{zhu2026stable} obtains a stable recovery guarantee by combining RIC and restricted orthogonality constant (ROC) estimates, in which the error bound depends on the dimension of the recovery signal, such a dependence is undesirable in compressed sensing. Thus, the first goal of this paper is to develop some new techniques  to obtain an error bound which is independent of the ambient dimension. 

Sampling complexity is another important aspect in the study of CS. Deterministic conditions such as NSP, RIP, ROP, or norm-ratio lower bounds clarify which geometric properties of the sensing matrix ensure recovery, but they do not by themselves specify how many random measurements are sufficient for these properties to hold. We would like to mention that Xu et al. \cite{xu2021analysis} gave the geometric analysis for lower bound of sample complexity of \eqref{eq:lp_lq} with $p=1$ and $q=2$, including local optimality, uniform recoverability, and random-matrix guarantees.  However, to the best of our knowledge, there is no report on the lower bound of sample complexity for \eqref{eq:lp_lq}. Therefore, our second goal is to give a lower bound of sample complexity for \eqref{eq:lp_lq} under some mild conditions.

It is well known that how to design algorithms with convergence guarantees for the generalized ratio model is an essential issue. Existing methods for ratio-based sparse recovery are largely designed for special structures. There have been numerous studies addressing the algorithms for solving the ratio models and some special cases of \eqref{eq:lp_lq_models}, as well as their regularized problem. For example, Zhou \cite{zhou2025recovery} solved the $\ell_p/\ell_1$ minimization by fast iterative shrinkage-thresholding method (FISTA); Zhan et al. \cite{zhan2025p} combined iterative reweighting, DCA, and bisection to apply weighted $\ell_p/\ell_q$ minimization to block $k$-sparse recovery;  Cherni et al. \cite{cherni2020spoq} introduced a logarithmic smoothing of the $\ell_p^p/\ell_q^p$ ratio and embedded the smoothed model into a trust-region majorize-minimize framework with convergence to a critical point under the Kurdyka--\L{}ojasiewicz property; Wang et al. \cite{wang2023variant} derived a closed-form proximal operator for the $\ell_1/\ell_\infty$ regularizer corresponding to \eqref{eq:lp_lq} with $p=1$ and $q=\infty$, and developed FISTA-type and ADMM-type algorithms for the associated regularized problem; and further Wang et al. \cite{wang2020accelerated} developed an adaptive Dinkelbach-type framework for \eqref{eq:lp_lq} with $p=1$ and $q=2$, together with a convergence analysis. It is worth mentioning that the work mentioned above relies on convex numerators, special denominators, smoothing devices, or model-specific proximal structures. Thus, the above algorithms can not be directly used to address \eqref{eq:lp_lq}. The third motivation of this paper is to design an effective algorithm for \eqref{eq:lp_lq} by employing the prox-linear Dinkelbach scheme and to establish its convergence.

The present paper is thus devoted to the studies of theory and algorithm for \eqref{eq:lp_lq_models} under some mild conditons. We obtain some improved recovery conditions and error bounds, and develop an implementable algorithm for \eqref{eq:lp_lq}. The main contributions of this paper can be summarized as follows:
 
\begin{enumerate}[{\rm(i)}]
	\item We develop a local optimality criterion for \eqref{eq:lp_lq}, which shows how the dynamic range of a sparse signal affects local recovery and gives a sufficient condition for uniform exact recovery. For the special case of \eqref{eq:lp_lq} with $p=1$, we further derive  high-probability sample complexity bounds under isotropic sub-Gaussian measurements.
	\item For the $k$-sparse setting of \eqref{eq:lp_lq_eps}, we establish an RIP-based stable recovery guarantee by using several new inequalities. Compared with \cite{zhu2026stable}, our result requires a weaker RIP condition and gives a tighter reconstruction error bound. 
	\item For compressible signals without the $k$-sparsity assumption of \eqref{eq:lp_lq_eps}, we refine the existing RIP--ROP recovery theory by removing its explicit dependence on the ambient dimension, and further we obtain an RIP-only guarantee by eliminating the ROC condition.
	\item We propose a prox-linear Dinkelbach scheme for \eqref{eq:lp_lq} and prove its convergence by combining a Dinkelbach-type transformation with a proximal quadratic term to guarantee the descent of the sequence.
\end{enumerate}

The rest of this paper is organized as follows. Section \ref{sec:Preliminaries} recalls the notation and basic concepts used throughout the paper. Section \ref{sec:LocalOptimality} establishes a local optimality condition for the $\ell_p^p/\ell_q^p$ model. Section \ref{sec:ExactRecovery} studies uniform exact recovery and derives high-probability guarantees for the $\ell_1/\ell_q$ case. Section \ref{sec:ReDCENprop} proves stable recovery under the RIP and the $k$-sparsity assumption, and compares the obtained bounds with existing results. Section \ref{sec:StablerecoverywithoutK} gives stable recovery results for compressible signals without the $k$-sparsity assumption. Section \ref{sec:Algorithm} presents the prox-linear Dinkelbach scheme and proves its convergence. Section \ref{sec:NumericalExperiments} reports numerical experiments. Finally, Section \ref{sec:Conclusion} concludes the work and discusses future work.

\section{Notation and Preliminaries}\label{sec:Preliminaries}

We collect in this section the notation and basic notions that will be used in the sequel. For a positive integer $n$, we write $[n]:=\{1,2,\ldots,n\}$. Given an index set $\mathcal{S}\subseteq [n]$, its complement in $[n]$ is denoted by $\mathcal{S}^c$. For a vector $\bm{u}\in\mathbb{R}^n$, the restricted vector $\bm{u}_{\mathcal{S}}$ is defined by $(\bm{u}_{\mathcal{S}})_i=u_i\mathbf{1}_{\{i\in\mathcal{S}\}}$ for each $i$. The support of the vector $\bm{u}$ is $\operatorname{supp}(\bm{u}):=\{i\in[n]: u_i\neq 0\}$, and $\bm{u}$ is said to be $s$-sparse if $\operatorname{card}(\operatorname{supp}(\bm{u}))\le s$. Equivalently, $\|\bm{u}\|_0:=\operatorname{card}(\operatorname{supp}(\bm{u}))$ counts the number of nonzero components of $\bm{u}$.

Throughout the paper, vectors and matrices are usually denoted by bold lowercase and bold uppercase letters, respectively. 
The zero vector is denoted by $\bm{0}$, and $\mathbb{R}^{m\times n}$ stands for the set of all real $m\times n$ matrices. 
For $0<p<\infty$, we use $\|\bm{u}\|_p:=\left(\sum_{i=1}^n |u_i|^p\right)^{1/p}$, while $\|\bm{u}\|_\infty:=\max_{1\le i\le n}|u_i|$.
When $0<p<1$, $\|\cdot\|_p$ is not a norm in the usual sense; nevertheless, following the common convention in sparse recovery, we still refer to it as the $\ell_p$ norm or quasi-norm when no ambiguity arises.

For later use, we also introduce the best sparse approximation error. 
For $p>0$ and $\bm{u}\in\mathbb{R}^n$, the best $s$-term approximation error of $\bm{u}$ in the $\ell_p$ norm is $\sigma_s(\bm{u})_p:=\inf_{\|\bm{z}\|_0\le s}\|\bm{u}-\bm{z}\|_p$. In particular, if $\bm{u}_{\max(s)}$ denotes the vector obtained by retaining the $s$ largest entries of $\bm{u}$ in magnitude and setting the others to zero, then $\bm{u}_{-\max(s)}:=\bm{u}-\bm{u}_{\max(s)}$
is the tail part of $\bm{u}$ after its best $s$ dominant entries are removed. We shall use the following null-space-type properties in the analysis.

\begin{definition}[$\ell_p$ null space property, $\ell_p$-NSP]
	Let $0<p\le 1$ and let $s$ be a positive integer. 
	A matrix $\bm{A}\in\mathbb{R}^{m\times n}$ is said to satisfy the $\ell_p$ null space property ($\ell_p$-NSP) of order $s$ if there exists a constant $\mu\in(0,1)$ such that $\|\bm{v}_{\mathcal{S}}\|_p^p <\mu \|\bm{v}_{\mathcal{S}^c}\|_p^p$ for every $\bm{v}\in\ker(\bm{A})\setminus\{\bm{0}\}$ and every index set $\mathcal{S}\subseteq[n]$ with $\operatorname{card}(\mathcal{S})\le s$.
\end{definition}

\begin{definition}[$\ell_q$ robust null space property, $\ell_q$-rNSP]
	Let $q\ge 1$, and let $\|\cdot\|$ be a norm on the measurement space. A matrix $\bm{A}\in\mathbb{R}^{m\times n}$ is said to satisfy the $\ell_q$ robust null space property ($\ell_q$-rNSP) of order $s$ with constants $\rho\in(0,1)$ and $\tau>0$ if $\|\bm{v}_{\mathcal{S}}\|_q \le \frac{\rho}{s^{1-1/q}}\|\bm{v}_{\mathcal{S}^c}\|_1	+ \tau \|\bm{A}\bm{v}\|$ holds for all $\bm{v}\in\mathbb{R}^n$ and all index sets $S\subseteq[n]$ satisfying $\operatorname{card}(\mathcal{S})\le s$.
\end{definition}

We next recall two matrix properties frequently used in compressed sensing. 
The first one quantifies the extent to which a sensing matrix preserves the Euclidean norm of sparse vectors.

\begin{definition}[Restricted isometry property, RIP]
	Let $\bm{A}$ be a real matrix and let $s$ be a positive integer. 
	The restricted isometry constant (RIC) of order $s$, denoted by $\delta_s$, is the smallest number for which $(1-\delta_s)\|\bm{u}\|_2^2	\le	\|\bm{A}\bm{u}\|_2^2	\le	(1+\delta_s)\|\bm{u}\|_2^2$ holds for all $s$-sparse vectors $\bm{u}$. If $\delta_s\in[0,1)$, then $\bm{A}$ is said to satisfy the $s$-restricted isometry property, abbreviated as $s$-RIP.
\end{definition}

The second property controls the near-orthogonality of the images of sparse vectors with prescribed sparsity levels.

\begin{definition}[Restricted orthogonality property, ROP]
	Let $s_1$ and $s_2$ be positive integers. 
	The restricted orthogonality constant (ROC) $\theta_{s_1,s_2}$ of a matrix $\bm{A}$ is the smallest nonnegative number such that $\|\langle \bm{A}\bm{u},\bm{\Phi}\bm{v}\rangle\| \le \theta_{s_1,s_2}\|\bm{u}\|_2\|\bm{v}\|_2$ for all $s_1$-sparse vectors $\bm{u}$ and all $s_2$-sparse vectors $\bm{v}$ with disjoint supports whenever such a disjointness assumption is required. In this case, $\bm{A}$ is said to satisfy the $(s_1,s_2)$-ROP with ROC $\theta_{s_1,s_2}$.
\end{definition}

Finally, we recall several standard notions from variational analysis that are needed for the algorithmic part of the papeet $F:\mathbb{R}^n\to(-\infty,+\infty]$ be an extended-real-valued function. 
Its effective domain is $\operatorname{dom}F:=\{\bm{x}\in\mathbb{R}^n:\ F(\bm{x})<+\infty\}$. The function $F$ is called proper if $\operatorname{dom}F\neq\varnothing$ and $F(\bm{x})>-\infty$ for all $\bm{x}\in\mathbb{R}^n$. It is lower semicontinuous (l.s.c) at $\bm{x}$ if, for every sequence $\bm{x}^k\to \bm{x}$, $F(\bm{x})\le \liminf_{k\to\infty} F(\bm{x}^k)$. If this property holds at every point, then $F$ is said to be l.s.c. A function is called closed if it is proper and l.s.c. For a point $\bm{x}^*\in\operatorname{dom}F$, we call $\bm{x}^*$ a critical point, or stationary point, of $F$ if $\bm{0}\in \partial F(\bm{x}^*)$, where $\partial F$ denotes the limiting subdifferential.

\section{Local Optimality Criterion of $\mathcal{Q}_{p,q}$- Minimization}\label{sec:LocalOptimality}

In this section, we derive a strict local optimality condition that ensure a non-zero $s$-sparse vector $\bm{x}_0$ is a strict local minimizer of \eqref{eq:lp_lq}. We start with the definition of the generalized norm-ratio coefficient.
\begin{definition}[Generalized Norm-Ratio Coefficient]
	Let $\bm{x}_0 \in \mathbb{R}^n$ be a non-zero vector. We define its generalized norm-ratio coefficient (GNRC) $\kappa_{p,q}(\bm{x}_0)$ as
	\begin{equation*}
		\kappa_{p,q}(\bm{x}_0) := \frac{\|\bm{x}_0\|_p^p \|\bm{x}_0\|_\infty^{q-p}}{\|\bm{x}_0\|_q^q}. \label{eq:def_kappa}
	\end{equation*}
\end{definition}

In the next theorem, we show that the local optimality criterion of $\mathcal{Q}_{p,q}$. Specifically, there exists a strictly positive, uniform radius $\delta_r > 0$ such that for any $\bm{h} \in \ker(\bm{A}) \setminus \{\bm{0}\}$ with $\|\bm{h}\|_\infty \le \delta_r$, we have $\frac{\|\bm{x}_0 + \bm{h}\|_p^p}{\|\bm{x}_0 + \bm{h}\|_q^p} > \frac{\|\bm{x}_0\|_p^p}{\|\bm{x}_0\|_q^p}$.
\begin{theorem}[Local optimality]\label{thm:local_optimality}
	Let $\bm{x}_0\in\mathbb{R}^n\setminus\{\bm{0}\}$ be an $s$-sparse vector with support $\mathcal{S}:=\operatorname{supp}(\bm{x}_0)$, and set $\bm{b}=\bm{A}\bm{x}_0$. Assume that $\bm{A}\in\mathbb{R}^{m\times n}$ satisfies the $\ell_p$-NSP of order $s$ with a constant $\mu\in(0,1)$. Then $\bm{x}_0$ is a strict local minimizer of \eqref{eq:lp_lq} over the feasible set $\{\bm{x}\in\mathbb{R}^n:\bm{A}\bm{x}=\bm{b}\}$ under the following conditions:
	\begin{enumerate}[{\rm(i)}]
		\item If $p=1$, assume in addition that $\mu \le	\frac{1}{1+\kappa_{1,q}(\bm{x}_0)}$.
		\item If $0<p<1$, no further restriction on $\mu$ is required beyond the $\ell_p$-NSP condition.
	\end{enumerate}
\end{theorem}

\begin{proof}
	Let $\bm{h}\in\ker(\bm{A})\setminus\{\bm{0}\}$. 
	Since every feasible point sufficiently close to $\bm{x}_0$ can be written as $\bm{x}_0+\bm{h}$ with $\bm{h}\in\ker(\bm{A})$, it is enough to prove that $\frac{\|\bm{x}_0+\bm{h}\|_p^p}{\|\bm{x}_0+\bm{h}\|_q^p} > \frac{\|\bm{x}_0\|_p^p}{\|\bm{x}_0\|_q^p}$ for all nonzero $\bm{h}\in\ker(\bm{A})$ with sufficiently small norm. Equivalently, we shall show
	\begin{equation}\label{eq:ratio_ineq}
		\frac{\|\bm{x}_0+\bm{h}\|_p^p}{\|\bm{x}_0\|_p^p}
		>
		\left(
		\frac{\|\bm{x}_0+\bm{h}\|_q^q}{\|\bm{x}_0\|_q^q}
		\right)^{p/q}.
	\end{equation}
	
	Choose $\delta_r>0$ such that $0<\delta_r<\min_{i\in\mathcal{S}} |x_{0,i}|$ and $\text{sgn}(x_{0,i} + h_i) = \text{sgn}(x_{0,i})$ for all $i \in \mathcal{S}$. We restrict attention to perturbations satisfying $\|\bm{h}\|_\infty\le\delta_r$. By the subadditivity of $t\mapsto |t|^p$ for $0<p\le 1$, one has $|x+h|^p \ge |x|^p - |h|^p$. Splitting the vector into the support $\mathcal{S}$ and its complement $\mathcal{S}^c$, we have
	\begin{equation}
		\|\bm{x}_0 + \bm{h}\|_p^p = \sum_{i \in \mathcal{S}} |x_{0,i} + h_i|^p + \sum_{i \in \mathcal{S}^c} |h_i|^p \ge \|\bm{x}_0\|_p^p - \|\bm{h}_{\mathcal{S}}\|_p^p + \|\bm{h}_{\mathcal{S}^c}\|_p^p. \label{eq:lp_lower_bound}
	\end{equation}
	
	Next we estimate the $\ell_q$ term. 
	For $i\in\mathcal{S}$, the mean value theorem applied to $t\mapsto |t|^q$ gives $|x_{0,i}+h_i|^q-|x_{0,i}|^q \le q\bigl(|x_{0,i}|+|h_i|\bigr)^{q-1}|h_i|$. Thus, over the support $\mathcal{S}$, one has
	\begin{equation*}
		\sum_{i \in \mathcal{S}} |x_{0,i} + h_i|^q \le \|\bm{x}_0\|_q^q + q (\|\bm{x}_0\|_\infty + \delta_r)^{q-1} \sum_{i \in \mathcal{S}} |h_i|.
	\end{equation*}
	For $i \in \mathcal{S}^c$, it follows from $\|\bm{h}\|_\infty\le\delta_r$ that $|h_i|^q = |h_i|^{q-p} |h_i|^p \le \delta_r^{q-p} |h_i|^p$. 
	We rewrite $|h_i|$ as $|h_i| = |h_i|^{1-p} |h_i|^p \le \delta_r^{1-p} |h_i|^p$ for any $i\in\mathcal{S}$. This yields an upper bound
	\begin{equation}\label{eq:lq_upper_bound}
		\|\bm{x}_0+\bm{h}\|_q^q	\le	\|\bm{x}_0\|_q^q + q(\|\bm{x}_0\|_\infty+\delta_r)^{q-1}	\delta_r^{1-p}
		\|\bm{h}_{\mathcal{S}}\|_p^p + \delta_r^{q-p} \|\bm{h}_{\mathcal{S}^c}\|_p^p.
	\end{equation}
	Using Bernoulli's inequality $(1+z)^\alpha\le 1+\alpha z$ for $z\ge0$ and $\alpha=p/q\in(0,1)$, \eqref{eq:lq_upper_bound} implies
	\begin{equation}\label{eq:lq_ratio_upper_bound}
		\left(\frac{\|\bm{x}_0+\bm{h}\|_q^q}{\|\bm{x}_0\|_q^q}\right)^{p/q} \le\,1 + p(\|\bm{x}_0\|_\infty+\delta_r)^{q-1}\delta_r^{1-p} \frac{\|\bm{h}_{\mathcal{S}}\|_p^p}{\|\bm{x}_0\|_q^q}+\frac{p}{q}
		\delta_r^{q-p}\frac{\|\bm{h}_{\mathcal{S}^c}\|_p^p}{\|\bm{x}_0\|_q^q}.
	\end{equation}
	Combining \eqref{eq:lp_lower_bound} and \eqref{eq:lq_ratio_upper_bound}, a sufficient condition for 
	\eqref{eq:ratio_ineq} is
	\begin{equation}\label{eq:sufficient_local_ineq}
		\|\bm{h}_{\mathcal{S}^c}\|_p^p
		\left(1-\frac{p}{q}\delta_r^{q-p}\frac{\|\bm{x}_0\|_p^p}{\|\bm{x}_0\|_q^q}\right) >\|\bm{h}_{\mathcal{S}}\|_p^p\left(1+p(\|\bm{x}_0\|_\infty+\delta_r)^{q-1}\delta_r^{1-p}\frac{\|\bm{x}_0\|_p^p}{\|\bm{x}_0\|_q^q}\right).
	\end{equation}
	
	It remains to verify that \eqref{eq:sufficient_local_ineq} holds uniformly for all sufficiently small nonzero $\bm{h}\in\ker(\bm{A})$. 
	By the $\ell_p$-NSP applied to $\mathcal{S}$, every nonzero $\bm{h}\in\ker(\bm{A})$ satisfies $\|\bm{h}_{\mathcal{S}}\|_p^p<\mu\|\bm{h}_{\mathcal{S}^c}\|_p^p$. In particular, $\|\bm{h}_{\mathcal{S}^c}\|_p>0$. Consider the normalized set $\mathcal{K}	:=	\{	\bm{h}\in\ker(\bm{A}): \|\bm{h}_{\mathcal{S}^c}\|_p=1\}$. Then $\mathcal{K}$ is nonempty and closed. Moreover, the NSP gives $\|\bm{h}_{\mathcal{S}}\|_p^p<\mu$ for all $\bm{h}\in\mathcal{K}$, so $\mathcal{K}$ is bounded. The continuous function $R(\bm{h}):=\frac{\|\bm{h}_{\mathcal{S}}\|_p^p}{\|\bm{h}_{\mathcal{S}^c}\|_p^p}$ therefore attains its maximum on $\mathcal{K}$. Denote this maximum by $\mu_*$. Since $R(\bm{h})<\mu$ pointwise on $\mathcal{K}$, compactness yields $\mu_*<\mu<1$. Consequently, for every nonzero $\bm{h}\in\ker(\bm{A})$, $\frac{\|\bm{h}_{\mathcal{S}}\|_p^p}	{\|\bm{h}_{\mathcal{S}^c}\|_p^p} \le\mu_*$. Dividing \eqref{eq:sufficient_local_ineq} by $\|\bm{h}_{\mathcal{S}^c}\|_p^p$, it is enough to show
	\begin{equation}\label{eq:delta_condition}
		1 -	\frac{p}{q}	\delta_r^{q-p}\frac{\|\bm{x}_0\|_p^p}{\|\bm{x}_0\|_q^q}	>\mu_* \left(1	+ p(\|\bm{x}_0\|_\infty+\delta_r)^{q-1}\delta_r^{1-p}	\frac{\|\bm{x}_0\|_p^p}{\|\bm{x}_0\|_q^q}\right).
	\end{equation}
	
	If $0<p<1$, then both $\delta_r^{q-p}$ and $\delta_r^{1-p}$ tend to zero as $\delta_r\to0^+$.  It follows that the left-hand side (LHS) of \eqref{eq:delta_condition} tends to $1$, while the right-hand side (RHS) tends to $\mu_*<1$. Hence \eqref{eq:delta_condition} holds for all sufficiently small $\delta_r>0$.  Thus, for $0<p<1$, the standard $\ell_p$-NSP alone is sufficient to guarantee strict local minimality, and no additional upper bound on the NSP constant $\mu$ is needed. If $p=1$, then $\delta_r^{1-p}=1$, and \eqref{eq:delta_condition} tends, as $\delta_r\to0^+$, to
	\begin{equation*}
		1>\mu_*	\left(1+\|\bm{x}_0\|_\infty^{q-1} \frac{\|\bm{x}_0\|_1}{\|\bm{x}_0\|_q^q}\right)=\mu_*\bigl(1+\kappa_{1,q}(\bm{x}_0)\bigr).
	\end{equation*}
	By assumption, $\mu_*<\mu\le\frac{1}{1+\kappa_{1,q}(\bm{x}_0)}$, and therefore the above strict inequality holds. Thus \eqref{eq:delta_condition} again holds for all sufficiently small $\delta_r>0$. This proves that every sufficiently small nonzero feasible perturbation strictly increases the objective value. 
	Hence $\bm{x}_0$ is a strict local minimizer of \eqref{eq:lp_lq}.
\end{proof}

The preceding theorem gives a signal-dependent local optimality condition through the coefficient $\kappa_{1,q}(\bm{x}_0)$ when $p=1$. For uniform statements over all $s$-sparse signals, it is useful to replace this quantity by its worst-case upper bound on the class of nonzero $s$-sparse vectors. The following corollary provides such a uniform estimate and consequently yields a signal-independent sufficient condition for local optimality.

\begin{corollary}[Uniform local optimality]\label{cor:uniform_local}
	Let $p=1$, $q>1$, and $s$ be a positive integer. For nonzero $s$-sparse vectors, the coefficient $\kappa_{1,q}$ satisfies
$$
		\sup_{\substack{\bm{x}\neq\bm{0},\; \|\bm{x}\|_0\le s}}	\kappa_{1,q}(\bm{x})=K_{1,q,s},
$$
	where $K_{1,q,1}=1$	when $s=1$, and when $s\ge2$, $K_{1,q,s}=\frac{1+(s-1)x^*}{1+(s-1)(x^*)^q}$. Here $x^*\in(0,1)$ is the unique positive solution of $1-qx^{q-1}	=	(q-1)(s-1)x^q$. Moreover, if $\bm{A}$ satisfies the $\ell_1$-NSP of order $s$ with a constant $\mu	\le	\frac{1}{1+K_{1,q,s}}$, then, for every nonzero $s$-sparse vector $\bm{x}_0$, the vector $\bm{x}_0$ is a strict local minimizer of the $\ell_1/\ell_q$ minimization over the feasible set $\{\bm{x}\in\mathbb{R}^n:\bm{A}\bm{x}=\bm{A}\bm{x}_0\}$.
\end{corollary}

\begin{proof}
	The case $s=1$ is immediate. Indeed, if $\bm{x}_0$ is nonzero and one-sparse, then $\kappa_{1,q}(\bm{x}_0)=1$,	and hence $K_{1,q,1}=1$.
	
	We now assume $s\ge2$. Let $\bm{x}_0\in\mathbb{R}^n\setminus\{\bm{0}\}$ be arbitrary with $\|\bm{x}_0\|_0\le s$. Since $\kappa_{1,q}$ is invariant under permutations, sign changes, and nonzero scalings, we may assume without loss of generality that $\|\bm{x}_0\|_\infty=1$ and $1=x_{0,1}\ge x_{0,2}\ge\cdots\ge x_{0,s} \ge 0$, where zero components are appended if the actual support size of $\bm{x}_0$ is smaller than $s$. Under this normalization,
$$
		\kappa_{1,q}(\bm{x}_0)=	\frac{1+\sum_{i=2}^s x_{0,i}} {1+\sum_{i=2}^s x_{0,i}^q}.
$$
	To rigorously find the global maximum over the hypercube $[0, 1]^{s-1}$, we define the arithmetic mean of the tail components
$$
		\nu	:= \frac{1}{s-1}\sum_{i=2}^s x_{0,i} \in[0,1].
$$
	Then the numerator is equal to $1+(s-1)\nu$. 
	Since $t\mapsto t^q$ is strictly convex for $q>1$, Jensen's inequality gives
	\begin{equation*}
		\frac{1}{s-1} \sum_{i=2}^s x_{0,i}^q \ge ( \frac{1}{s-1} \sum_{i=2}^s x_{0,i} )^q = \nu^q \implies \sum_{i=2}^s x_{0,i}^q \ge (s-1)\nu^q.
	\end{equation*}
	Applying this strict lower bound to the denominator, we have
$$
		\kappa_{1,q}(\bm{x}_0) \le \frac{1+(s-1)\nu}{1+(s-1)\nu^q}=:g(\nu).
$$
	Equality holds if and only if $x_{0,2}=x_{0,3}=\cdots=x_{0,s}=\nu$. This formulation inherently covers all $k \le s$ support sizes since $\nu$ spans down to $0$. It remains to maximize $g$ on $[0,1]$. Taking the derivative of $g(\nu)$ with respect to $\nu$, one has
$$
		g'(\nu) = \frac{s-1}{[1 + (s-1)\nu^q]^2} \left( 1 - q\nu^{q-1} - (q-1)(s-1)\nu^q \right).
$$
	Define $H(\nu):=(q-1)(s-1)\nu^q+q\nu^{q-1}-1$. Then $\operatorname{sgn}(g'(\nu))=-\operatorname{sgn}(H(\nu))$. Although $\nu$ may take the endpoint value $0$, the derivative of $H$ is considered on $(0,1]$. For $\nu\in(0,1]$, we have $H'(\nu)=q(q-1)(s-1)\nu^{q-1}+q(q-1)\nu^{q-2}>0$ for $ \nu\in(0,1]$.
	Hence $H$ is strictly increasing on $(0,1]$. Moreover, $\lim_{\nu\to0^+}H(\nu)=-1$ and $H(1)=(q-1)s>0.$ Therefore, there exists a unique $x^*\in(0,1)$ such that $H(x^*)=0$, namely, $1-q(x^*)^{q-1}=(q-1)(s-1)(x^*)^q$. Consequently, $g$ is strictly increasing on $(0,x^*)$ and strictly decreasing on $(x^*,1]$. Hence
$$
		g(x^*)=\frac{1+(s-1)x^*}{1+(s-1)(x^*)^q}
$$
	is the maximum of $g$ over $[0,1]$. Since the above argument holds for every nonzero $\bm{x}_0$ with $\|\bm{x}_0\|_0\le s$, and equality is attained by vectors whose nonzero magnitudes are proportional to $(1,x^*,x^*,\ldots,x^*)$, we obtain
	\begin{equation*}
		\sup_{\substack{\bm{x}_0\neq\bm{0}\\ \|\bm{x}_0\|_0\le s}} \kappa_{1,q}(\bm{x}_0)=\frac{1+(s-1)x^*}{1+(s-1)(x^*)^q}=K_{1,q,s}.
	\end{equation*}
	
	Finally, if $\bm{A}$ satisfies the $\ell_1$-NSP of order $s$ with $\mu \le \frac{1}{1+K_{1,q,s}}$, then for every nonzero $s$-sparse vector $\bm{x}_0$, one has
$$
		\mu	\le	\frac{1}{1+K_{1,q,s}} \le \frac{1}{1+\kappa_{1,q}(\bm{x}_0)}.
$$
	The conclusion follows from Theorem~\ref{thm:local_optimality}.
\end{proof}

\section{Exact Recovery for $\mathcal{Q}_{p,q}$ Minimization}\label{sec:ExactRecovery}
In this section, we propose a sufficient condition that guarantees the exact recovery for sparse vectors using $\mathcal{Q}_{p,q}$-minimization. As we will see, the condition to be obtained will hold with overwhelming probability for a large class of sub-Gaussian random matrices. 
\subsection{Uniform recoverability}
In this subsection, we turn to the uniform exact recovery property of $\mathcal{Q}_{p,q}$. Unlike the local optimality result, which depends on the particular sparse vector through signal-dependent quantities, a uniform recovery guarantee requires a condition ensuring that every $s$-sparse signal is recovered as the unique global minimizer. The following theorem gives such a sufficient condition in terms of a null-space norm-ratio lower bound. For convenience, we first recall the definition of sub-Gaussian and isotropic vectors.

\begin{definition}[Sub-Gaussian random vector]
	Let $\bm{x}$ be an $\mathbb{R}^n$-valued random vector. We call $\bm{x}$ sub-Gaussian if every linear functional of $\bm{x}$ has a finite sub-Gaussian Orlicz norm. More precisely, for each $\bm{g}\in\mathbb{R}^n$, the scalar random variable $\langle \bm{g},\bm{x}\rangle$ satisfies
	\begin{equation*}
		\|\langle \bm{g},\bm{x}\rangle\|_{\psi_2}:=\inf\{t>0:\mathbb{E}\exp(\frac{|\langle \bm{g},\bm{x}\rangle|^2}{t^2})\le 2\}<\infty .
	\end{equation*}
	The sub-Gaussian norm of the random vector $\bm{x}$ is then defined by $\|\bm{x}\|_{\psi_2}:=\sup_{\|\bm{g}\|_2=1}\|\langle \bm{g},\bm{x}\rangle\|_{\psi_2}$. In particular, if $\bm{x}\sim\mathcal{N}(\bm{0},\bm{I}_n)$, then every one-dimensional projection $\langle \bm{g},\bm{x}\rangle$ with $\|\bm{g}\|_2=1$ follows the standard normal distribution, and hence $\|\bm{x}\|_{\psi_2}=\|\mathcal{N}(0,1)\|_{\psi_2}$.
\end{definition}

\begin{definition}[Isotropic random vector]
	An $\mathbb{R}^n$-valued random vector $\bm{x}$ is called isotropic if its second-moment matrix equals the identity matrix, namely, $\mathbb{E}[\bm{x}\bm{x}^{T}]=\bm{I}_n$. Equivalently, for every $\bm{g}\in\mathbb{R}^n$, $\mathbb{E}[\langle \bm{g},\bm{x}\rangle^2]=\|\bm{g}\|_2^2$.
\end{definition}
We next turn to the uniform exact recovery property of the proposed norm-ratio model. 
Unlike the local optimality result, which depends on the particular sparse vector through signal-dependent quantities, a uniform recovery guarantee requires a condition ensuring that every $s$-sparse signal is recovered as the unique global minimizer. 
The following theorem gives such a sufficient condition in terms of a null-space norm-ratio lower bound.

\begin{theorem}[Uniform exact recovery]\label{thm:exact_recovery}
	Let $0<p\le 1$, $q>1$, and let $s\in\mathbb{N}$. 
	Assume that the measurement matrix $\bm{A}\in\mathbb{R}^{m\times n}$ satisfies
	\begin{equation}\label{eq:null_space_cond}
		\inf_{\bm{h}\in\ker(\bm{A})\setminus\{\bm{0}\}}
		\frac{\|\bm{h}\|_p}{\|\bm{h}\|_q}
		>
		3^{1/p}s^{1/p-1/q}.
	\end{equation}
	Then every nonzero $s$-sparse vector $\bm{x}_0$ is the unique global minimizer of $\mathcal{Q}_{p,q}$. Consequently, $\mathcal{Q}_{p,q}$ exactly recovers all nonzero $s$-sparse signals.
\end{theorem}

\begin{proof}
	Let $\bm{x}_0$ be an arbitrary nonzero $s$-sparse vector. Set $\mathcal{S}:=\operatorname{supp}(\bm{x}_0)$ and $|\mathcal{S}|\le s$.
	To prove exact recovery, it suffices to show that for every nonzero 
	$\bm{h}\in\ker(\bm{A})$, one has
$$
		\frac{\|\bm{x}_0\|_p^p}{\|\bm{x}_0\|_q^p} < \frac{\|\bm{x}_0+\bm{h}\|_p^p}{\|\bm{x}_0+\bm{h}\|_q^p}.
$$
	Indeed, every feasible point distinct from $\bm{x}_0$ can be written as $\bm{x}_0+\bm{h}$ for some nonzero $\bm{h}\in\ker(\bm{A})$. Since $0<p\le 1$, the map $t\mapsto |t|^p$ is subadditive. It follows that $|a+b|^p\ge |a|^p-|b|^p$. Splitting the perturbation over $\mathcal{S}$ and $\mathcal{S}^c$ gives
	\begin{equation}\label{eq:num_bound}
		\|\bm{x}_0 + \bm{h}\|_p^p = \|(\bm{x}_0)_{\mathcal{S}} + \bm{h}_{\mathcal{S}}\|_p^p + \|\bm{h}_{\mathcal{S}^c}\|_p^p \ge \|\bm{x}_0\|_p^p - \|\bm{h}_{\mathcal{S}}\|_p^p + \|\bm{h}_{\mathcal{S}^c}\|_p^p. 
	\end{equation}
	Moreover, since $q>1$, the $\ell_q$ norm satisfies the triangle inequality. Using again the subadditivity of $t\mapsto t^p$ on $[0,\infty)$, one has
	\begin{equation}\label{eq:den_bound}
		\|\bm{x}_0+\bm{h}\|_q^p \le \left(\|\bm{x}_0\|_q+\|\bm{h}\|_q\right)^p \le \|\bm{x}_0\|_q^p+\|\bm{h}\|_q^p.
	\end{equation}
	For brevity, define $A_0:=\|\bm{x}_0\|_p^p$, $B_0:=\|\bm{x}_0\|_q^p$,	$E:=\|\bm{h}\|_q^p$, and $T:=\|\bm{h}_{\mathcal{S}^c}\|_p^p - \|\bm{h}_{\mathcal{S}}\|_p^p = \|\bm{h}\|_p^p - 2\|\bm{h}_{\mathcal{S}}\|_p^p$.	From \eqref{eq:num_bound} and \eqref{eq:den_bound}, we have
	\begin{equation}
		\frac{\|\bm{x}_0+\bm{h}\|_p^p}
		{\|\bm{x}_0+\bm{h}\|_q^p}
		\ge
		\frac{A_0+T}{B_0+E}.
		\label{eq:ratio_lower_bound}
	\end{equation}
	
	We now estimate $T/E$. By H\"older's inequality, for any $s$-sparse vector $\bm{v}$ and $p < q$, we have $\|\bm{v}\|_p^p \le s^{1-p/q}\|\bm{v}\|_q^p$. Applying this inequality to $\bm{x}_0$ yields
	\begin{equation}\label{eq:x0_bound}
		\frac{A_0}{B_0} = \frac{\|\bm{x}_0\|_p^p}{\|\bm{x}_0\|_q^p} \le s^{1-p/q}.
	\end{equation}
	It follows from applying inequality \eqref{eq:x0_bound} to the $s$-sparse perturbation vector $\bm{h}_{\mathcal{S}}$ that $\|\bm{h}_{\mathcal{S}}\|_p^p \le s^{1-p/q}\|\bm{h}_{\mathcal{S}}\|_q^p \le s^{1-p/q}\|\bm{h}\|_q^p$.
	On the other hand, the assumption condition \eqref{eq:null_space_cond} implies $\|\bm{h}\|_p^p > 3s^{1-p/q}\|\bm{h}\|_q^p$. Therefore, one has
	\begin{equation}\label{eq:T_bound}
		T = \|\bm{h}\|_p^p - 2\|\bm{h}_{\mathcal{S}}\|_p^p > 3s^{1-p/q}\|\bm{h}\|_q^p - 2s^{1-p/q}\|\bm{h}\|_q^p = s^{1-p/q}\|\bm{h}\|_q^p.
	\end{equation}
	Combining \eqref{eq:T_bound} with \eqref{eq:x0_bound}, we obtain $\frac{T}{E} > s^{1-p/q} \ge \frac{A_0}{B_0}$. Since $B_0>0$ and $E>0$, it follows from the standard mediant property of fractions that $\frac{A_0+T}{B_0+E}>\frac{A_0}{B_0}$. Together with \eqref{eq:ratio_lower_bound}, we conclude that
$$
		\frac{\|\bm{x}_0+\bm{h}\|_p^p}{\|\bm{x}_0+\bm{h}\|_q^p} > \frac{\|\bm{x}_0\|_p^p} {\|\bm{x}_0\|_q^p}.
$$
	Thus $\bm{x}_0$ is the unique global minimizer of $\mathcal{Q}_{p,q}$. 
	Since $\bm{x}_0$ was arbitrary, the result holds uniformly for all nonzero $s$-sparse signals.
\end{proof}

\begin{remark}[]
	Theorem~\ref{thm:exact_recovery} recovers, as a special case, the uniform recoverability condition for the $\ell_r/\ell_2$ minimization studied by Xu et al. \cite{xu2021analysis}. Indeed, for $0<r\le 1$, minimizing $\frac{\|\bm{x}\|_r}{\|\bm{x}\|_2}$ is equivalent to minimizing $(\frac{\|\bm{x}\|_r}{\|\bm{x}\|_2})^r = \frac{\|\bm{x}\|_r^r}{\|\bm{x}\|_2^r}$. Taking $p=r$ and $q=2$ in \eqref{eq:null_space_cond} gives
$
		\inf_{\bm{h}\in\ker(\bm{A})\setminus\{\bm{0}\}} \frac{\|\bm{h}\|_r}{\|\bm{h}\|_2} > 3^{1/r}s^{1/r-1/2},
$
	which coincides with the sufficient condition in Xu et al. \cite[Theorem 4.1]{xu2021analysis}. Thus, Theorem \ref{thm:exact_recovery} extends that uniform exact recovery condition from the $\ell_r/\ell_2$ ratio to the general $\ell_p^p/\ell_q^p$ ratio. This reduction demonstrates consistency with the existing result; it should not be interpreted as a sharpness claim for the constant in \eqref{eq:null_space_cond}.
\end{remark}

\subsection{Sample Complexity Bounds for the Uniform Null Space Condition}

The uniform exact recovery criterion in Theorem~\ref{thm:exact_recovery} is expressed as a geometric condition \eqref{eq:null_space_cond} on the null space of the measurement matrix. To convert this condition into a probabilistic sampling guarantee, we next estimate the number of measurements required for random sub-Gaussian matrices. Since the Gaussian-width analysis of nonconvex $\ell_p$ balls with $0<p<1$ is substantially more delicate, we restrict this subsection to the convex numerator case $p=1$. In this case, the null space condition in Theorem~\ref{thm:exact_recovery} becomes
\begin{equation}\label{eq:target_ratio_pure}
	\inf_{\bm{h}\in\ker(\bm{A})\setminus\{\bm{0}\}}
	\frac{\|\bm{h}\|_1}{\|\bm{h}\|_q}
	> 3s^{1-1/q}.
\end{equation}

\begin{theorem}[]\label{thm:sample_complexity_q_ge_2}
	Let $2\le q<\infty$, $s\in\mathbb{N}$, $u\ge 1$, and $n\ge 2$. Let $\bm{A}\in\mathbb{R}^{m\times n}$ be a random matrix whose rows are independent, isotropic, sub-Gaussian random vectors with sub-Gaussian norm bounded by $F$. Then there exists an absolute constant $D>0$ such that, if
	\begin{equation}\label{eq:sample_q_ge_2}
		m>D F^4 u^2 s^{2-2/q}\log(2n),
	\end{equation}
	then \eqref{eq:target_ratio_pure} holds with probability at least $1-2e^{-u^2}$. 
\end{theorem}

\begin{proof}
	Let $B_1^n:=\{\bm{x}\in\mathbb{R}^n:\|\bm{x}\|_1\le 1\}$. We apply the matrix deviation inequality to the deterministic set $B_1^n$. With probability at least $1-2e^{-u^2}$, one has
$$ 
		\sup_{\bm{x}\in B_1^n} \left| \|\bm{A}\bm{x}\|_2-\sqrt{m}\|\bm{x}\|_2 \right| \le C F^2\left(\omega(B_1^n)+u\,\operatorname{rad}(B_1^n)\right),
$$
	where $C>0$ is an absolute constant, $\omega(\cdot)$ denotes Gaussian width, and $\operatorname{rad}(\cdot)$ denotes Euclidean radius. Since
$$
		\omega(B_1^n)=\mathbb{E}\|\bm{g}\|_\infty\le \sqrt{2\log(2n)}, \quad \operatorname{rad}(B_1^n)\le 1,
$$
	with $\bm{g}\sim\mathcal{N}(\bm{0},\bm{I}_n)$, it follows from $u\ge 1$ and $n\ge 2$ that
	\begin{equation}\label{eq:mdi_b1_simplified}
		\sup_{\bm{x}\in B_1^n}
		\left|
		\|\bm{A}\bm{x}\|_2-\sqrt{m}\|\bm{x}\|_2
		\right|
		\le C_1 F^2 u\sqrt{\log(2n)}
	\end{equation}
	for another absolute constant $C_1>0$. Now take any nonzero $\bm{h}\in\ker(\bm{A})$ and set $\bm{z}:=\frac{\bm{h}}{\|\bm{h}\|_1}$. Then $\bm{z}\in B_1^n$ and $\bm{A}\bm{z}=\bm{0}$. Substituting $\bm{z}$ into \eqref{eq:mdi_b1_simplified} gives $\sqrt{m}\|\bm{z}\|_2 \le C_1F^2u\sqrt{\log(2n)}$. Equivalently, we have
$$
		\frac{\|\bm{h}\|_1}{\|\bm{h}\|_2} = \frac{1}{\|\bm{z}\|_2} \ge \frac{\sqrt{m}}{C_1F^2u\sqrt{\log(2n)}}.
$$
	For $q\ge 2$, the norm monotonicity relation $\|\bm{h}\|_q\le \|\bm{h}\|_2$ yields
$$
		\frac{\|\bm{h}\|_1}{\|\bm{h}\|_q} \ge \frac{\|\bm{h}\|_1}{\|\bm{h}\|_2} \ge \frac{\sqrt{m}}{C_1F^2u\sqrt{\log(2n)}}.
$$
	Therefore \eqref{eq:target_ratio_pure} holds whenever
$$
		\frac{\sqrt{m}}{C_1F^2u\sqrt{\log(2n)}} > 3s^{1-1/q},
$$
	which is guaranteed by \eqref{eq:sample_q_ge_2} after choosing $D=9C_1^2$. The exact recovery assertion then follows directly from Theorem~\ref{thm:exact_recovery}.
\end{proof}

\begin{remark}
	When $q=2$, Theorem~\ref{thm:sample_complexity_q_ge_2} gives the sufficient sampling rate $m \gtrsim F^4u^2s\log(2n)$, which has the same sparsity dependence as the standard sub-Gaussian sampling bound for the $\ell_1/\ell_2$ model. As $q$ increases, the exponent $2-2/q$ increases and approaches $2$ in the limiting $\ell_1/\ell_\infty$ regime. These estimates are sufficient conditions derived from the lower bound on the null-space $\ell_1/\ell_2$ ratio, and should not be interpreted as sharp phase transition thresholds.
\end{remark}

For the range $1<q<2$, the preceding argument is no longer effective, since $\|\bm{h}\|_q\le\|\bm{h}\|_2$ fails in general. A sharper estimate can be obtained by combining the RIP-induced rNSP with Stechkin's inequality.
\begin{theorem}[Sample complexity for $1<q<2$]
	\label{thm:sample_complexity_q_lt_2}
	Let $1<q<2$, $s\in\mathbb{N}$, $u\ge 1$, and $n\ge 2$. Let
	$\bm{A}\in\mathbb{R}^{m\times n}$ be a random matrix whose rows are independent,
	isotropic, sub-Gaussian random vectors, and suppose that their sub-Gaussian norms
	are uniformly bounded by $F$. Then there exist a constant $C_q>0$, depending only
	on $q$, and absolute constants $D_0,D_1>0$ such that, if $C_qs\le n$ and
	\begin{equation}\label{eq:sample_q_lt_2}
		m\ge
		D_0F^4 C_qs\log\left(\frac{en}{s}\right)
		+
		D_1F^4u^2,
	\end{equation}
	then \eqref{eq:target_ratio_pure} holds with probability at least $1-2e^{-u^2}$.
\end{theorem}

\begin{proof}
	Define $c_{1,q}:=\left(1/q\right)^{1/q}\left(1-1/q\right)^{1-1/q} \leq 1$. Fix $\delta:=\frac{4}{\sqrt{41}}$. By the standard result that the restricted isometry property (RIP) implies the $\ell_2$-robust NSP, see Theorem~6.13 in \cite{foucart2013mathematical}, there exist constants $\rho\in(0,1)$ and $\tau>0$, depending only on the fixed value of $\delta$, such that any matrix with $\delta_{2k}<\delta$ satisfies the $\ell_2$-robust NSP of order $k$ with constants $\rho$ and $\tau$. Set $\alpha_q:=\left[3(\rho+c_{1,q})\right]^{q/(q-1)}$ and choose $k:=\lceil \alpha_q s\rceil+1$. Then, one has $k > \alpha_q s$ and $2k\le 2(\alpha_qs+2)\le 2(\alpha_q+2)s$. Let $C_q:=2(\alpha_q+2)$. Hence $2k\le C_qs$. By the assumption $C_qs\le n$, we have $2k\le n$. Let $\bm{\Phi} := \frac{1}{\sqrt{m}} \bm{A}$ be the normalized measurement matrix and define the event $\mathcal E:=\{\delta_{2k}(\bm{\Phi})<\delta\}$. By the standard RIP estimate for isotropic sub-Gaussian random matrices, see Theorem~9.2 in \cite{foucart2013mathematical}, there exist absolute constants $C,c>0$ such that $\mathbb P(\mathcal E)\ge 1-2e^{-u^2}$ provided that
	\begin{equation}\label{eq:rip_sample_bound_q_lt_2}
		m \ge CF^4\delta^{-2}(2k)\log(\frac{en}{2k})	+ cF^4\delta^{-2}u^2.
	\end{equation}
	Since the function $t \mapsto t\log\left(\frac{en}{t}\right)$ is nondecreasing on $(0,n]$, it follows from $2k\le C_qs\le n$ that
	\begin{equation}\label{eq:monotonicity_q_lt_2}
		(2k)\log(\frac{en}{2k})
		\le
		C_qs\log(\frac{en}{C_qs})
		\le
		C_qs\log(\frac{en}{s}).
	\end{equation}
	Thus, by choosing $D_0:=C\delta^{-2}$ and $D_1:=c\delta^{-2}$, condition \eqref{eq:sample_q_lt_2}, together with \eqref{eq:monotonicity_q_lt_2}, implies \eqref{eq:rip_sample_bound_q_lt_2}. It follows that $ \mathbb P(\mathcal E)\ge 1-2e^{-u^2}$.
	
	We now work on the event $\mathcal E$. Then $\delta_{2k}(\bm{\Phi})<\delta$, and therefore $\bm{\Phi}$ satisfies
	the $\ell_2$-robust NSP of order $k$. Namely, for every index set $\mathcal{S}\subset[n]$ with $\operatorname{card}(\mathcal{S})\le k$ and every $\bm v\in\mathbb R^n$, one has
	\begin{equation}\label{eq:l2_rnsp_used}
		\|\bm v_{\mathcal{S}}\|_2 \le \frac{\rho}{\sqrt{k}}\|\bm v_{\mathcal{S}^c}\|_1 + \tau\|\bm{\Phi}\bm v\|_2 .
	\end{equation}
	Let $\bm{h}\in\ker(\bm{A})\setminus\{\bm{0}\}$. Since $\ker(\bm{A})=\ker(\bm{\Phi})$, we have $\bm{\Phi}\bm{h}=\bm{0}$. Let $\mathcal{S}_0$ be the index set of the $k$ largest entries of $\bm{h}$ in absolute value. Applying \eqref{eq:l2_rnsp_used} to $\bm{v}=\bm{h}$ and $\mathcal{S}=\mathcal{S}_0$, we get $\|\bm{h}_{\mathcal{S}_0}\|_2 \le \frac{\rho}{\sqrt{k}}\|\bm{h}_{\mathcal{S}_0^c}\|_1$. Since $1<q<2$, H\"older's inequality gives
	\begin{equation}\label{eq:head_q_lt_2_correct}
		\|\bm{h}_{\mathcal{S}_0}\|_q \le k^{1/q-1/2}\|\bm{h}_{\mathcal{S}_0}\|_2 \le \rho k^{1/q-1}\|\bm{h}_{\mathcal{S}_0^c}\|_1 \le \rho k^{1/q-1}\|\bm{h}\|_1 .
	\end{equation}
	On the other hand, by Stechkin's inequality, see Theorem~2.5 in \cite{foucart2013mathematical}, we have
	\begin{equation}\label{eq:tail_q_lt_2_correct}
		\|\bm{h}_{\mathcal{S}_0^c}\|_q = \sigma_k(\bm{h})_q \le c_{1,q} k^{1/q-1}\|\bm{h}\|_1 .
	\end{equation}
	Combining \eqref{eq:head_q_lt_2_correct} and \eqref{eq:tail_q_lt_2_correct}, we obtain
$$
		\|\bm{h}\|_q \le \|\bm{h}_{\mathcal{S}_0}\|_q+\|\bm{h}_{\mathcal{S}_0^c}\|_q \le (\rho+c_{1,q})k^{1/q-1}\|\bm{h}\|_1 .
$$
	It follows that
	\begin{equation}\label{eq:ratio_q_lt_2_correct}
		\frac{\|\bm{h}\|_1}{\|\bm{h}\|_q} \ge \frac{1}{\rho+c_{1,q}}k^{1-1/q}.
	\end{equation}
	Since $k>\alpha_qs$ and $1-1/q>0$, taking the power yields $k^{1-1/q} > \alpha_q^{1-1/q}s^{1-1/q} = 3(\rho+c_{1,q})s^{1-1/q}$. Substituting this into \eqref{eq:ratio_q_lt_2_correct} yields $\frac{\|\bm{h}\|_1}{\|\bm{h}\|_q} > 3s^{1-1/q}$. Because $\bm{h}\in\ker(\bm{A})\setminus\{\bm{0}\}$ was arbitrary, it follows that
$$
		\inf_{\bm{h}\in\ker(\bm{A})\setminus\{\bm{0}\}} \frac{\|\bm{h}\|_1}{\|\bm{h}\|_q} > 3s^{1-1/q}.
$$
	This proves \eqref{eq:target_ratio_pure}. The exact recovery conclusion then follows from Theorem~\ref{thm:exact_recovery}.
\end{proof}

\begin{remark}
	Theorem~\ref{thm:sample_complexity_q_lt_2} gives a sufficient sampling rate of order $m\gtrsim F^4 C_q s\log\left(\frac{en}{s}\right)+F^4u^2$
	for each fixed $1<q<2$. The constant $C_q$ depends on $q$ through the auxiliary sparsity level used in the robust-NSP argument. Together with Theorem~\ref{thm:sample_complexity_q_ge_2}, this provides probabilistic sufficient conditions for the null space criterion \eqref{eq:target_ratio_pure} in the full range $q>1$, but the stated bounds should be understood as sufficient recovery guarantees rather than sharp optimal thresholds.
\end{remark}

\section{Stable recovery guarantee with $k$-sparse Assumption}\label{sec:ReDCENprop}
This section presents the main result under the $k$-sparsity assumption, together with a corollary for the special case $p = 1$, $q > 1$. The improvements over the prior work \cite[Theorem~1]{zhu2026stable} are also discussed.
\subsection{RIP Analysis for $\mathcal{Q}_{p,q}^{\delta}$-Minimization}

We first list two lemmas and a corollary that will assist the proof of Theorem \ref{thm:RIPstablerecovery}. For convenience, the compressed sensing measurement model described in Section \ref{sec:introduction} is denoted by $\mathcal{M}\{\bm{A}, \bm{x}, \bm{b}, m, n, \epsilon\}.$
\begin{lemma}[{\cite[Lemma 1]{zhu2026stable}}]\label{lem:lemma1}
	Assume that $\bm{x}\neq\bm{0}$ belongs to $\{\bm{z} \mid \|\bm{b} - \bm{A}\bm{z}\|_2 \leq \epsilon\}$, is $k$-sparse, and satisfies $\frac{\|\bm{x}\|_p}{\|\bm{x}\|_q} \leq \beta$ with $0 < p \leq 1$, $q > 1$ and $\beta > 0$. Let $\hat{\bm{x}}$ be a solution to $\mathcal{Q}_{p,q}^{\epsilon}$, and set $\bm{h} := \hat{\bm{x}} - \bm{x}$. Then, we have
	\begin{equation}\label{eq:lem1ieq}
		\|\bm{h}_{-\max(k)}\|_p^p \leq (\beta^p + k^{\frac{q-p}{q}})\|\bm{h}_{\max(k)}\|_q^p + \beta^p\|\bm{h}_{-\max(k)}\|_q^p. 
	\end{equation}
\end{lemma}

\begin{lemma}[{\cite[Lemma 2]{zhu2026stable}}]\label{lem:lemma2}
	Let $f_{p,q}(z) = z^q - \beta^p k^{-\frac{q-p}{q}} z^p - \beta^p k^{-\frac{q-p}{q}} - 1$, where $0 < p \leq 1$, $q > 1$, $\beta > 0$ and $k \geq 1$ be an integer. Then, 
	\begin{enumerate}[{\rm(i)}]
 	\item  $f_{p,q}(z)$ is monotonically decreasing on $(0, z^*)$ and increasing on $(z^*, +\infty)$, where $z^* = k^{-1/q}\left(\beta^p p/q\right)^{1/(q-p)}$ is the stationary-point;
	\item $f_{p,q}(z)$ has a unique zero-point $z_0$ in $[0, +\infty)$ satisfying $\tilde{z} < z_0 < \bar{z}$, where $\tilde{z} := (kq)^{-\frac{1}{q}}(p\beta^p)^{\frac{1}{q-p}}$ and $\bar{z} := (1 + \beta^p k^{-\frac{q-p}{q}})^{\frac{2}{q-p}}$.
	\end{enumerate}
\end{lemma}

We now turn to the stable recovery analysis under the additional assumption that the ground-truth signal is $k$-sparse. The preceding exact recovery result is formulated through a null space condition, whereas the following theorem provides a verifiable sufficient condition in terms of the restricted isometry constant of order $2k$. The proof refines the tail estimate for the recovery error by exploiting the zero point of the auxiliary function in Lemma~\ref{lem:lemma2}, and then combines this estimate with a sparse convex decomposition argument. This leads to a $2k$-RIP condition and an explicit stability bound for the solution of $\mathcal{Q}_{p,q}^{\epsilon}$.
\begin{theorem}\label{thm:RIPstablerecovery}
	Assume that $\bm{x} \neq \bm{0}$ is $k$-sparse and satisfies $\frac{\|\bm{x}\|_p}{\|\bm{x}\|_q}\leq \beta$ with $0<p\leq 1$, $q>1$, $\beta>0$. Let $\hat{\bm{x}}$ be a solution to $\mathcal{Q}_{p,q}^{\epsilon}$. The function $f_{p,q}(z)$ is defined as in Lemma \ref{lem:lemma2}. If $\bm{A}$ satisfies $2k$-RIP with RIC
	\begin{equation*}\label{eq:RIC}
		\delta_{2k}< \frac{1}{\sqrt{1 + k^{\frac{2}{\min(q,2)}-\frac{2q+2-2p}{q}}\Psi^2}},
	\end{equation*}
	where $\Psi := \beta^p(1+z_0^p) + k^{\frac{q-p}{q}}$ and $z_0$ is the unique zero-point of $f_{p,q}(z)$ in $[0,+\infty)$, then
	\begin{equation}\label{eq:RIPupper_general}
		\|\hat{\bm{x}}-\bm{x}\|_2 \leq
		\frac{2\varepsilon\sqrt{1+\delta_{2k}}(1+k^{\frac{1}{\min(q,2)}-\frac{2-p+q}{2q}}\sqrt{\Psi})}
		{1-\delta_{2k}\sqrt{1+k^{\frac{2}{\min(q,2)}-\frac{2q+2-2p}{q}}\Psi^2}}.
	\end{equation}
\end{theorem}

\begin{proof}
	Define $\bm{h} = \hat{\bm{x}} - \bm{x}$. The case $\bm{h} = \bm{0}$ is trivial, so the subsequent analysis assumes $\bm{h} \neq \bm{0}$ and aims to bound $\|\bm{h}\|_2$. Decomposing $\bm{h} = \bm{h}_{\max(k)} + \bm{h}_{-\max(k)}$, upper bounds for $\|\bm{h}_{\max(k)}\|_2$ and $\|\bm{h}_{-\max(k)}\|_2$ are derived separately.
	
	First, we estimate $\|\bm{h}_{-\max(k)}\|_q$. It follows from the definition of the $\ell_\infty$-norm and the inequality $|v_i|^q = |v_i|^{q-p} \cdot |v_i|^p \le \|\bm{v}\|_{\infty}^{q-p} |v_i|^p$ that
	\begin{equation}\label{eq:inftynormieq}
		\|\bm{h}_{-\max(k)}\|_q^q \leq \|\bm{h}_{-\max(k)}\|_\infty^{q-p} \|\bm{h}_{-\max(k)}\|_p^p.
	\end{equation}

	For the vector $\bm{h}_{-\max(k)}$, the ordering of entries implies
	\begin{equation}\label{eq:orderingieq}
		 \begin{aligned}
			\|\bm{h}_{-\max(k)}\|_q^q \le k^{-\frac{q-p}{q}} \|\bm{h}_{\max(k)}\|_q^{q-p} \|\bm{h}_{-\max(k)}\|_p^p.
		\end{aligned}
	\end{equation}
	Lemma~\ref{lem:lemma1} applied to the term $\|\bm{h}_{-\max(k)}\|_p^p$, together with \eqref{eq:orderingieq}, gives
	\begin{equation}\label{eq:ieqtail}
		\begin{aligned}
			\|\bm{h}_{-\max(k)}\|_q^q 
			&\le k^{-\frac{q-p}{q}} \|\bm{h}_{\max(k)}\|_q^{q-p} 
			\Bigl[ \bigl(\beta^p + k^{\frac{q-p}{q}}\bigr) \|\bm{h}_{\max(k)}\|_q^p + \beta^p \|\bm{h}_{-\max(k)}\|_q^p \Bigr] \\
			&= \bigl(1 + \beta^p k^{-\frac{q-p}{q}}\bigr) \|\bm{h}_{\max(k)}\|_q^q 
			+ \beta^p k^{-\frac{q-p}{q}} \|\bm{h}_{-\max(k)}\|_q^p \|\bm{h}_{\max(k)}\|_q^{q-p}.
		\end{aligned}
	\end{equation}
	Let $z:=\frac{\|\bm{h}_{-\max(k)}\|_q}{\|\bm{h}_{\max(k)}\|_q}$. From the definition of $f_{p,q}(z)$ in Lemma~\ref{lem:lemma2} and rearranging \eqref{eq:ieqtail}, it follows that
	\begin{equation*}
		\frac{\|\bm{h}_{-\max(k)}\|_q^q}{\|\bm{h}_{\max(k)}\|_q^q}
		- \beta^p k^{-\frac{q-p}{q}} \frac{\|\bm{h}_{-\max(k)}\|_q^p}{\|\bm{h}_{\max(k)}\|_q^p}
		- \beta^p k^{-\frac{q-p}{q}} - 1 = f_{p,q}(z) \le 0.
	\end{equation*}
	By Lemma~\ref{lem:lemma2}, it follows that $z \le z_0$, i.e.,
	\begin{equation}\label{eq:rootieq}
		\|\bm{h}_{-\max(k)}\|_q \le z_0 \|\bm{h}_{\max(k)}\|_q.
	\end{equation}
	Substituting \eqref{eq:rootieq} in \eqref{eq:lem1ieq}, we directly obtain the upper bound
	\begin{equation}\label{eq:tightbound}
		\begin{aligned}
			\|\bm{h}_{-\max(k)}\|_p^p &\leq (\beta^p+k^{\frac{q-p}{q}})\|\bm{h}_{\max(k)}\|_q^p + \beta^p z_0^p \|\bm{h}_{\max(k)}\|_q^p \\
			&= [ \beta^p(1+z_0^p) + k^{\frac{q-p}{q}}] \|\bm{h}_{\max(k)}\|_q^p = \Psi \|\bm{h}_{\max(k)}\|_q^p.
		\end{aligned}
	\end{equation}
	By \cite[Lemma~2.2]{zhang2019optimal}, there exist $N\in\mathbb{N}$, weights $\{\lambda_i\}_{i=1}^N$ with $\lambda_i>0$, $\sum_{i=1}^N \lambda_i=1$, and $k$-sparse vectors $\{\bm{u}_i\}_{i=1}^N$ with $\operatorname{supp}(\bm{u}_i)\subseteq \operatorname{supp}(\bm{h}_{-\max(k)})$ such that $\bm{h}_{-\max(k)}=\sum_{i=1}^N \lambda_i \bm{u}_i$, and
	\begin{equation*}
		\sum_{i=1}^N \lambda_i \|\bm{u}_i\|_2^2 
		\le k^{-1} \|\bm{h}_{-\max(k)}\|_p^p \|\bm{h}_{-\max(k)}\|_{2-p}^{2-p}.
	\end{equation*}
	Combining \eqref{eq:inftynormieq} with \eqref{eq:rootieq} yields
$$
		k^{-1} \|\bm{h}_{-\max(k)}\|_p^p \|\bm{h}_{-\max(k)}\|_{2-p}^{2-p} \le k^{-1} \|\bm{h}_{-\max(k)}\|_p^{2p} \|\bm{h}_{-\max(k)}\|_\infty^{2-2p}.
$$
	It follows from
	\begin{equation}\label{eq:holderieq}
		\|\bm{h}_{-\max(k)}\|_{\infty} \le k^{-1} \|\bm{h}_{\max(k)}\|_1 
		\le k^{-\frac{1}{q}} \|\bm{h}_{\max(k)}\|_q
	\end{equation}
	that
	\begin{equation}\label{eq:uppera}
		k^{-1} \|\bm{h}_{-\max(k)}\|_p^{2p} \|\bm{h}_{-\max(k)}\|_\infty^{2-2p} \le k^{-1-\frac{2-2p}{q}} \|\bm{h}_{\max(k)}\|_q^{2-2p} \|\bm{h}_{-\max(k)}\|_p^{2p}.
	\end{equation}
	Substituting \eqref{eq:tightbound} into \eqref{eq:uppera} yields
$$
		k^{-1-\frac{2-2p}{q}} \|\bm{h}_{\max(k)}\|_q^{2-2p} \|\bm{h}_{-\max(k)}\|_p^{2p} 
		\leq k^{-1-\frac{2-2p}{q}} \Psi^2 \|\bm{h}_{\max(k)}\|_q^2.
$$
	It follows from H\"older's inequality with $1<q<2$ and the property of $\ell_q$-norm with $q\geq2$ that
	\begin{equation}\label{eq:finupper}
		\sum_{i=1}^N \lambda_i \|\bm{u}_i\|_2^2 
		\le k^{\frac{2}{\min(q,2)}-\frac{2q+2-2p}{q}} \Psi^2 \|\bm{h}_{\max(k)}\|_2^2.
	\end{equation}
	Define $\mu := \frac{1}{1+\sqrt{1+k^{\frac{2}{\min(q,2)}-\frac{2q+2-2p}{q}}\Psi^2}}$, $\bm{w} := (1-\mu)\bm{h}_{\max(k)} + \mu\bm{h}$, and $\bm{w}_i := \bm{h}_{\max(k)} + \mu\bm{u}_i$ for $i=1,2,\dots,N$, It follows immediately that
	\begin{equation}\label{eq:huper}
		\sum_{i=1}^N \lambda_i \| \bm{A}(\bm{w}-\frac{1}{2}\bm{w}_i)\|_2^2 = \sum_{i=1}^N \frac{\lambda_i}{4} \|\bm{A}\bm{w}_i\|_2^2.
	\end{equation}
	Since each $\bm{w}_i$ is $2k$-sparse, the $2k$-RIP gives $\|\bm{A}\bm{w}_i\|_2^2 \ge (1 - \delta_{2k})\|\bm{w}_i\|_2^2$. Expanding $\|\bm{w}_i\|_2^2 = \|\bm{h}_{\max(k)}\|_2^2 + \mu^2\|\bm{u}_i\|_2^2$ and applying the weights $\lambda_i$, we obtain the lower bound
	\begin{equation}
		\sum_{i=1}^N \frac{\lambda_i}{4} \|\bm{A}\bm{w}_i\|_2^2 \ge \frac{1 - \delta_{2k}}{4}( \|\bm{h}_{\max(k)}\|_2^2 + \mu^2 \sum_{i=1}^N \lambda_i \|\bm{u}_i\|_2^2). \label{eq:lower_bound}
	\end{equation}
	For the left-hand side of \eqref{eq:huper}, we rewrite the term inside the norm as
	\begin{equation}
		\bm{w} - \frac{1}{2}\bm{w}_i = (\frac{1}{2} - \mu)\bm{h}_{\max(k)} - \frac{1}{2}\mu \bm{u}_i + \mu \bm{h}:= \bm{v}_i.
	\end{equation}
	The vector $\bm{v}_i$ is strictly $2k$-sparse since $\bm{h}_{\max(k)}$ and $\bm{u}_i$ have disjoint supports of size at most $k$. Expanding the squared $\ell_2$-norm yields
	\begin{equation}\label{eq:expand_LHS}
		\| \bm{A}(\bm{w}-\frac{1}{2}\bm{w}_i) \|_2^2 = \|\bm{A}\bm{v}_i\|_2^2 + \mu^2\|\bm{A}\bm{h}\|_2^2 + 2\mu \langle \bm{A}\bm{v}_i, \bm{A}\bm{h} \rangle. 
	\end{equation}
	Since $\sum_{i=1}^N \lambda_i \bm{u}_i = \bm{h}_{-\max(k)} = \bm{h} - \bm{h}_{\max(k)}$, the weighted sum of $\bm{v}_i$ evaluates exactly to
	\begin{equation}\label{eq:weightsum}
		\sum_{i=1}^N \lambda_i \bm{v}_i = (\frac{1}{2} - \mu)\bm{h}_{\max(k)} - \frac{\mu}{2}\bm{h}_{-\max(k)} = \frac{1-\mu}{2}\bm{h}_{\max(k)} - \frac{\mu}{2}\bm{h}.
	\end{equation}
	Substituting \eqref{eq:weightsum} into the sum of the cross-terms $2\mu \langle \bm{A}\bm{v}_i, \bm{A}\bm{h} \rangle$, one has
	\begin{equation}\label{eq:cross-termssum}
		\sum_{i=1}^N \lambda_i 2\mu \langle \bm{A}\bm{v}_i, \bm{A}\bm{h} \rangle = \mu(1-\mu)\langle \bm{A}\bm{h}_{\max(k)}, \bm{A}\bm{h} \rangle - \mu^2\|\bm{A}\bm{h}\|_2^2.
	\end{equation}
	Notice that the term $-\mu^2\|\bm{A}\bm{h}\|_2^2$ perfectly annihilates the $\sum \lambda_i \mu^2\|\bm{A}\bm{h}\|_2^2$ term arising from \eqref{eq:expand_LHS}. Thus, the left-hand side of \eqref{eq:huper} simplifies to
	\begin{equation*}
		\sum_{i=1}^N \lambda_i \| \bm{A}(\bm{w} - \frac{1}{2}\bm{w}_i) \|_2^2 = \sum_{i=1}^N \lambda_i \|\bm{A}\bm{v}_i\|_2^2 + \mu(1-\mu)\langle \bm{A}\bm{h}_{\max(k)}, \bm{A}\bm{h} \rangle.
	\end{equation*}
	Applying the $2k$-RIP to $\|\bm{A}\bm{v}_i\|_2^2$, the Cauchy-Schwarz inequality to the inner product, and utilizing the feasibility constraint $\|\bm{A}\bm{h}\|_2 \le 2\varepsilon$, we obtain the upper bound
	\begin{equation}
		\begin{aligned}
			\sum_{i=1}^N \lambda_i \| \bm{A}(\bm{w} - \frac{1}{2}\bm{w}_i) \|_2^2
			&\le (1+\delta_{2k})[ (\frac{1}{2} - \mu)^2 \|\bm{h}_{\max(k)}\|_2^2 + \frac{\mu^2}{4}\sum_{i=1}^N \lambda_i \|\bm{u}_i\|_2^2 ] \\
			&\quad + 2\varepsilon \mu(1-\mu) \sqrt{1+\delta_{2k}}\, \|\bm{h}_{\max(k)}\|_2. \label{eq:upper_bound}
		\end{aligned}
	\end{equation}
	Combining the lower \eqref{eq:lower_bound} and upper \eqref{eq:upper_bound} estimates, we multiply both sides by $\frac{4}{\mu^2}$ and group the terms for $\|\bm{h}_{\max(k)}\|_2^2$ and $\sum \lambda_i \|\bm{u}_i\|_2^2$. Let $C^2 := k^{\frac{2}{\min(q,2)}-\frac{2q+2-2p}{q}} \Psi^2$ as defined in \eqref{eq:finupper}. After algebraic simplification utilizing the definition $\mu = \frac{1}{1+\sqrt{1+C^2}}$, the coefficients meticulously resolve to
$$
		4\sqrt{1+C^2} \left[ 1 - \delta_{2k}\sqrt{1+C^2} \right] \|\bm{h}_{\max(k)}\|_2^2 \le 8\varepsilon \sqrt{1+C^2} \sqrt{1+\delta_{2k}} \|\bm{h}_{\max(k)}\|_2.
$$
	Dividing by $4\sqrt{1+C^2}\|\bm{h}_{\max(k)}\|_2 > 0$ and combining with $\delta_{2k} \sqrt{1 + C^2} < 1$ yields
	\begin{equation}\label{eq:h_max_bound}
		\|\bm{h}_{\max(k)}\|_2 \leq \frac{2\varepsilon\sqrt{1+\delta_{2k}}}{1-\delta_{2k}\sqrt{1+k^{\frac{2}{\min(q,2)}-\frac{2q+2-2p}{q}}\Psi^2}}.
	\end{equation}
	
	According to Lemma \ref{lem:lemma1} and \eqref{eq:holderieq}, one has
	\begin{equation*}
		\begin{aligned}
			\|\bm{h}_{-\max(k)}\|_2^2 &\leq \|\bm{h}_{-\max(k)}\|_\infty^{2-p} \|\bm{h}_{-\max(k)}\|_p^p \\
			&\leq k^{-\frac{2-p}{q}} ( \beta^p+k^{\frac{q-p}{q}} ) ( \|\bm{h}_{\max(k)}\|_q^2 + \frac{\beta^p}{\beta^p+k^{\frac{q-p}{q}}} \|\bm{h}_{-\max(k)}\|_q^p \|\bm{h}_{\max(k)}\|_q^{2-p} ).
		\end{aligned}
	\end{equation*}
	Combining \eqref{eq:rootieq} with H\"older's inequality ( $1<q<2$) and the $\ell_q$-norm property for $q\ge 2$ yields
	\begin{equation}\label{eq:h-maxupper}
		\begin{aligned}
			\|\bm{h}_{-\max(k)}\|_2^2 &\leq k^{-\frac{2-p}{q}} ( \beta^p(1+z_0^p) + k^{\frac{q-p}{q}} ) \|\bm{h}_{\max(k)}\|_q^2 = k^{-\frac{2-p}{q}} \Psi \|\bm{h}_{\max(k)}\|_q^2 \\
			&\leq k^{\frac{2}{\min(q,2)}-\frac{2-p+q}{q}} \Psi \|\bm{h}_{\max(k)}\|_2^2.
		\end{aligned}
	\end{equation}
	It follows from the triangle inequality and  \eqref{eq:h-maxupper} that
	\begin{equation}\label{eq:hupper}
		\|\bm{h}\|_2 \leq \|\bm{h}_{\max(k)}\|_2 + \|\bm{h}_{-\max(k)}\|_2 \leq ( 1 + k^{\frac{1}{\min(q,2)}-\frac{2-p+q}{2q}} \sqrt{\Psi} ) \|\bm{h}_{\max(k)}\|_2.
	\end{equation}
	Combining \eqref{eq:hupper} with \eqref{eq:h_max_bound}, we obtain the final upper bound \eqref{eq:RIPupper_general}.
\end{proof}

When $p=1$ in Theorem \ref{thm:RIPstablerecovery}, we have the following result.

\begin{corollary}[{\cite[Corollary 1]{zhu2026stable}}]\label{corollary1}
	Assume that $\bm{x}\neq\bm{0}$ is $k$-sparse and satisfies $\frac{\|\bm{x}\|_1}{\|\bm{x}\|_q} \leq \beta$ with $q > 1$ and $\beta > 0$. Let $\hat{\bm{x}}$ be a solution to $\mathcal{Q}_{1,q}^{\epsilon}$. Define a function $f_{1,q}(z) = z^q - \beta k^{-\frac{q-1}{q}} z - \beta k^{-\frac{q-1}{q}} - 1$. If $\bm{A}$ satisfies $2k$-RIP with RIC
$$
		\delta_{2k} < \frac{1}{\sqrt{1 + k^{\frac{2}{\min\{q,2\}}} (\beta(1+z_0)k^{-1} + k^{-\frac{1}{q}})^2}},
$$
	where $z_0$ is the unique zero-point of $f_{1,q}(z)$ in $[0, +\infty)$, then the distance between $\hat{\bm{x}}$ and $\bm{x}$ is upper-bounded by
	\begin{equation}\label{eq:RIPupper}
		\|\hat{\bm{x}} - \bm{x}\|_2 \leq \frac{2\epsilon \sqrt{1 + \delta_{2k}} (1 + k^{\frac{1}{2} - \frac{1}{q}}) \sqrt{\beta(1+z_0) + k^{\frac{q-1}{q}}}}{1 - \delta_{2k} \sqrt{1 + k^{\frac{2}{\min\{q,2\}}} (\beta(1+z_0)k^{-1} + k^{-\frac{1}{q}})^2}}.
	\end{equation}
\end{corollary}

In order to compare Theorem \ref{thm:RIPstablerecovery} with Theorem 1 in \cite{zhu2026stable}, we recall it as follows.

\begin{theorem}[{\cite[Theorem~1]{zhu2026stable}}]\label{thm:zhu_RIP_recall}
	Let $0<p\le 1$, $q>1$, and $\beta>0$. Assume that $\bm{x}\neq \bm{0}$ is $k$-sparse in $\mathcal{M}\{\bm{A},\bm{x},\bm{b},m,n,\epsilon\}$ and satisfies
$$
		\frac{\|\bm{x}\|_p}{\|\bm{x}\|_q}\le \beta .
$$
	Let $\hat{\bm{x}}$ be a solution of $\mathcal{Q}_{p,q}^{\epsilon}$, and let $z_0$ be the unique zero point of $f_{p,q}$ in Lemma~\ref{lem:lemma2}. If $\bm{A}$ satisfies the $2k$-RIP with RIC
	\begin{equation}\label{eq:zhuRIC}
		\delta_{2k} < \frac{1}{\sqrt{1+ 3^{2-2p} k^{\frac{2}{\min\{q,2\}}-\frac{2-2p}{q}} (\beta(1+z_0)k^{-\frac{1}{p}} + k^{-\frac{1}{q}})^{2p}}},
	\end{equation}
	then
	\begin{equation}\label{eq:zhuerrbound}
		\|\hat{\bm{x}}-\bm{x}\|_2 \le \frac{2\epsilon\sqrt{1+\delta_{2k}} (1+ k^{\frac{1}{\min\{q,2\}}-\frac{2-p+q}{2q}} \sqrt{\beta^p(1+z_0^p)+k^{\frac{q-p}{q}}})}{1- \delta_{2k} \sqrt{1+ 3^{2-2p} k^{\frac{2}{\min\{q,2\}}-\frac{2-2p}{q}} (\beta(1+z_0)k^{-\frac{1}{p}} + k^{-\frac{1}{q}})^{2p}}}.
	\end{equation}
\end{theorem}

Let $T_1 :=	3^{2-2p}k^{\frac{2}{\min\{q,2\}}-\frac{2-2p}{q}}(\beta(1+z_0)k^{-\frac{1}{p}}+k^{-\frac{1}{q}})^{2p}$. Then the corresponding upper bound on $\delta_{2k}$ in \eqref{eq:zhuRIC} can be rewritten as $\Delta_{\rm Zhu}:=\frac{1}{\sqrt{1+T_1}}$. In contrast, Theorem~\ref{thm:RIPstablerecovery} involves $T_2:=k^{\frac{2}{\min\{q,2\}}-\frac{2q+2-2p}{q}}\Psi^2$,  and its upper bound on $\delta_{2k}$ can be rewritten as $\Delta_{\rm new}:=\frac{1}{\sqrt{1+T_2}}$.  Thus, we have the following proposition.

\begin{proposition}\label{prop:deltaimprove}
	For all parameters $0<p\le 1$, $q>1$, $\beta>0$, $k\ge1$, and with $z_0$ defined as in Lemma~\ref{lem:lemma2}, we have $T_1 \geq T_2$. Consequently, $\Delta_{\rm new}\ge \Delta_{\rm Zhu}$. Moreover, equality holds when $p=1$, while the inequality is strict for $0<p<1$.
\end{proposition}

\begin{proof}
	The identity $\frac{2q+2-2p}{q}= \frac{2-2p}{q}+2$ implies $k^{\frac{2}{\min(q,2)}-\frac{2q+2-2p}{q}} \Psi^2= k^{\frac{2}{\min(q,2)}-\frac{2-2p}{q}} k^{-2} \Psi^2$. Hence $T_{1}\ge T_{2}$ is equivalent to
$$
		k^{-2}\Psi^2 \le 3^{2-2p}(\beta(1+z_0)k^{-1/p}+k^{-1/q})^{2p}.
$$
	Taking square roots (all quantities are positive) gives the equivalent inequality
$$
		k^{-1}\Psi \le 3^{1-p}\bigl(\beta(1+z_0)k^{-1/p}+k^{-1/q}\bigr)^{p}.
$$
	According to the definition of $\Psi$, one has
$$
		\Psi = \beta^p(1+z_0^p) + k^{\frac{q-p}{q}}= k [ \bigl(\beta k^{-1/p}\bigr)^p + (\beta z_0 k^{-1/p})^p + \bigl(k^{-1/q}\bigr)^p ].
$$
	Letting $a = \beta k^{-1/p}$, $b = \beta z_0 k^{-1/p}$, and $c = k^{-1/q}$ yields $\Psi = k(a^p+b^p+c^p)$. 
	
	For $0<p\le1$, the function $t\mapsto t^p$ is concave. By the power mean inequality (or the generalized H\"older inequality), we have $a^p+b^p+c^p \le 3^{1-p}(a+b+c)^p$. Hence,
$$
		\Psi \le k\cdot 3^{1-p}\,(a+b+c)^p = k[\,3^{\frac{1-p}{p}}(\beta(1+z_0)k^{-1/p}+k^{-1/q})]^p.
$$
	Dividing both sides by $k$ yields the desired inequality
$$
		k^{-1}\Psi \le 3^{1-p}\bigl(\beta(1+z_0)k^{-1/p}+k^{-1/q}\bigr)^p.
$$
	When $0<p<1$, the inequality $a^p+b^p+c^p \le 3^{1-p}(a+b+c)^p$ is strict unless $a=b=c$; this would require $z_0=1$,
	whereas $f_{p,q}(1)=1-\beta^p k^{-\frac{q-p}{q}}-\beta^p k^{-\frac{q-p}{q}}-1<0$, so $z_0\neq1$. Hence $T_2<T_1$ for $0<p<1$. For $p=1$ the inequality becomes an equality, and consequently $T_{1}=T_{2}$. This implies $\Delta_{\rm new}\ge \Delta_{\rm Zhu}$. 
\end{proof}
Figure \ref{fig:deltaComparison} intuitively illustrates the magnitudes of two different upper bounds under varying sparsity levels.

\begin{remark}
	Proposition~\ref{prop:deltaimprove} shows that the proposed RIP condition
	is never more restrictive than that of \cite[Theorem~1]{zhu2026stable}.
	For $0<p<1$, it permits a strictly larger range for
	$\delta_{2k}$. For $p=1$, the two upper bounds on $\delta_{2k}$ coincide, so the result
	reduces to the corresponding $\ell_1/\ell_q$ case.
\end{remark}

\begin{figure}[htbp]
	\centering
	\includegraphics[scale=0.4]{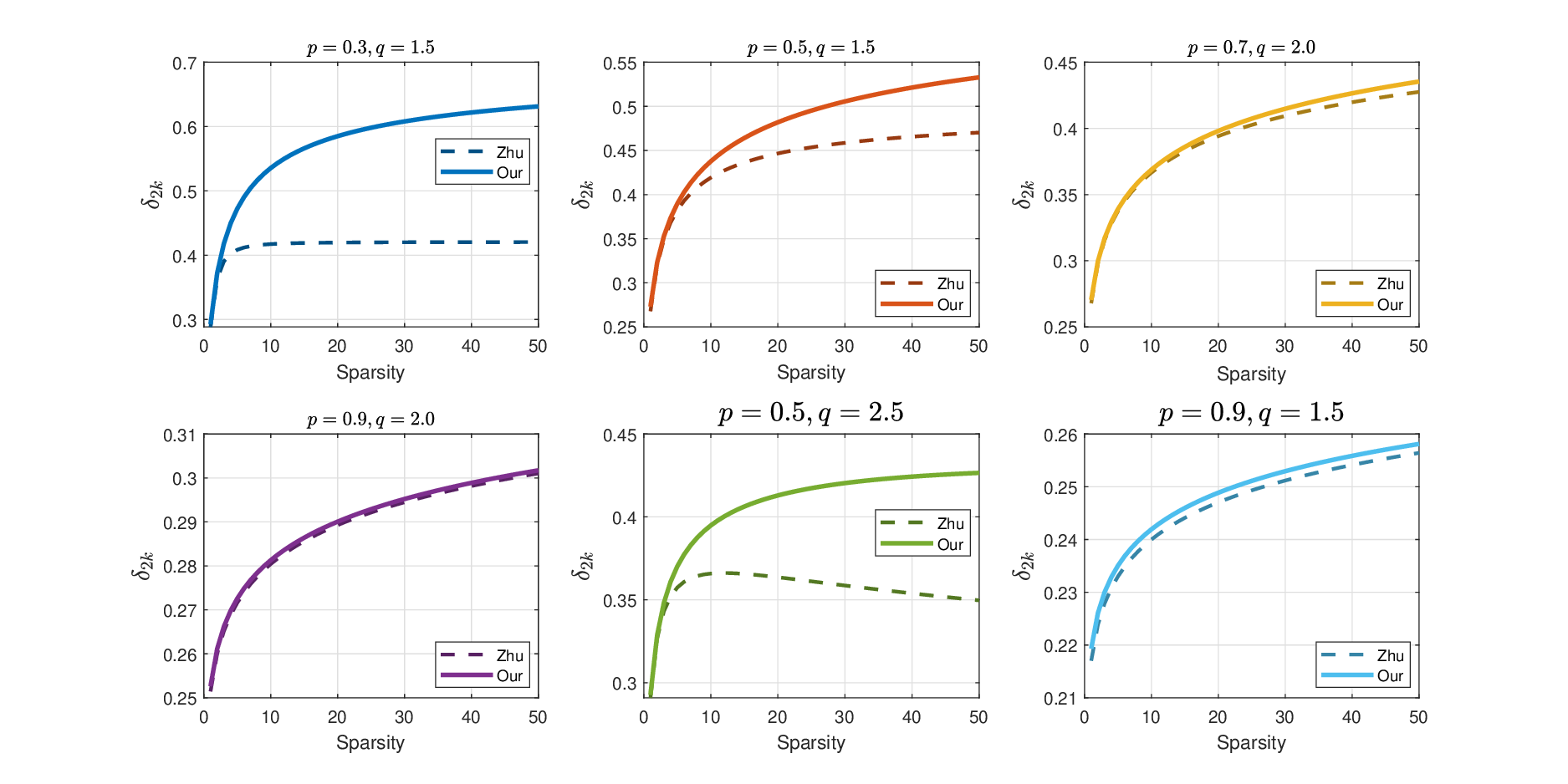}  
	\caption{The behavior of two different upper bounds under different sparsity levels $k$.}
	\label{fig:deltaComparison}
\end{figure}

\begin{proposition}[]\label{prop:tighter_bound}
	Let $B_{o}$ denote the upper bound in Theorem \ref{thm:RIPstablerecovery}, and $B_{z}$ denote the upper bound in Theorem \ref{thm:zhu_RIP_recall}. If the RIC $\delta_{2k}$ of the measurement matrix $\bm{A}$ satisfies the sufficient conditions of both results, then $B_{o} \le B_{z}$.
	
\end{proposition}

\begin{proof}
	If $p=1$, Proposition~\ref{prop:deltaimprove} gives $T_1=T_2$, and the two error bounds coincide under the same range of $\delta_{2k}$. Thus the claim is immediate. In the following, assume that $0<p<1$.

	The numerators of both bounds are identical, reducing the comparison exclusively to their respective denominators. The recovery bound $B_{o}$ is governed by the denominator
$$
		D_{o} = 1 - \delta_{2k}\sqrt{1 + k^{\frac{2}{\min(q,2)} - \frac{2q+2-2p}{q}} \Psi^2},
$$
	In contrast, the recovery bound $B_{z}$ is governed by the denominator
	\begin{equation*}
		D_{z} = 1 - \delta_{2k}\sqrt{1 + 3^{2-2p} k^{\frac{2}{\min(q,2)} - \frac{2-2p}{q}} \left( \beta(1 + z_0)k^{-1/p} + k^{-1/q} \right)^{2p}}.
	\end{equation*}
	Define an auxiliary variable
	\begin{equation}\label{eq:zhu_penalty_extracted}
		\Phi_{z}^2 := [ 3^{\frac{1-p}{p}} \left( \beta(1 + z_0) k^{\frac{p-q}{pq}} + 1 \right) ]^{2p}. 
	\end{equation}
	Factoring out $k^{\frac{q-p}{q}}$ from $\Psi^2$ yields
$$
		\Psi^2 = [ \beta^p(1 + z_0^p) + k^{\frac{q-p}{q}} ]^2 = k^{\frac{2q-2p}{q}} [ \beta^p(1 + z_0^p) k^{\frac{p-q}{q}} + 1 ]^2.
$$
	Substituting the factorized form of $\Psi^2$ and combining the powers of $k$ yields
	\begin{equation}\label{eq:our_term_simplified}
		k^{\frac{2}{\min(q,2)} - \frac{2q+2-2p}{q}} \Psi^2 = k^{\frac{2}{\min(q,2)} - \frac{2}{q}} [ \beta^p(1 + z_0^p) k^{\frac{p-q}{q}} + 1 ]^2.
	\end{equation}
	Similarly, multiplying the factor $k^{\frac{2}{\min(q,2)} - \frac{2-2p}{q}}$ by $\Phi_{z}$ from \eqref{eq:zhu_penalty_extracted} yields
	\begin{equation}\label{eq:zhu_term_simplified}
		\begin{aligned}
			k^{\frac{2}{\min(q,2)} - \frac{2-2p}{q}}\Phi_{z}^2 &= k^{\frac{2}{\min(q,2)} - \frac{2-2p}{q}} \cdot 3^{2-2p} k^{\frac{2p}{q}} ( \beta(1 + z_0)k^{-1/p} + k^{-1/q} )^{2p} \\
			&= k^{\frac{2}{\min(q,2)} - \frac{2}{q}} \cdot 3^{2-2p} [ \beta(1 + z_0) k^{\frac{p-q}{pq}} + 1 ]^{2p}. 
		\end{aligned}
	\end{equation}
	Since the factor $k^{\frac{2}{\min(q,2)} - \frac{2}{q}}$ appears in both \eqref{eq:our_term_simplified} and \eqref{eq:zhu_term_simplified}, it suffices to compare $\Lambda_{o} = [ \beta^p(1 + z_0^p) k^{\frac{p-q}{q}} + 1 ]^2$ and $\Lambda_{z} =[ \beta(1 + z_0) k^{\frac{p-q}{pq}} + 1 ]^{2p}$. Let $x := \beta(1+z_0) k^{\frac{p-q}{pq}} > 0$ be an auxiliary variable. Then $\Lambda_{z} = (1+x)^{2p}$. The inequality $(a+b)^p \ge 2^{p-1}(a^p+b^p)$ for $p \in (0,1]$ implies $1+z_0^p \le 2^{1-p}(1+z_0)^p$. Consequently,
$$
		\beta^p(1+z_0^p) k^{\frac{p-q}{q}} \le 2^{1-p} \beta^p(1+z_0)^p k^{\frac{p-q}{q}} = 2^{1-p} x^p,
$$
	and hence $\Lambda_{o} \le \bigl( 1 + 2^{1-p} x^p \bigr)^2$.
	It remains to show that for all $x > 0$ and $0 < p < 1$,
	\begin{equation*}
		\frac{\left(1 + 2^{1-p}x^p\right)^2}{(1+x)^{2p}} \le 3^{2-2p} \quad\Longleftrightarrow\quad \frac{1 + 2^{1-p}x^p}{(1+x)^p} \le 3^{1-p}.
	\end{equation*}
	Let us define the univariate function $f(x) = \frac{1 + 2^{1-p} x^p}{(1+x)^p}$ for $x > 0$. 
	Taking the derivative with respect to $x$, one has
	\begin{equation}\label{eq:derivativef}
		f'(x) = \frac{p 2^{1-p} x^{p-1} (1+x)^p - p (1 + 2^{1-p} x^p) (1+x)^{p-1}}{(1+x)^{2p}}.
	\end{equation}
	Setting the numerator of \eqref{eq:derivativef} to zero to find the critical points yields
	\begin{equation*}
		2^{1-p} x^{p-1} (1+x) = 1 + 2^{1-p} x^p \implies 2^{1-p} x^{p-1} + 2^{1-p} x^p = 1 + 2^{1-p} x^p \implies x^{1-p} = 2^{1-p}.
	\end{equation*}
	Since $p \in (0, 1)$, the exponent $p-1$ is strictly negative. Therefore, $g(x) = 2^{1-p} x^{p-1} - 1$ is a strictly monotonically decreasing function of $x$ on $(0, +\infty)$. 
	Setting $g(x) = 0$ to find the unique critical point $x = 2$. Because $g(x)$ is strictly decreasing, we have $g(x) > 0 \implies f'(x) > 0$ for all $x \in (0, 2)$, and $g(x) < 0 \implies f'(x) < 0$ for all $x \in (2, +\infty)$. 
	That is, the function $f(x)$ is strictly increasing on $(0, 2)$ and strictly decreasing on $(2, +\infty)$. This implies a unique global maximum exists precisely at $x = 2$. At the global maximum $x = 2$, the function $f(x)$ takes the value
$$
		f(2) = \frac{1 + 2^{1-p} 2^p}{(1+2)^p} = \frac{1 + 2}{3^p} = 3^{1-p}. 
$$
	It follows that $f(x)$ is bounded above by $3^{1-p}$ on $(0,\infty)$, with the maximum attained uniquely at $x = 2$. Squaring this result, we have
	\begin{equation*}
		\frac{\Lambda_{o}}{\Lambda_{z}} \le \left( 3^{1-p} \right)^2 = 3^{2-2p} \implies \Lambda_{o} \le 3^{2-2p} \Lambda_{z}. 
	\end{equation*}
	It follows universally that $B_{o} \leq B_{z}$. This completes the proof.
\end{proof}
Figure~\ref{fig:recoveryboubd_comp} provides a comparison of the upper bounds of two recovery errors across different sparsity levels.
\begin{figure}[htbp]
	\centering
	\includegraphics[scale=0.35]{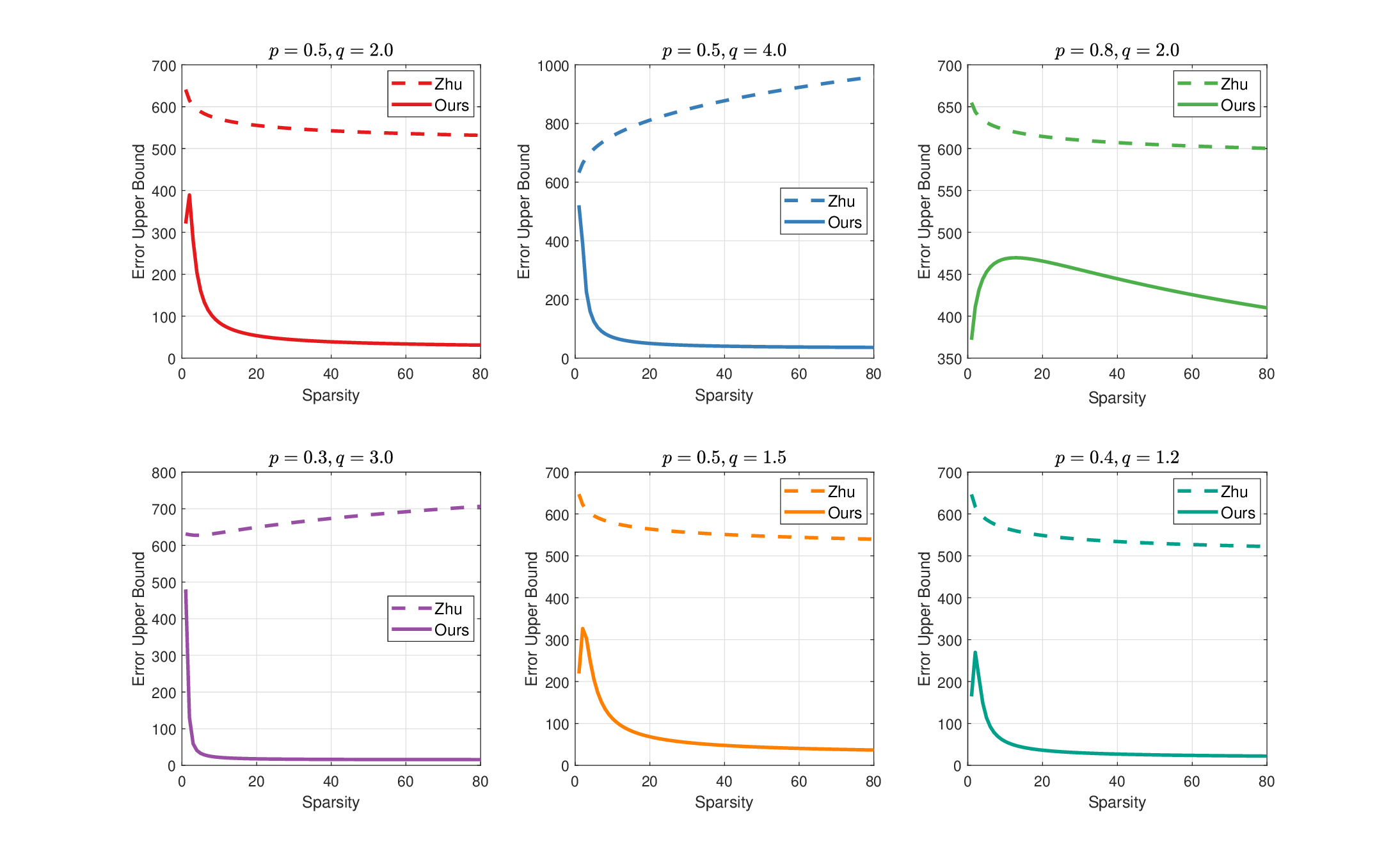}  
	\caption{Upper bounds of two recovery errors under different sparsity levels.}
	\label{fig:recoveryboubd_comp}
\end{figure}

\subsection{Extreme Cases for the RIP Sufficient Condition}
 
In this subsection, we investigate the limiting behavior of the RIP sufficient condition in Theorem~\ref{thm:RIPstablerecovery}. The purpose is twofold. First, we follow the same worst-case sparsity-based normalization used in \cite{zhu2026stable}. Second, we further discuss the bounded norm-ratio regime, which more clearly reveals the improvement of the proposed condition. In the following, we discuss and analyze the limiting cases under the two scenarios.
 
First, the worst-case sparsity-based normalization. For a nonzero $k$-sparse vector $\bm{x}$, the standard norm inequality gives $\frac{\|\bm{x}\|_p}{\|\bm{x}\|_q} \le k^{\frac1p-\frac1q}$. Following the discussion in \cite{zhu2026stable}, we first take $\beta=k^{\frac1p-\frac1q}$. Then $\beta^p k^{-\frac{q-p}{q}}=1$, and hence the auxiliary equation in Lemma~\ref{lem:lemma2} becomes $z^q-z^p-2=0$. Moreover, $\Psi=k^{\frac{q-p}{q}}(2+z_0^p)$. Sybstituting this into $\Psi$ and $T_2$ yields $T_2=k^{\frac{2}{\min\{q,2\}}-\frac{2}{q}}(2+z_0^p)^2$. Similarly, \eqref{eq:zhuRIC} becomes $T_1	= 3^{2-2p} k^{\frac{2}{\min\{q,2\}}-\frac{2}{q}} (2+z_0)^{2p}$.
 
 \textbf{Case (1): $p\to0^+$ and $q\to1^+$.}
 Since $1<q\le2$ in this regime, $\min\{q,2\}=q$, and hence $\frac{2}{\min\{q,2\}}-\frac{2}{q}=0$. The root $z_0$ of equation $z^q-z^p-2=0$ satisfies $z_0 \to 3$ and $z_0^p\to1$. Therefore, $T_1 \to 9$. Thus the improved sufficient condition tends to $\delta_{2k}<\frac{1}{\sqrt{10}}.$ In the same limit, $T_2 \to 9$, so the limiting sufficient condition of \cite[Theorem~1]{zhu2026stable} is the same. This shows that, under the worst-case normalization of $\beta$, the strict finite-parameter improvement may disappear in this singular limit.
 
 \textbf{Case (2): $p\to0^+$ and $q\to\infty$.}
 Here $\min\{q,2\} = 2$, and hence $\frac{2}{\min\{q,2\}}-\frac{2}{q} \to 1$. From the auxiliary equation, one has $z_0\to1^+$ and $z_0^p\to1$. It follows that $T_{\rm new}(p,q,k)\to9k$. Therefore, $\delta_{2k}<\frac{1}{\sqrt{1+9k}}$. The original quantity $T_1$ has the same limiting value $9k$. Thus, in this worst-case scaling, both conditions have the same leading asymptotic form.
 
 \textbf{Case (3): $p=1$ and $q\to1^+$.}
 When $p=1$, the auxiliary equation becomes $z^q-z-2=0$. As $q \to 1^+$, its unique positive root satisfies $z_0\to+\infty$. Consequently, both $T_{\rm new}(1,q,k)$ and $T_{\rm Zhu}(1,q,k)$ diverge, and the corresponding upper bounds on $\delta_{2k}$ tend to zero. This limiting regime is
 therefore not useful for stable recovery.
 
 \textbf{Case (4): $p=1$ and $q\to\infty$.}
 In this case, $z_0\to1^+$, and $T_{\rm new}(1,q,k)\to9k$. Therefore, $\delta_{2k}<\frac{1}{\sqrt{1+9k}}$. Since $p=1$, Proposition~\ref{prop:deltaimprove} gives equality between the two RIP quantities, and the limiting condition coincides with that of \cite[Theorem~1]{zhu2026stable}.
 
 Second, the bounded norm-ratio regime. The preceding discussion follows the worst-case sparsity-based choice of $\beta$. To reveal the improvement of the proposed condition more explicitly, we now consider the regime where the norm-ratio parameter remains bounded as $p \to 0^+$. More precisely, assume that $\beta$ is fixed, or more generally that $\beta^p\to1$ as $p\to0^+$. This regime corresponds to signals whose norm-ratio parameter does not grow at the worst-case rate $k^{1/p-1/q}$.
 
 \textbf{Case (5): $p \to 0^+$ and $q \to 1^+$ with bounded $\beta$.}
 Since $\beta^p\to1$, and $k^{-\frac{q-p}{q}}\to k^{-1}$, the auxiliary equation in Lemma~\ref{lem:lemma2} tends to $z-k^{-1}-k^{-1}-1=0$. Thus $z_0\to1+\frac{2}{k}$ and $z_0^p\to1$. Moreover, $\Psi\to k+2$. Since $\frac{2}{\min\{q,2\}}-\frac{2q+2-2p}{q}\to-2$, we obtain $T_2 \to k^{-2}(k+2)^2 = \left(1+\frac{2}{k}\right)^2$. Hence the improved RIP condition tends to
$$
	 \delta_{2k}<\frac{1}{\sqrt{1+\left(1+\frac{2}{k}\right)^2}}.
$$
 In particular, as $k\to\infty$, $\delta_{2k}<\frac{1}{\sqrt{2}}$. On the other hand, the corresponding quantity in \cite[Theorem~1]{zhu2026stable} satisfies $T_1 \to 9$, so that the limiting condition is $\delta_{2k}<\frac{1}{\sqrt{10}}$. Thus, in the bounded norm-ratio regime, the improved condition gives a
 substantially larger upper bound on $\delta_{2k}$.
 
 \textbf{Case (6): $p\to0^+$ and $q\to\infty$ with bounded $\beta$.}
 In this regime, $\beta^p \to 1$ and $k^{-\frac{q-p}{q}}\to k^{-1}$. The auxiliary equation asymptotically becomes $z^q-\frac{2}{k}-1=0$, and hence $z_0\to1^+$, and $z_0^p\to1$. Again, $\Psi \to k+2$. Since $\frac{2}{\min\{q,2\}}-\frac{2q+2-2p}{q}\to-1$, we get $T_2 \to k^{-1}(k+2)^2 = k+4+\frac{4}{k}$. Therefore,
$$
	 \delta_{2k} < \frac{1}{\sqrt{k+5+4/k}}.
$$
 For large $k$, this behaves as $\delta_{2k}\lesssim \frac{1}{\sqrt{k}}$. By contrast, the corresponding limiting condition in \cite[Theorem~1]{zhu2026stable} is $\delta_{2k} < \frac{1}{\sqrt{1+9k}} \sim \frac{1}{3\sqrt{k}}$. Hence, in this bounded norm-ratio regime, the upper bound on $\delta_{2k}$ given by the improved condition is asymptotically three times larger than that of \cite{zhu2026stable}.
 
 The above analysis shows that the proposed RIP sufficient condition is always at least as mild as that in \cite[Theorem~1]{zhu2026stable}. Under the worst-case sparsity-based normalization, the leading constants may coincide in singular limits. However, for bounded norm-ratio signal classes, the improvement remains visible even in the extreme regimes $p\to 0^+$, yielding substantially larger upper bounds on $\delta_{2k}$.

\section{Stable recovery guarantee without $k$-sparse assumption}\label{sec:StablerecoverywithoutK}
The recovery results in the preceding section are established under the assumption that the true signal is exactly $k$-sparse. This assumption is mathematically useful, but it is often too restrictive in applications where the sparsity level is unknown in advance or the signal is only compressible. In such cases, the integer $k$ should be regarded as an analysis parameter rather than the exact number of nonzero entries, and the recovery error should naturally depend on the best $k$-term approximation tail $\bm{x}_{-\max(k)}$ as well as the noise level $\epsilon$.

We therefore develop stable recovery guarantees for the $\ell_p^p/\ell_q^p$ model without imposing the $k$-sparse assumption on the true signal. The starting point is the RIP--ROP framework of Zhu et al.~\cite[Theorem~2]{zhu2026stable}. Their result already covers the non-$k$-sparse setting, but for $1<q<2$, Zhu et al.'s proof uses the finite-dimensional estimate $\|\bm{h}\|_q \le n^{1/q-1/2}\|\bm{h}\|_2$ for the full error vector $\bm{h}$. Our proof avoids this full-vector estimate by applying the same norm relation on localized blocks whose cardinalities are controlled by $k$ and $t$. This yields a dimension-free RIP--ROP recovery bound for all $q>1$. We then give a related RIP-only formulation by controlling the disjoint-support cross terms through the RIC $\delta_{k+t}$, which provides a condition stated solely in terms of RIP constants.

\subsection{Main result based on RIP and ROP}

This subsection first treats the non-$k$-sparse case within the RIP--ROP framework. The integer $k$ is not assumed to be the true sparsity level of $\bm{x}$; instead, it specifies the leading part of the recovery error and the approximation tail $\bm{x}_{-\max(k)}$. The additional integer $t$ is used to partition the remaining error into ordered blocks. This block structure replaces the full-vector estimate used in \cite[Theorem~2]{zhu2026stable} by separate estimates on individual blocks: for each block, we first control its error using the $\ell_q$ norm and then translate this control into an $\ell_2$ error bound, which removes the explicit ambient-dimension factor in the range $1<q<2$. The RIP--ROP theorem below gives a sufficient condition based on $k$-RIP and $(k,t)$-ROP, expressed through the constants $\delta_k$ and $\theta_{k,t}$. It also provides a stable error bound depending on $\|\bm{x}_{-\max(k)}\|_p$ and $\epsilon$. We first recall the auxiliary estimates needed for the proof and then introduce the localized block $\ell_q$--$\ell_2$ estimate in Lemma \ref{lem:block_lq_l2}.
\begin{lemma}[{\cite[Lemma~3]{zhu2026stable}}]\label{lem:zhu_tail_lp}
	Assume that $\bm{x}\neq\bm{0}$ belongs to $\{ \bm{z} \mid \| \bm{b} - \bm{A} \bm{z} \|_2 \leq \epsilon \}$ and satisfies $\frac{\| \bm{x} \|_p}{\| \bm{x} \|_q} \leq \beta$ with $0 < p \leq 1$, $q > 1$ and $\beta > 0$. Let $\bm{\hat{x}}$ be a solution to $\mathcal{Q}_{p,q}^{\epsilon}$, $\bm{h} := \bm{\hat{x}} - \bm{x}$, $k \in (0, n]$ be an integer. Then, one has
$$
		\| \bm{h}_{-\max(k)} \|_p^p \leq \| \bm{h}_{\max(k)} \|_p^p + 2 \| \bm{x}_{-\max(k)} \|_p^p + \beta^p \| \bm{h} \|_q^p.
$$
\end{lemma}

\begin{lemma}[{\cite[Lemma~5.5]{cai2013compressed}}]\label{lem:cai_zhang_tail}
	Suppose $n \geq k$, $u \geq 0$, $v_1 \geq v_2 \geq \cdots \geq v_n$, and $\sum_{i=1}^k v_i + u \geq \sum_{i=k+1}^n v_i$, then for all $\rho \geq 1$,
$$
		\sum_{i=k+1}^n v_i^\rho \leq k ( ( \frac{\sum_{i=1}^k v_i^\rho}{k} )^{1/\rho} + \frac{u}{k} )^\rho.
$$
\end{lemma}

\begin{lemma}[{\cite[Proposition~2.1]{cai2010newbounds}}]\label{lem:lem5}
	For any $\bm{a} \in \mathbb{R}^n$,
$$
		\| \bm{a} \|_2 - \frac{\| \bm{a} \|_1}{\sqrt{n}} \leq \frac{\sqrt{n}}{4} \left( \max_{1 \leq i \leq n} |a_i| - \min_{1 \leq i \leq n} |a_i| \right).
$$
\end{lemma}

\begin{lemma} \label{lem:block_lq_l2}
	Let $T_0 = \text{supp}(\bm{h}_{\max(k)})$ and let $T_i$ ($i \ge 1$) be the sets of indices of the $t$ largest elements in absolute value of $\bm{h}_{(T_0 \cup \cdots \cup T_{i-1})^c}$. Let
$$
		\vartheta_q:=\max\{\frac1q-\frac12,0\}.
$$
	Then, for every $q>1$,
	\begin{equation}\label{eq:lq_block_bound}
		\|\bm{h}\|_q \le k^{\vartheta_q} \|\bm{h}_{\max(k)}\|_2 + t^{\vartheta_q} \sum_{i=1}^J \|\bm{h}_{T_i}\|_2. 
	\end{equation}
\end{lemma}

\begin{proof}
	For $q > 1$, the $\ell_q$-norm is a true norm and satisfies the triangle inequality. Decomposing $\bm{h}$ over the disjoint index blocks, we have
$$
		\|\bm{h}\|_q \le \|\bm{h}_{T_0}\|_q + \sum_{i=1}^J \|\bm{h}_{T_i}\|_q.
$$
	Since $|T_0| \le k$ and $|T_i| \le t$, we use the standard finite-dimensional norm relation
$$
		\|\bm{v}\|_q \le s^{\max\{1/q-1/2,0\}}\|\bm{v}\|_2,\quad \|\bm{v}\|_0\le s.
$$
	It follows that $\|\bm{h}_{T_0}\|_q \le k^{\vartheta_q}\|\bm{h}_{T_0}\|_2$ and $\|\bm{h}_{T_i}\|_q \le t^{\vartheta_q}\|\bm{h}_{T_i}\|_2$. Substituting these local bounds into the triangle inequality and noting that $\bm{h}_{T_0}=\bm{h}_{\max(k)}$ gives \eqref{eq:lq_block_bound}.
\end{proof}

We now present a dimension-free refinement of the non-$k$-sparse recovery guarantee of \cite[Theorem~2]{zhu2026stable}, which relies on both RIP and ROP.

\begin{theorem}[]\label{thm:ImprovedROPRIP}
	Assume that the true signal $\bm{x}\neq\bm{0}$ in $\mathcal{M}\{\bm{A},\bm{x},\bm{b},m,n,\epsilon\}$ satisfies $\frac{\|\bm{x}\|_p}{\|\bm{x}\|_q}\le \beta$ with $0<p\le 1$, $q>1$, and $\beta>0$. Let $\hat{\bm{x}}$ be a solution to $\mathcal{Q}_{p,q}^{\epsilon}$. Let $1\le k<n$ and $1\le t\le n-k$ be integers. Define
	$a_p:=3^{\frac{1-p}{p}}$, $\vartheta_q:=\max\{\frac{1}{q}-\frac{1}{2},\,0\}$,			$\widetilde{\eta}_{p,q}(k,t,\beta):=
			a_p\,\beta\,k^{\frac{p-1}{p}}\,t^{\vartheta_q-\frac12}$,
		$	\widetilde{\tau}_{p,q}(k,t,\beta):=\frac{a_p\sqrt{\frac{k}{t}}+\frac14\sqrt{\frac{t}{k}}+a_p\,\beta\,k^{\frac{p-1}{p}+\vartheta_q}\,t^{-\frac12}}{1-\widetilde{\eta}_{p,q}(k,t,\beta)}$,	
	and $\widetilde{\psi}_{p,q}(k,t,\beta):=\delta_k+\widetilde{\tau}_{p,q}(k,t,\beta)\,\theta_{k,t}$. If $\bm{A}$ satisfies $k$-RIP and $(k,t)$-ROP with $\widetilde{\eta}_{p,q}(k,t,\beta)<1$ and $\widetilde{\psi}_{p,q}(k,t,\beta)<1$, then
$$
		\|\hat{\bm{x}}-\bm{x}\|_2\le \widetilde{C}_1\|\bm{x}_{-\max(k)}\|_p+\widetilde{C}_2\epsilon,
$$
	where
	\begin{equation*}
			\widetilde{C}_1= \frac{a_p\,2^{\frac1p}\,k^{\frac{p-1}{p}}\,t^{-\frac12}\,(1-\delta_k+\theta_{k,t})}{\bigl(1-\widetilde{\eta}_{p,q}(k,t,\beta)\bigr)\bigl(1-\widetilde{\psi}_{p,q}(k,t,\beta)\bigr)},\quad
			\widetilde{C}_2=\frac{2\bigl(1+\widetilde{\tau}_{p,q}(k,t,\beta)\bigr)\sqrt{1+\delta_k}}{1-\widetilde{\psi}_{p,q}(k,t,\beta)}.
	\end{equation*}
\end{theorem}

\begin{proof}
	Let $\bm{h} := \hat{\bm{x}} - \bm{x}$. If $\bm{h} = \bm{0}$, then the conclusion is trivial. Hence we only consider the case $\bm{h} \neq \bm{0}$.
	
	Let $T_0 := \operatorname{supp}(\bm{h}_{\max(k)})$, and let $T_i$ $(i\ge 1)$ be the sets of indices of the $t$ largest elements in absolute value of $\bm{h}_{(T_0\cup\cdots\cup T_{i-1})^c}$. Let $J = \lceil (n-k)/t \rceil$ be the total number of such tail blocks; the last block is allowed to have cardinality smaller than $t$.
	For notational brevity, we define $S := \sum_{i=1}^J \|\bm{h}_{T_i}\|_2$, $ u := \|\bm{h}_{\max(k)}\|_2$,  and $\sigma := \|\bm{x}_{-\max(k)}\|_p$. Since $\bm{h}\neq\bm{0}$ and $k\ge1$, one has $u>0$. It follows from Lemma~\ref{lem:zhu_tail_lp} that
	\begin{equation*}
		  \|\bm{h}_{-\max(k)}\|_p^p \le \|\bm{h}_{\max(k)}\|_p^p +2\|\bm{x}_{-\max(k)}\|_p^p + \beta^p\|\bm{h}\|_q^p:= u_0.
	\end{equation*}
	Applying Lemma~\ref{lem:cai_zhang_tail} with $\rho=\frac{1}{p}$ to bridge the $\ell_p$ and $\ell_1$ norms, we obtain
	\begin{equation*}
		\|\bm{h}_{-\max(k)}\|_1 \le k( (\frac{\|\bm{h}_{\max(k)}\|_1}{k})^p + \frac{u_0}{k})^{\frac{1}{p}}.
	\end{equation*}
	Using the $p$-norm subadditivity inequality $(A+B+C)^{\frac{1}{p}} \le 3^{\frac{1-p}{p}}\left(A^{\frac{1}{p}}+B^{\frac{1}{p}}+C^{\frac{1}{p}}\right)$, and substituting $u_0$, one has
$$
		\|\bm{h}_{-\max(k)}\|_1 \le k \cdot 3^{\frac{1-p}{p}} ( \frac{\|\bm{h}_{\max(k)}\|_1}{k} + \frac{2^{\frac{1}{p}}\sigma}{k^{\frac{1}{p}}} + \frac{\beta\|\bm{h}\|_q}{k^{\frac{1}{p}}} ).
$$
	Together with the standard Cauchy-Schwarz bound $\|\bm{h}_{\max(k)}\|_1 \le \sqrt{k}\,\|\bm{h}_{\max(k)}\|_2 = \sqrt{k}u$, and recalling $a_p := 3^{\frac{1-p}{p}}$, we obtain the foundational $\ell_1$ tail bound
$$
		\|\bm{h}_{-\max(k)}\|_1 \le a_p\sqrt{k}\,u + a_p\,k^{\frac{p-1}{p}} (2^{\frac1p}\sigma + \beta\|\bm{h}\|_q).
$$
	Applying Lemma~\ref{lem:lem5} to each tail block, with the final block padded by zeros if its cardinality is smaller than $t$, gives
	\begin{equation}
		S \leq \frac{1}{\sqrt{t}} \sum_{i=1}^{J} \| \bm{h}_{T_i} \|_1 + \frac{\sqrt{t}}{4} \max_{j \in T_1} |h_j| \leq \frac{1}{\sqrt{t}}\|\bm{h}_{-\max(k)}\|_1 + \frac14\sqrt{\frac{t}{k}}\,u.
	\end{equation}
	In the displayed estimate, the maximum and minimum for the last block are understood after the possible zero-padding. The ordering of the blocks implies $\max_{j\in T_{i+1}}|h_j|\le \min_{j\in T_i}|h_j|$ for full preceding blocks, while zero-padding makes the last minimum equal to zero if the last block is not full. Moreover, if $|h|_{(1)}\ge\cdots\ge |h|_{(n)}$ denotes the nonincreasing rearrangement of $|h|$, then
$$
		\max_{j\in T_1}|h_j|\le |h|_{(k+1)}\le (\frac1k\sum_{j=1}^k |h|_{(j)}^2)^{1/2}=\frac{u}{\sqrt{k}},
$$
	with the same conclusion being trivial when $\bm{h}$ has at most $k$ nonzero entries.
	Substituting the bound for $\|\bm{h}_{-\max(k)}\|_1$, it follows that
	\begin{equation}\label{eq:sbound}
		S \le ( a_p\sqrt{\frac{k}{t}} + \frac14\sqrt{\frac{t}{k}} )u + a_p\,k^{\frac{p-1}{p}}t^{-\frac12} (2^{\frac1p}\sigma + \beta\|\bm{h}\|_q).
	\end{equation}
	Since the block sizes satisfy $|T_0|\le k$ and $|T_i|\le t$, we apply the finite-dimensional norm relation $\|\bm{h}_{T_0}\|_q\le k^{\vartheta_q}\|\bm{h}_{T_0}\|_2$ and   $\|\bm{h}_{T_i}\|_q\le t^{\vartheta_q}\|\bm{h}_{T_i}\|_2$, where $\vartheta_q=\max\{\frac1q-\frac12,0\}$. By the triangle inequality for $q>1$, one has
	\begin{equation}\label{eq:sloop}
		\|\bm{h}\|_q \le \|\bm{h}_{T_0}\|_q + \sum_{i=1}^J \|\bm{h}_{T_i}\|_q \le k^{\vartheta_q}u + t^{\vartheta_q}S.
	\end{equation}
	Substituting \eqref{eq:sloop} into \eqref{eq:sbound} yields
$$
		S \le ( a_p\sqrt{\frac{k}{t}} + \frac14\sqrt{\frac{t}{k}} + a_p\beta k^{\frac{p-1}{p}+\vartheta_q}t^{-\frac12} )u + a_p\,2^{\frac1p}k^{\frac{p-1}{p}}t^{-\frac12}\sigma + a_p\beta k^{\frac{p-1}{p}}t^{\vartheta_q-\frac12}S.
$$
	It follows from the definition of $\widetilde{\eta}_{p,q}(k,t,\beta)$ and $\widetilde{\tau}_{p,q}(k,t,\beta)$ that 
$$
		(1-\widetilde{\eta}_{p,q}(k,t,\beta))S \le (1-\widetilde{\eta}_{p,q}(k,t,\beta))\widetilde{\tau}_{p,q}(k,t,\beta)\,u + a_p\,2^{\frac1p}k^{\frac{p-1}{p}}t^{-\frac12}\sigma.
$$
	Since the condition $\widetilde{\eta}_{p,q}(k,t,\beta)<1$ guarantees positivity, dividing to isolate $S$ yields
	\begin{equation}\label{eq:Sbpund}
		S \le \widetilde{\tau}_{p,q}(k,t,\beta)\,u + \frac{ a_p\,2^{\frac1p}k^{\frac{p-1}{p}}t^{-\frac12} }{ 1-\widetilde{\eta}_{p,q}(k,t,\beta) }\sigma.
	\end{equation}
	It follows from the triangle inequality $\|\bm{h}\|_2 \le u+S$ that
	\begin{equation}\label{eq:hbpund}
		\|\bm{h}\|_2 \le \bigl(1+\widetilde{\tau}_{p,q}(k,t,\beta)\bigr)u + \frac{ a_p\,2^{\frac1p}k^{\frac{p-1}{p}}t^{-\frac12} }{ 1-\widetilde{\eta}_{p,q}(k,t,\beta) }\sigma. 
	\end{equation}
	By the Cauchy--Schwarz inequality and the $k$-RIP, together with the bound $\|\bm{A}\bm{h}\|_2 \le 2\epsilon$, we obtain
	\begin{equation}\label{eq:Ahbound}
		|\langle \bm{A}\bm{h},\bm{A}\bm{h}_{\max(k)}\rangle| \le \|\bm{A}\bm{h}\|_2\,\|\bm{A}\bm{h}_{\max(k)}\|_2 \le 2\epsilon\sqrt{1+\delta_k}\,u.
	\end{equation}
	Expanding the inner product and applying the $(k,t)$-ROP to the disjoint blocks, it follows that
	\begin{equation}
		\begin{aligned}
			\|\langle \bm{A}\bm{h}, \bm{A}\bm{h}_{\max(k)} \rangle \|
			&\overset{\text{a}}{=}\| \langle \bm{A}\bm{h}_{\max(k)}, \bm{A}\bm{h}_{\max(k)} \rangle + \sum_{i=1}^J \langle \bm{A}\bm{h}_{T_i}, \bm{A}\bm{h}_{\max(k)} \rangle \| \nonumber \\
			&\overset{b}{\ge} \left|\langle \bm{A}\bm{h}_{\max(k)}, \bm{A}\bm{h}_{\max(k)} \rangle\right| - \sum_{i=1}^J \left| \langle \bm{A}\bm{h}_{T_i}, \bm{A}\bm{h}_{\max(k)} \rangle \right| \nonumber \\
			&\overset{c}{\ge} (1-\delta_k)\|\bm{h}_{\max(k)}\|_2^2 - \sum_{i=1}^J \left| \langle \bm{A}\bm{h}_{T_i}, \bm{A}\bm{h}_{\max(k)} \rangle \right| \nonumber \\
			&\overset{d}{\ge} (1-\delta_k)u^2 - \theta_{k,t}u \sum_{i=1}^J \|\bm{h}_{T_i}\|_2,
		\end{aligned}
	\end{equation}
	where (a) follows from $A \left( \bm{h}_{\max(k)} + \sum_i \bm{h}_{\mathcal{T}_i} \right) = A \bm{h}_{\max(k)} + \sum_i A \bm{h}_{\mathcal{T}_i}$, (b) follows from $\left| a + \sum_i b_i \right| \geq |a| - \sum_i |b_i|$, (c) follows from $\| A \bm{h}_{\max(k)} \|_2^2 \geq (1 - \delta_k) \| \bm{h}_{\max(k)} \|_2^2$, and (d) follows from $|\langle A x, A y \rangle| \leq \theta_{k,t} \| x \|_2 \| y \|_2$.  Substituting the bound into \eqref{eq:Sbpund} yields
	\begin{equation}\label{eq:Ahmaxbound}
		|\langle \bm{A}\bm{h},\bm{A}\bm{h}_{\max(k)}\rangle| \ge \bigl(1-\widetilde{\psi}_{p,q}(k,t,\beta)\bigr)u^2 - \frac{ a_p\,2^{\frac1p}k^{\frac{p-1}{p}}t^{-\frac12}\theta_{k,t} }{ 1-\widetilde{\eta}_{p,q}(k,t,\beta) }\sigma\,u. 
	\end{equation}
	Combining the upper bound \eqref{eq:Ahbound} and lower bound \eqref{eq:Ahmaxbound}, and utilizing the assumption $1-\widetilde{\psi}_{p,q}(k,t,\beta)>0$, dividing by $u>0$ and isolate $u$ yields
	\begin{equation}\label{eq:absboubd}
		u \le \frac{ a_p\,2^{\frac1p}k^{\frac{p-1}{p}}t^{-\frac12}\theta_{k,t} }{ \bigl(1-\widetilde{\eta}_{p,q}(k,t,\beta)\bigr)\bigl(1-\widetilde{\psi}_{p,q}(k,t,\beta)\bigr) }\sigma + \frac{ 2\sqrt{1+\delta_k} }{ 1-\widetilde{\psi}_{p,q}(k,t,\beta) }\epsilon.
	\end{equation}
	Substituting \eqref{eq:absboubd} into \eqref{eq:hbpund}, and using the algebraic simplification $1 + \frac{(1+\widetilde{\tau})\theta_{k,t}}{1-\widetilde{\psi}} = \frac{1-\delta_k+\theta_{k,t}}{1-\widetilde{\psi}}$ yields precisely
	\begin{equation*}
		\|\bm{h}\|_2 \le \frac{ a_p\,2^{\frac1p}k^{\frac{p-1}{p}}t^{-\frac12}(1-\delta_k+\theta_{k,t}) }{ \bigl(1-\widetilde{\eta}_{p,q}(k,t,\beta)\bigr)\bigl(1-\widetilde{\psi}_{p,q}(k,t,\beta)\bigr) }\sigma + \frac{ 2\bigl(1+\widetilde{\tau}_{p,q}(k,t,\beta)\bigr)\sqrt{1+\delta_k} }{ 1-\widetilde{\psi}_{p,q}(k,t,\beta) }\epsilon.
	\end{equation*}
	That is, $\|\hat{\bm{x}}-\bm{x}\|_2 \le \widetilde{C}_1\|\bm{x}_{-\max(k)}\|_p+\widetilde{C}_2\epsilon$. This completes the proof.
\end{proof}

\begin{remark}[Dimension-free refinement]\label{rem:dimension_free}
	Compared with \cite[Theorem~2]{zhu2026stable}, the main technical change is that the finite-dimensional estimate $\|\bm{h}\|_q\le n^{1/q-1/2}\|\bm{h}\|_2$ is not applied to the full error vector when $1<q<2$. Instead, the same norm relation is applied on localized blocks whose cardinalities are controlled by $k$ and $t$, as in Lemma~\ref{lem:block_lq_l2}. Consequently, the sufficient condition and the error constants above contain only $k$, $t$, $p$, $q$, and $\beta$, and have no explicit dependence on the ambient dimension $n$. This removes the dimension-dependent factor appearing in \cite[Theorem~2 and Remark~6]{zhu2026stable} for the range $1<q<2$.
\end{remark}

The preceding theorem avoids the full-vector factor $n^{1/q-1/2}$, but it still involves the ROC $\theta_{k,t}$. In applications, ROC estimates are often less convenient to check than RIC estimates. It is therefore useful to record a RIP-only variant of the same argument. The only modification is in the treatment of the cross terms: for two disjointly supported vectors whose supports have cardinalities at most $k$ and $t$, respectively, the $(k+t)$-RIP implies the bound
\[
|\langle \bm{A}\bm{u},\bm{A}\bm{v}\rangle|\le \delta_{k+t}\|\bm{u}\|_2\|\bm{v}\|_2.
\]
Replacing the ROC estimate by this RIP-controlled estimate yields the following result. Since generally $\theta_{k,t}\le \delta_{k+t}$, the RIP-only version should be viewed as more directly verifiable, rather than uniformly sharper, than the RIP--ROP version.

\begin{theorem}\label{thm:pureRIP}
	Assume that the true signal $\bm{x}\neq\bm{0}$ in $\mathcal{M}\{\bm{A},\bm{x},\bm{b},m,n,\epsilon\}$ satisfies $\frac{\|\bm{x}\|_p}{\|\bm{x}\|_q}\le \beta$
	with $0<p\le 1$, $q>1$, and $\beta>0$. Let $\hat{\bm{x}}$ be a solution to $\mathcal{Q}_{p,q}^{\epsilon}$. Let $1\le k<n$ and $1\le t\le n-k$ be integers. Define
$a_p:=3^{\frac{1-p}{p}}$, $\vartheta_q:=\max\left\{\frac{1}{q}-\frac{1}{2},\,0\right\}$,
		$	\rho_{p}(k,t):=	a_p\sqrt{\frac{k}{t}}+\frac14\sqrt{\frac{t}{k}}$, 
			$\alpha_{p,q}(k,t,\beta):=a_p\,\beta\,k^{\frac{p-1}{p}+\vartheta_q}\,t^{-\frac12}$,
		$	\overline{\eta}_{p,q}(k,t,\beta):=a_p\,\beta\,k^{\frac{p-1}{p}}\,t^{\vartheta_q-\frac12}$,  $C_{p}(k,t):=a_p\,2^{\frac1p}\,k^{\frac{p-1}{p}}\,t^{-\frac12}$,
			$\overline{\tau}_{p,q}(k,t,\beta):=\frac{\rho_p(k,t)+\alpha_{p,q}(k,t,\beta)}{1-\overline{\eta}_{p,q}(k,t,\beta)}$,
	and $\overline{\psi}_{p,q}(k,t,\beta):=\delta_k+\delta_{k+t}\,\overline{\tau}_{p,q}(k,t,\beta)$. If $\bm{A}$ satisfies $(k+t)$-RIP with $\overline{\eta}_{p,q}(k,t,\beta)<1$ and $\overline{\psi}_{p,q}(k,t,\beta)<1$, then
$$
		\|\hat{\bm{x}}-\bm{x}\|_2 \le \overline{C}_1\|\bm{x}_{-\max(k)}\|_p+\overline{C}_2\epsilon,
$$
	where
	$$
			\overline{C}_1:=\frac{C_p(k,t)}{1-\overline{\eta}_{p,q}(k,t,\beta)}(1+	\frac{(1+\overline{\tau}_{p,q}(k,t,\beta))\delta_{k+t}}{1-\overline{\psi}_{p,q}(k,t,\beta)}),\quad
			\overline{C}_2:=\frac{2(1+\overline{\tau}_{p,q}(k,t,\beta)r)\sqrt{1+\delta_k}}{1-\overline{\psi}_{p,q}(k,t,\beta)}.
		$$
\end{theorem}
\begin{proof}
	Let the recovery error vector be denoted by $\bm{h} := \hat{\bm{x}} - \bm{x}$. The case $\bm{h} = \bm{0}$ is trivial; thus, we solely consider $\bm{h} \neq \bm{0}$.
	
	Following the same block-decomposition strategy established in Theorem~\ref{thm:ImprovedROPRIP}, we define $T_0 := \operatorname{supp}(\bm{h}_{\max(k)})$. The complementary support $T_0^c$ is partitioned into disjoint subsets $T_1, \dots, T_J$, where each $T_i$ contains the indices of the $t$ largest elements in absolute value of $\bm{h}_{(T_0\cup\cdots\cup T_{i-1})^c}$. Set
	$
		S:=\sum_{i=1}^J\|\bm{h}_{T_i}\|_2,\quad u:=\|\bm{h}_{\max(k)}\|_2$, and $ \sigma:=\|\bm{x}_{-\max(k)}\|_p$. Since $\bm{h}\neq \bm{0}$ and $k\ge 1$, one has $u>0$. The part of the proof of Theorem~\ref{thm:ImprovedROPRIP} up to the estimate of $S$ uses only Lemmas~\ref{lem:zhu_tail_lp}--\ref{lem:block_lq_l2} and does not use ROP. Hence, with the present notation,
	\begin{equation}\label{eq:pure_rip_S_bound}
		S\le \overline{\tau}_{p,q}(k,t,\beta)u+\frac{C_p(k,t)}{1-\overline{\eta}_{p,q}(k,t,\beta)}\sigma .
	\end{equation}
	Consequently,
	\begin{equation}\label{eq:pure_rip_h2_total}
		\|\bm{h}\|_2 \le \bigl(1+\overline{\tau}_{p,q}(k,t,\beta)\bigr)u + \frac{ C_p(k,t) }{ 1-\overline{\eta}_{p,q}(k,t,\beta) }\sigma. 
	\end{equation}
	Since $(k+t)$-RIP implies $k$-RIP by monotonicity of the RIC, the Cauchy--Schwarz inequality and the feasibility constraint $\|\bm{A}\bm{h}\|_2 \le 2\epsilon$ give
$$
		|\langle \bm{A}\bm{h}_{\max(k)},\bm{A}\bm{h}\rangle| \le 2\epsilon\sqrt{1+\delta_k}\,u. 
$$
	For the lower bound, expanding the inner product yields
$$
		|\langle \bm{A}\bm{h}_{\max(k)},\bm{A}\bm{h}\rangle| \ge (1-\delta_k)u^2 - \sum_{i=1}^J \left| \langle \bm{A}\bm{h}_{\max(k)},\bm{A}\bm{h}_{T_i}\rangle \right|.
$$
	Since the vectors $\bm{h}_{\max(k)}$ and $\bm{h}_{T_i}$ possess disjoint supports with cardinalities at most $k$ and $t$, respectively, their union support is bounded by $k+t$. By the standard disjoint-support consequence of the $(k+t)$-RIP, obtained by applying the RIP to $\bm{u}\pm c\bm{v}$ and optimizing over $c>0$, one has
$$
		\left| \langle \bm{A}\bm{h}_{\max(k)},\bm{A}\bm{h}_{T_i}\rangle \right| \le \delta_{k+t}\,u\,\|\bm{h}_{T_i}\|_2.
$$
	Consequently, the lower bound becomes
$$
		|\langle \bm{A}\bm{h}_{\max(k)},\bm{A}\bm{h}\rangle| \ge (1-\delta_k)u^2 - \delta_{k+t}\,u\,S.
$$
	Substituting \eqref{eq:pure_rip_S_bound}, one has
$$
		|\langle \bm{A}\bm{h}_{\max(k)},\bm{A}\bm{h}\rangle| \ge \bigl(1-\overline{\psi}_{p,q}(k,t,\beta)\bigr)u^2 - \delta_{k+t} \frac{ C_p(k,t) }{ 1-\overline{\eta}_{p,q}(k,t,\beta) }\sigma\,u, 
$$
	where $\overline{\psi}_{p,q}(k,t,\beta) := \delta_k + \delta_{k+t}\overline{\tau}_{p,q}(k,t,\beta)$. Under the strict assumption that $1-\overline{\psi}_{p,q}(k,t,\beta)>0$, we have
	\begin{equation}\label{eq:proof_u_final}
		u \le \frac{ \delta_{k+t}C_p(k,t) }{ \bigl(1-\overline{\eta}_{p,q}(k,t,\beta)\bigr)\bigl(1-\overline{\psi}_{p,q}(k,t,\beta)\bigr) }\sigma + \frac{ 2\sqrt{1+\delta_k} }{ 1-\overline{\psi}_{p,q}(k,t,\beta) }\epsilon.
	\end{equation}
	Finally, substituting \eqref{eq:proof_u_final} back into \eqref{eq:pure_rip_h2_total} yields
	\begin{equation*}
		\|\bm{h}\|_2 \le \frac{ C_p(k,t) }{ 1-\overline{\eta}_{p,q}(k,t,\beta) } \left( 1+ \frac{ \bigl(1+\overline{\tau}_{p,q}(k,t,\beta)\bigr)\delta_{k+t} }{ 1-\overline{\psi}_{p,q}(k,t,\beta) } \right)\sigma + \frac{ 2\bigl(1+\overline{\tau}_{p,q}(k,t,\beta)\bigr)\sqrt{1+\delta_k} }{ 1-\overline{\psi}_{p,q}(k,t,\beta) }\epsilon.
   	\end{equation*}
	This precisely coincides with $\|\hat{\bm{x}}-\bm{x}\|_2 \le \overline{C}_1\|\bm{x}_{-\max(k)}\|_p + \overline{C}_2\epsilon$.
\end{proof}

\begin{remark}
	Theorem~\ref{thm:pureRIP} has the same proof architecture as Theorem~\ref{thm:ImprovedROPRIP}, with the ROC $\theta_{k,t}$ replaced by the RIC $\delta_{k+t}$ in the estimates of the cross terms. Since the standard RIP implication gives $\theta_{k,t}\le \delta_{k+t}$, the RIP-only condition is generally easier to verify but may be more conservative than the RIP--ROP condition. Thus the purpose of Theorem~\ref{thm:pureRIP} is not to claim a uniformly sharper constant, but to provide a dimension-free recovery guarantee stated solely in terms of RIP constants.
\end{remark}

\section{Algorithm}\label{sec:Algorithm}
In this section, we develop a standard Dinkelbach-type linearized proximal framework for the proposed ratio minimization model and establish convergence properties for the resulting outer sequence. In the numerical implementation, the resulting linearized-proximal subproblems are treated by an ADMM routine.

The starting point is the link between the constrained ratio problem~\eqref{eq:ratio_model} and a parameter-dependent difference formulation. Proposition~\ref{prop:equivalence} provides the sign characterization that motivates the standard Dinkelbach ratio update. The convergence analysis is then developed for the prox-linear outer scheme below.

\subsection{Dinkelbach Method for the Generalized Ratio Model}

Let $\mathcal{F}:=\{\bm{x}\in\mathbb{R}^n\mid \bm{A}\bm{x}=\bm{b}\}$. Fractional programming connects the constrained ratio model with the parametric value function
$
g(\alpha):=\sup_{\bm{x}\in\mathcal{F}}\left\{\alpha\|\bm{x}\|_q^p-\|\bm{x}\|_p^p\right\}.
$
The following proposition records the sign characterization of $g$ and provides the motivation for the standard Dinkelbach ratio update used below.

\begin{proposition}[] \label{prop:equivalence}
	Consider the constrained ratio minimization problem and denote its optimal value as $\alpha^*$, that is
	\begin{equation}
		\alpha^* := \inf_{\bm{x} \in \mathbb{R}^n} \left\{ \frac{\|\bm{x}\|_p^p}{\|\bm{x}\|_q^p} \quad \text{s.t.} \quad \bm{A}\bm{x} = \bm{b} \right\}, \label{eq:ratio_model}
	\end{equation}
	where $0 < p \le 1$, $q > 1$, and $\bm{b} \neq \bm{0}$. 
	For any scalar $\alpha > 0$, define the corresponding parameter-dependent value function
	\begin{equation}\label{eq:g_alpha_def}
		g(\alpha) := \sup_{\bm{x} \in \mathbb{R}^n} \left\{ \alpha\|\bm{x}\|_q^p - \|\bm{x}\|_p^p \quad \text{s.t.} \quad \bm{A}\bm{x} = \bm{b} \right\}. 
	\end{equation}
	Then, the parametric function $g(\alpha)$ exhibits the following rigorous properties linking it to the ratio model~\eqref{eq:ratio_model}:
		(a) If $g(\alpha) > 0$, then $\alpha > \alpha^*$;
		(b) If $g(\alpha) \le 0$, then $\alpha \le \alpha^*$;
		(c) If $g(\alpha) = 0$, then $\alpha = \alpha^*$.
\end{proposition}

\begin{proof}
	Let $\mathcal{F} = \{\bm{x} \in \mathbb{R}^n \mid \bm{A}\bm{x} = \bm{b}\}$ denote the feasible set.  Since $\bm{b} \neq \bm{0}$, the origin is not feasible, i.e., $\bm{0} \notin \mathcal{F}$. Consequently, for any $\bm{x} \in \mathcal{F}$, the denominator is strictly positive ($\|\bm{x}\|_q^p > 0$).
	
	(a) If $g(\alpha) > 0$, by the definition of the supremum~\eqref{eq:g_alpha_def}, there exists a feasible vector $\bm{x} \in \mathcal{F}$ such that $\alpha\|\bm{x}\|_q^p - \|\bm{x}\|_p^p > 0$. Since $\|\bm{x}\|_q^p > 0$, dividing both sides by this strictly positive value yields $\alpha > \frac{\|\bm{x}\|_p^p}{\|\bm{x}\|_q^p}$. 
	By definition~\eqref{eq:ratio_model}, $\frac{\|\bm{x}\|_p^p}{\|\bm{x}\|_q^p} \ge \inf_{\bm{x} \in \mathcal{F}} \frac{\|\bm{x}\|_p^p}{\|\bm{x}\|_q^p} = \alpha^*$. Therefore, we deduce $\alpha > \alpha^*$.
	
	(b) If $g(\alpha) \le 0$, it implies that for all $\bm{x} \in \mathcal{F}$, we have $\alpha\|\bm{x}\|_q^p - \|\bm{x}\|_p^p \le 0$. Dividing by $\|\bm{x}\|_q^p > 0$ gives $\alpha \le \frac{\|\bm{x}\|_p^p}{\|\bm{x}\|_q^p}$ for all $\bm{x} \in \mathcal{F}$. 
	Taking the infimum over the entire feasible set $\mathcal{F}$, we obtain $\alpha \le \inf_{\bm{x} \in \mathcal{F}} \frac{\|\bm{x}\|_p^p}{\|\bm{x}\|_q^p} = \alpha^*$. Hence, $\alpha \le \alpha^*$.
	
	(c) If $g(\alpha) = 0$, it immediately follows from part (b) that $\alpha \le \alpha^*$. 
	Furthermore, by the definition of the supremum, there exists a sequence $\{\bm{x}_k\} \subset \mathcal{F}$ such that
	\begin{equation}
		\lim_{k \to \infty} \left( \alpha\|\bm{x}_k\|_q^p - \|\bm{x}_k\|_p^p \right) = 0. \label{eq:limit_seq}
	\end{equation}
	Since $\bm{x}_k \in \mathcal{F}$, we have $\bm{A}\bm{x}_k = \bm{b}$. By matrix norm properties, $\|\bm{b}\|_2 = \|\bm{A}\bm{x}_k\|_2 \le \|\bm{A}\|_2 \|\bm{x}_k\|_2$. Thus, $\|\bm{x}_k\|_2 \ge \frac{\|\bm{b}\|_2}{\|\bm{A}\|_2} > 0$.
	Due to the finite-dimensional norm equivalence in $\mathbb{R}^n$, there exists a constant $C_{q,2} > 0$ (specifically, $C_{q,2} = n^{\min(0, 1/q - 1/2)}$) such that $\|\bm{x}_k\|_q \ge C_{q,2}\|\bm{x}_k\|_2$. 
	Therefore, $\|\bm{x}_k\|_q^p \ge \left( C_{q,2} \frac{\|\bm{b}\|_2}{\|\bm{A}\|_2} \right)^p := M > 0$. 
	Because the sequence $\|\bm{x}_k\|_q^p$ is bounded away from zero by a strictly positive constant $M$, dividing the limit~\eqref{eq:limit_seq} by $\|\bm{x}_k\|_q^p$ yields
	\begin{equation*}
		\lim_{k \to \infty} \frac{\alpha\|\bm{x}_k\|_q^p - \|\bm{x}_k\|_p^p}{\|\bm{x}_k\|_q^p} = 0 \implies \lim_{k \to \infty} \frac{\|\bm{x}_k\|_p^p}{\|\bm{x}_k\|_q^p} = \alpha.
	\end{equation*}
	By the definition of the optimal ratio, $\alpha^* \le \frac{\|\bm{x}_k\|_p^p}{\|\bm{x}_k\|_q^p}$ holds for all $k$. Taking the limit as $k \to \infty$ yields $\alpha^* \le \alpha$. 
	Combining $\alpha \le \alpha^*$ and $\alpha^* \le \alpha$, we definitively conclude $\alpha = \alpha^*$.
\end{proof}

Although $g$ is written as a maximization problem, its maximizers are exactly the minimizers of the sign-reversed difference objective $\|\bm{x}\|_p^p-\alpha\|\bm{x}\|_q^p$ on $\mathcal{F}$. Solving this global problem is generally intractable for the present nonconvex model. Therefore, at the $k$-th outer iteration, we use a linearized-proximal approximation of this linearly constrained minimization problem. Since the present subproblem is highly nonconvex, we keep the nonsmooth numerator in proximal form and linearize only $s(\bm{x})=\|\bm{x}\|_q^p$ at the current point. A proximal term $\frac{\beta}{2}\|\bm{x}-\bm{x}^{(k)}\|_2^2$ is added to obtain descent control and boundedness of the outer scheme. The resulting linearized proximal subproblem is
\begin{equation}\label{eq:outer_subproblem}
	\bm{x}^{(k+1)}= \arg\min_{\bm{x}\in\mathbb{R}^n}\;\left\{h(\bm{x})-\langle \bm{c}^{(k)},\bm{x}\rangle+\frac{\beta}{2}\|\bm{x}-\bm{x}^{(k)}\|_2^2\right\},
\end{equation}
where $h(\bm{x})=\|\bm{x}\|_p^p+I_{\mathcal{F}}(\bm{x})$, $I_{\mathcal{F}}$ is the indicator function of $\mathcal{F}$, and the linearization coefficient $\bm{c}^{(k)}$ is given by
$$
	\bm{c}^{(k)} = \alpha^{(k)} \cdot p \|\bm{x}^{(k)}\|_q^{p-q} \operatorname{sign}(\bm{x}^{(k)}) \odot |\bm{x}^{(k)}|^{q-1}.
$$

For numerical implementation of the linearization proximal problem~\eqref{eq:outer_subproblem}, we introduce an auxiliary variable $\bm{y} = \bm{x}$ and construct the augmented Lagrangian function
\begin{equation}
	\mathcal{L}_{\rho}(\bm{x}, \bm{y}, \bm{u}) = \|\bm{y}\|_p^p + I_{\mathcal{F}}(\bm{x}) - \langle \bm{c}^{(k)}, \bm{x} \rangle + \frac{\beta}{2}\|\bm{x} - \bm{x}^{(k)}\|_2^2 + \frac{\rho}{2}\left\|\bm{x} - \bm{y} + \frac{\bm{u}}{\rho}\right\|_2^2,
	\label{eq:aug_lagrangian}
\end{equation}
where $\bm{u}$ is the Lagrange multiplier, $\rho$ is the penalty parameter, and $\beta$ is the proximal parameter.

In the $l$-th inner step, the $\bm{x}$-update admits a closed-form solution via orthogonal affine projection, that is
\begin{equation*}
	\bm{x}_{\text{in}}^{(l+1)} = \boldsymbol{\psi}^{(l)} - \bm{P}_A(\bm{A}\boldsymbol{\psi}^{(l)} - \bm{b}), \quad \text{with} \;\; \boldsymbol{\psi}^{(l)} = \frac{\rho\bm{y}^{(l)} + \beta\bm{x}^{(k)} - \bm{u}^{(l)} + \bm{c}^{(k)}}{\rho + \beta},
	\label{eq:x_update}
\end{equation*}
where $\bm{P}_A = \bm{A}^T(\bm{A}\bm{A}^T)^{\dagger}$ is the pre-computed pseudo-inverse projection operator, reducing the computational complexity from $\mathcal{O}(n^3)$ matrix inversions to $\mathcal{O}(mn)$ matrix-vector multiplications. 

The $\bm{y}$-subproblem is separable and corresponds to the proximal subproblem of the $\ell_p^p$ penalty, given by
\begin{equation*}
	\bm{y}^{(l+1)} = \arg\min_{\bm{y}} \|\bm{y}\|_p^p + \frac{\rho}{2}\left\|\bm{y} - \left(\bm{x}_{\text{in}}^{(l+1)} + \frac{\bm{u}^{(l)}}{\rho}\right)\right\|_2^2.
	\label{eq:y_proximal_def}
\end{equation*}
For $p = 1$,~\eqref{eq:y_proximal_def} reduces to the standard soft-thresholding operator. For $0 < p < 1$, the coordinate update is evaluated by the Generalized Soft-Thresholding (GST) formula $\mathcal{T}_p^{\rho}(\cdot)$ \cite{zuo2013generalized}. Denoting $t = x_{\text{in}, i}^{(l+1)} + u_i^{(l)} / \rho$, The operator $\mathcal{T}_p^{\rho}$ is defined by $\mathcal{T}_p^{\rho}(t)=0$ for $|t|\le \tau$ and $\mathcal{T}_p^{\rho}(t)=\operatorname{sign}(t)h^{-1}(|t|)$ for $|t|>\tau$, where $h(z) = p z^{p-1} / \rho + z$, and the threshold $\tau$ is analytically determined by $\tau = \beta_p + p \beta_p^{p-1} / \rho$ with $\beta_p = [2(1-p)/\rho]^{1/(2-p)}$. In the numerical implementation, the positive branch of $h^{-1}$ is evaluated by safeguarded Newton iteration, and a continuation strategy on $\rho$ is used to improve the stability of the inner ADMM iteration. 

The dual variable is updated by $\bm{u}^{(l+1)} = \bm{u}^{(l)} + \rho(\bm{x}_{\mathrm{in}}^{(l+1)} - \bm{y}^{(l+1)})$. After the $k$-th linearized subproblem update is computed, we evaluate $f_1^{(k)}=\|\bm{x}^{(k+1)}\|_p^p$, $f_2^{(k)}=\|\bm{x}^{(k+1)}\|_q^p$, and $\Delta_k=\alpha^{(k)}f_2^{(k)}-f_1^{(k)}$. The outer scheme analyzed below is
\begin{equation}\label{eq:standard_dinkelbach_outer_loop}
	\left\{
	\begin{aligned}
		&\bm{x}^{(k+1)}
		\in \arg\min_{\bm{x}\in\mathbb{R}^n}
		\left\{
		h(\bm{x})-\langle \bm{c}^{(k)},\bm{x}\rangle
		+\frac{\beta}{2}\|\bm{x}-\bm{x}^{(k)}\|_2^2
		\right\},\\
		&\Delta_k:=\alpha^{(k)}\|\bm{x}^{(k+1)}\|_q^p-\|\bm{x}^{(k+1)}\|_p^p,\\
		&\alpha^{(k+1)}
		:=\frac{\|\bm{x}^{(k+1)}\|_p^p}{\|\bm{x}^{(k+1)}\|_q^p}
		=\alpha^{(k)}-\frac{\Delta_k}{\|\bm{x}^{(k+1)}\|_q^p}.
	\end{aligned}
	\right.
\end{equation}
The Dinkelbach linearization proximal algorithm (DLPA) is summarized in Algorithm~\ref{alg:standard_dinkelbach}.
\begin{algorithm}[t]
	\caption{Prox-Linear Dinkelbach Method for $\mathcal{Q}_{p,q}$-Minimization}
	\label{alg:standard_dinkelbach}
	\LinesNumbered
	\KwIn{$0<p\le1$, $q>1$, sensing matrix $\bm{A}\in\mathbb{R}^{m\times n}$, observation vector $\bm{b}\in\mathbb{R}^m$, initial feasible point $\bm{x}^{(0)}\in\mathcal{F}$, proximal parameter $\beta>0$, maximum iteration number $K_{\max}$, and tolerance $\varepsilon>0$}
	\KwOut{Recovered sparse solution $\bm{x}$ and ratio value $\alpha$}
	
	Set $\alpha^{(0)}\gets \|\bm{x}^{(0)}\|_p^p/\|\bm{x}^{(0)}\|_q^p$\;
	\For{$k=0,1,\ldots,K_{\max}-1$}{
		\tcp{Linearization of the denominator}
		$\bm{c}^{(k)}\gets \alpha^{(k)}p\|\bm{x}^{(k)}\|_q^{p-q}\operatorname{sign}(\bm{x}^{(k)})\odot|\bm{x}^{(k)}|^{q-1}$\;
		\tcp{Prox-linear update}
		Compute
		$
			\bm{x}^{(k+1)}\in\arg\min_{\bm{x}\in\mathbb{R}^n}
			\{h(\bm{x})-\langle \bm{c}^{(k)},\bm{x}\rangle +\frac{\beta}{2}\|\bm{x}-\bm{x}^{(k)}\|_2^2\}
		$\;
		\tcp{Ratio update}
		$\Delta_k\gets \alpha^{(k)}\|\bm{x}^{(k+1)}\|_q^p-\|\bm{x}^{(k+1)}\|_p^p$\;
		$\alpha^{(k+1)}\gets \|\bm{x}^{(k+1)}\|_p^p/\|\bm{x}^{(k+1)}\|_q^p$\;
		\lIf{$|\Delta_k|<\varepsilon$ or $\|\bm{x}^{(k+1)}-\bm{x}^{(k)}\|_2/\|\bm{x}^{(k)}\|_2<\varepsilon$}{\textbf{break}}
	}
	
	\Return $\bm{x}^{(k+1)}$ and $\alpha^{(k+1)}$
\end{algorithm}

\subsection{Convergence analysis}

In this subsection, we establish a convergence framework for the standard prox-linear Dinkelbach scheme underlying Algorithm~\ref{alg:standard_dinkelbach}. The main difficulty is that the present model combines two nonconvex components: the numerator $\|\bm{x}\|_p^p$ is nonconvex when $0<p<1$, and the term $-\alpha\|\bm{x}\|_q^p$ must be linearized at every outer iteration. Therefore, unlike the $\ell_1/\ell_2$ analysis in \cite{wang2020accelerated}, one cannot directly reuse a convex-proximal-gradient argument. Let us begin by introducing the following two definitions.
\begin{definition}[Regular and limiting subdifferentials]
	Let $f: \mathbb{R}^n \to (-\infty, +\infty]$ be a proper closed function. 
	\begin{itemize}
		\item The regular (Fr\'echet) subdifferential of $f$ at $\bm{x}$ is
		\begin{equation*}
			\hat{\partial} f(\bm{x}) := \left\{ \bm{v} \in \mathbb{R}^n \;\Big|\; \liminf_{\bm{y} \to \bm{x},\, \bm{y} \neq \bm{x}} \frac{f(\bm{y}) - f(\bm{x}) - \langle \bm{v}, \bm{y} - \bm{x} \rangle}{\|\bm{y} - \bm{x}\|} \ge 0 \right\}.
		\end{equation*}
		\item The limiting (Mordukhovich) subdifferential of $f$ at $\bm{x} \in \mathrm{dom}\, f$ is
		\begin{equation*}
			\partial f(\bm{x}) := \left\{ \bm{v} \in \mathbb{R}^n \;\Big|\; \exists \bm{x}^k \to \bm{x},\, f(\bm{x}^k) \to f(\bm{x}),\, \bm{v}^k \in \hat{\partial} f(\bm{x}^k),\, \bm{v}^k \to \bm{v} \right\}.
		\end{equation*}
	\end{itemize}
\end{definition}

Let $s(\bm{x}):=\|\bm{x}\|_q^p$.  We work under the following conditions throughout the convergence analysis.
\begin{assumption}[]\label{assump:standing}
	We make the following assumptions:
	\begin{enumerate}[{\rm(i)}]
		\item $\mathcal{F} \neq \emptyset$, $0 < p \le 1$, $q > 1$, $\bm{b} \neq \bm{0}$, and $\bm{A}$ has full row rank.
		\item The outer sequence $\{\bm{x}^{(k)}\}$ generated by \eqref{eq:standard_dinkelbach_outer_loop} is bounded, that is, there exists a positive $\bar{\nu}$ such that $\|\bm{x}^{(k)}\|_2\leq \bar{\nu}$ for any $k\geq 0$.
		\item There exist a compact set $\Omega \subset \mathbb{R}^n$ and an open convex set $\mathcal{U}\subset\mathbb{R}^n$ such that
		$\Omega\subset\mathcal{U}$, $\bm{x}^{(k)} \in \Omega \cap \mathcal{F}$ for all $k$, and $\nabla s$ is $L_s$-Lipschitz continuous on $\mathcal{U}$, namely $\|\nabla s(\bm{x}) - \nabla s(\bm{y})\|_2 \le L_s \|\bm{x} - \bm{y}\|_2$ for any $\bm{x},\bm{y} \in \mathcal{U}$.
		\item The proximal parameter satisfies $\beta > \overline{\alpha} L_s$ and $\overline{\alpha} := n^{1 - p/q}$.
	\end{enumerate}
\end{assumption}

\begin{remark}
	The Lipschitz-gradient assumption in Assumption \ref{assump:standing}(iii) is automatic when $q\ge 2$, provided the iterates are bounded. Indeed,
	in that case $s(\bm{x})=\|\bm{x}\|_q^p$ is $C^1$ on
	$\mathbb{R}^n\setminus\{\bm{0}\}$. Since the feasible affine set is uniformly
	separated from the origin by Lemma~\ref{lem:feasible_denominator_bound} below, the compact convex hull of the bounded feasible iterates admits a small open convex neighborhood that is still separated from the origin, and $\nabla s$ is Lipschitz there. When $1<q<2$, this condition is not automatic near zero coordinates; in that case Assumption~\ref{assump:standing}(iii) should be read as requiring the open convex neighborhood visited by the analysis to avoid the coordinate singularities of $\nabla s$.
\end{remark}

\begin{lemma}[]\label{lem:feasible_denominator_bound}
	Let $\bm{P}_A:=\bm{A}^T(\bm{A}\bm{A}^T)^\dagger$ and $\bm{v}:=\bm{P}_A\bm{b}$. Then $\bm{v}\neq \bm{0}$, and there exists a constant $\underline{\nu}>0$ such that $\|\bm{x}\|_q^p\ge \underline{\nu}$ for any $\bm{x}\in\mathcal{F}$. Moreover, there exists $\overline{\nu}<\infty$ such that $\|\bm{x}\|_q^p\le \overline{\nu}$ for every $\bm{x}\in \Omega\cap \mathcal{F}$.
\end{lemma}

\begin{proof}
	Since $\bm{b}\neq \bm{0}$ and $\mathcal{F}\neq\emptyset$, the minimum-norm feasible point is exactly $\bm{v}=\bm{P}_A\bm{b}$, hence $\bm{v}\neq \bm{0}$. Moreover, every $\bm{x}\in\mathcal{F}$ admits the orthogonal decomposition $\bm{x}=\bm{v}+\bm{z}$, where $\bm{z}\in \ker(\bm{A})$ and $\bm{v}\perp \ker(\bm{A})$. It follows that $\|\bm{x}\|_2^2=\|\bm{v}\|_2^2+\|\bm{z}\|_2^2\ge \|\bm{v}\|_2^2$. By norm equivalence in $\mathbb{R}^n$, there exists $C_{q,2}>0$ such that $\|\bm{x}\|_q\ge C_{q,2}\|\bm{x}\|_2\ge C_{q,2}\|\bm{v}\|_2$.
	Therefore, $\|\bm{x}\|_q^p\ge \left(C_{q,2}\|\bm{v}\|_2\right)^p =:\underline{\nu}>0$ for every $\bm{x}\in\mathcal{F}$. The upper bound follows from compactness of $\Omega$ and continuity of $\bm{x}\mapsto \|\bm{x}\|_q^p$.
\end{proof}

\begin{lemma}[]\label{lem:ratio_interval}
	For every nonzero $\bm{x}\in\mathbb{R}^n$,
	$1\le \frac{\|\bm{x}\|_p^p}{\|\bm{x}\|_q^p}\le n^{1-p/q}$.	Consequently, for every outer iterate,
	$1\le \alpha^{(k)}\le \overline{\alpha}:=n^{1-p/q}$.
\end{lemma}

\begin{proof}
	Since $p\le q$, the norm inequalities imply $\|\bm{x}\|_q\le \|\bm{x}\|_p\le n^{1/p-1/q}\|\bm{x}\|_q$. Raising both sides to the power $p$ gives $\|\bm{x}\|_q^p\le \|\bm{x}\|_p^p
	\le n^{1-p/q}\|\bm{x}\|_q^p$, which is equivalent to the stated bound.
\end{proof}

\begin{lemma}[]
	\label{lem:g_monotone_convex}
	Let $g(\alpha)$ be defined by \eqref{eq:g_alpha_def}. Then on every interval where $g$ is finite, it is strictly increasing and convex.
\end{lemma}

\begin{proof}
	Fix $\alpha_1<\alpha_2$. For every $\bm{x}\in\mathcal{F}$, we have $\alpha_2\|\bm{x}\|_q^p-\|\bm{x}\|_p^p
	=\alpha_1\|\bm{x}\|_q^p-\|\bm{x}\|_p^p
	+(\alpha_2-\alpha_1)\|\bm{x}\|_q^p$. Using Lemma~\ref{lem:feasible_denominator_bound},  $\alpha_2\|\bm{x}\|_q^p-\|\bm{x}\|_p^p \ge \alpha_1\|\bm{x}\|_q^p-\|\bm{x}\|_p^p +(\alpha_2-\alpha_1)\underline{\nu}$. Taking the supremum over $\mathcal{F}$ yields $g(\alpha_2)\ge g(\alpha_1)+(\alpha_2-\alpha_1)\underline{\nu}>g(\alpha_1)$, which proves strict monotonicity. For convexity, note that for each fixed $\bm{x}\in\mathcal{F}$, the mapping $\alpha\mapsto \alpha\|\bm{x}\|_q^p-\|\bm{x}\|_p^p$ is affine. Since $g$ is the supremum of affine functions, it is convex.
\end{proof}

\begin{lemma}[]\label{lem:quadratic_majorization}
	Define $\Phi_\alpha(\bm{x}):= \|\bm{x}\|_p^p-\alpha\|\bm{x}\|_q^p+I(\bm{A}\bm{x}-\bm{b})$, and for each $k$ define the full linearized model
$
		\widetilde{\Phi}_k(\bm{x}):= h(\bm{x}) -\alpha^{(k)} \Bigl(\|\bm{x}^{(k)}\|_q^p +\langle \nabla s(\bm{x}^{(k)}),\bm{x}-\bm{x}^{(k)}\rangle\Bigr) +\frac{\beta}{2}\|\bm{x}-\bm{x}^{(k)}\|_2^2.
$
	Then $\widetilde{\Phi}_k$ and \eqref{eq:outer_subproblem} have the same minimizers.
	Moreover, for every $\bm{x}\in \Omega\cap\mathcal{F}$, one has
	\begin{equation}\label{eq:majorization_exact}
		\widetilde{\Phi}_k(\bm{x})
		\ge
		\Phi_{\alpha^{(k)}}(\bm{x})
		+\frac{\beta-\alpha^{(k)}L_s}{2}\|\bm{x}-\bm{x}^{(k)}\|_2^2.
	\end{equation}
\end{lemma}

\begin{proof}
	The first claim is immediate, because \eqref{eq:outer_subproblem} differs from
	$\widetilde{\Phi}_k$ only by a constant independent of $\bm{x}$.
	
	To prove the majorization inequality~\eqref{eq:majorization_exact}, we directly evaluate the difference $\widetilde{\Phi}_k(\bm{x}) - \Phi_{\alpha^{(k)}}(\bm{x})$. By substituting the definition of $\Phi_{\alpha^{(k)}}(\bm{x}) = h(\bm{x}) - \alpha^{(k)}s(\bm{x})$ and the definition of $\widetilde{\Phi}_k(\bm{x})$ into $\widetilde{\Phi}_k(\bm{x}) - \Phi_{\alpha^{(k)}}(\bm{x})$, the common term $h(\bm{x})$ cancels out, yielding
	\begin{equation}
		\widetilde{\Phi}_k(\bm{x}) - \Phi_{\alpha^{(k)}}(\bm{x}) = \alpha^{(k)} \left[ s(\bm{x}) - \Bigl(\|\bm{x}^{(k)}\|_q^p +\langle \nabla s(\bm{x}^{(k)}),\bm{x}-\bm{x}^{(k)}\rangle\Bigr) \right] + \frac{\beta}{2}\|\bm{x}-\bm{x}^{(k)}\|_2^2.
		\label{eq:difference_expanded}
	\end{equation}
	Because both $\bm{x}$ and $\bm{x}^{(k)}$ belong to $\Omega\subset\mathcal{U}$ and $\mathcal{U}$ is convex, the segment joining them is contained in $\mathcal{U}$. Since $\nabla s$ is $L_s$-Lipschitz continuous on $\mathcal{U}$, the descent lemma gives
$
		s(\bm{x}) \ge \|\bm{x}^{(k)}\|_q^p + \langle \nabla s(\bm{x}^{(k)}),\bm{x}-\bm{x}^{(k)}\rangle - \frac{L_s}{2}\|\bm{x}-\bm{x}^{(k)}\|_2^2.
$
	Rearranging this inequality, we have
	\begin{equation}\label{eq:lower_bound_descent}
		s(\bm{x}) - \Bigl(\|\bm{x}^{(k)}\|_q^p +\langle \nabla s(\bm{x}^{(k)}),\bm{x}-\bm{x}^{(k)}\rangle\Bigr) \ge - \frac{L_s}{2}\|\bm{x}-\bm{x}^{(k)}\|_2^2.
	\end{equation}
	Since $\alpha^{(k)} > 0$, we can substitute the lower bound~\eqref{eq:lower_bound_descent} into the bracketed term in~\eqref{eq:difference_expanded} to obtain
	\begin{equation*}
		\widetilde{\Phi}_k(\bm{x}) - \Phi_{\alpha^{(k)}}(\bm{x}) \ge \alpha^{(k)} \left( - \frac{L_s}{2}\|\bm{x}-\bm{x}^{(k)}\|_2^2 \right) + \frac{\beta}{2}\|\bm{x}-\bm{x}^{(k)}\|_2^2 
		= \frac{\beta - \alpha^{(k)} L_s}{2} \|\bm{x}-\bm{x}^{(k)}\|_2^2.
	\end{equation*}
	Moving $\Phi_{\alpha^{(k)}}(\bm{x})$ to the right-hand side of the inequality produces
$
		\widetilde{\Phi}_k(\bm{x}) \ge \Phi_{\alpha^{(k)}}(\bm{x}) + \frac{\beta-\alpha^{(k)}L_s}{2}\|\bm{x}-\bm{x}^{(k)}\|_2^2.
$
	This completes the proof.
\end{proof}

\begin{theorem}[]\label{thm:outer_sufficient_decrease}
	Let $\{(\bm{x}^{(k)},\alpha^{(k)})\}$ be generated by \eqref{eq:standard_dinkelbach_outer_loop}. Suppose that Assumption~\ref{assump:standing} holds. Then, for every $k\ge 0$,
	\begin{equation}\label{eq:delta_lower_bound}
		\Delta_k \ge \frac{\beta-\alpha^{(k)}L_s}{2}\|\bm{x}^{(k+1)}-\bm{x}^{(k)}\|_2^2 \ge \frac{\beta-\overline{\alpha}L_s}{2}\|\bm{x}^{(k+1)}-\bm{x}^{(k)}\|_2^2 \ge 0.
	\end{equation}
	Consequently,
	\begin{equation}\label{eq:alpha_decrease_final}
		\alpha^{(k)}-\alpha^{(k+1)} = \frac{\Delta_k}{\|\bm{x}^{(k+1)}\|_q^p} \ge \frac{\beta-\overline{\alpha}L_s}{2\overline{\nu}} \|\bm{x}^{(k+1)}-\bm{x}^{(k)}\|_2^2.
	\end{equation}
	Hence $\{\alpha^{(k)}\}$ is monotonically decreasing and convergent, and $\|\bm{x}^{(k+1)}-\bm{x}^{(k)}\|_2\to 0$.
\end{theorem}

\begin{proof}
	Since $\bm{x}^{(k+1)}$ minimizes \eqref{eq:outer_subproblem}, equivalently it
	minimizes $\widetilde{\Phi}_k$. Because
	$\alpha^{(k)}=\|\bm{x}^{(k)}\|_p^p/\|\bm{x}^{(k)}\|_q^p$, we have
$
		\widetilde{\Phi}_k(\bm{x}^{(k)}) =\Phi_{\alpha^{(k)}}(\bm{x}^{(k)}) =\|\bm{x}^{(k)}\|_p^p-\alpha^{(k)}\|\bm{x}^{(k)}\|_q^p=0.
$
	Therefore, $\widetilde{\Phi}_k(\bm{x}^{(k+1)})\le \widetilde{\Phi}_k(\bm{x}^{(k)})=0$. Using Lemma~\ref{lem:quadratic_majorization} at $\bm{x}=\bm{x}^{(k+1)}$, we obtain
	\begin{equation*}
		0\ge \widetilde{\Phi}_k(\bm{x}^{(k+1)}) \ge \Phi_{\alpha^{(k)}}(\bm{x}^{(k+1)}) +\frac{\beta-\alpha^{(k)}L_s}{2} \|\bm{x}^{(k+1)}-\bm{x}^{(k)}\|_2^2.
	\end{equation*}
	Expanding $\Phi_{\alpha^{(k)}}(\bm{x}^{(k+1)}) = \|\bm{x}^{(k+1)}\|_p^p - \alpha^{(k)}\|\bm{x}^{(k+1)}\|_q^p$ and rearranging the terms by moving the objective evaluation to the left-hand side yields
$
		\Delta_k=\alpha^{(k)}\|\bm{x}^{(k+1)}\|_q^p-\|\bm{x}^{(k+1)}\|_p^p \ge \frac{\beta-\alpha^{(k)}L_s}{2} \|\bm{x}^{(k+1)}-\bm{x}^{(k)}\|_2^2.
$
	Using $\alpha^{(k)}\le\overline{\alpha}$ from Lemma~\ref{lem:ratio_interval} and $\beta>\overline{\alpha}L_s$, this proves \eqref{eq:delta_lower_bound}. Moreover, Lemma~\ref{lem:feasible_denominator_bound} gives $\|\bm{x}^{(k+1)}\|_q^p\le \overline{\nu}$, and therefore
$
		\alpha^{(k)} - \alpha^{(k+1)}
		=\frac{\Delta_k}{\|\bm{x}^{(k+1)}\|_q^p}
		\ge \frac{\Delta_k}{\overline{\nu}}
		\ge \frac{\beta-\overline{\alpha}L_s}{2\overline{\nu}}\|\bm{x}^{(k+1)}-\bm{x}^{(k)}\|_2^2,
$
	which is \eqref{eq:alpha_decrease_final}.
	
	Lemma~\ref{lem:ratio_interval} shows that $\{\alpha^{(k)}\}$ is bounded below
	by $1$, while \eqref{eq:alpha_decrease_final} shows it is decreasing. Hence
	$\{\alpha^{(k)}\}$ converges. Summing the sufficient decrease inequality~\eqref{eq:alpha_decrease_final} from $k=0$ to $\infty$ yields
$
		\sum_{k=0}^{\infty}\|\bm{x}^{(k+1)}-\bm{x}^{(k)}\|_2^2 \le \frac{2\overline{\nu}}{\beta-\overline{\alpha}L_s} (\alpha^{(0)} - \alpha^{(\infty)}) < \infty.
$
	Hence concluding that $\|\bm{x}^{(k+1)}-\bm{x}^{(k)}\|_2\to 0$.
\end{proof}

\begin{theorem}[]\label{thm:cluster_critical_final}
	Let $\{(\bm{x}^{(k)},\alpha^{(k)})\}$ be generated by
	\eqref{eq:standard_dinkelbach_outer_loop}. Suppose that Assumption~\ref{assump:standing} holds. Then every cluster point $\bar{\bm{x}}$ of
	$\{\bm{x}^{(k)}\}$ satisfies
	\begin{equation}\label{eq:critical_stationary_final}
		\bm{0}
		\in
		\partial h(\bar{\bm{x}})
		-\bar{\alpha}\nabla s(\bar{\bm{x}}),
		\qquad
		\bar{\alpha}:=
		\frac{\|\bar{\bm{x}}\|_p^p}{\|\bar{\bm{x}}\|_q^p}.
	\end{equation}
	Equivalently,
	\begin{equation}\label{eq:critical_ratio_final}
		\bm{0}
		\in
		\|\bar{\bm{x}}\|_q^p\,\partial h(\bar{\bm{x}})
		-\|\bar{\bm{x}}\|_p^p\,\nabla s(\bar{\bm{x}}).
	\end{equation}
	Hence every cluster point is a limiting critical point of the original constrained
	ratio model \eqref{eq:ratio_model}.
\end{theorem}

\begin{proof}
	The optimality condition of \eqref{eq:outer_subproblem} gives
$
		\bm{0}\in\partial h(\bm{x}^{(k+1)})-\bm{c}^{(k)} +\beta(\bm{x}^{(k+1)}-\bm{x}^{(k)}).
$
	Since $\bm{c}^{(k)}=\alpha^{(k)}\nabla s(\bm{x}^{(k)})$, there exists
	$\boldsymbol{\xi}^{(k+1)}\in \partial h(\bm{x}^{(k+1)})$ such that
$
		\boldsymbol{\xi}^{(k+1)}-\alpha^{(k)}\nabla s(\bm{x}^{(k)})+\beta(\bm{x}^{(k+1)}-\bm{x}^{(k)})=\bm{0}.
$
	Define $\bm{r}^{(k+1)}:=\boldsymbol{\xi}^{(k+1)}-\alpha^{(k)}\nabla s(\bm{x}^{(k+1)})\in\partial h(\bm{x}^{(k+1)})-\alpha^{(k)}\nabla s(\bm{x}^{(k+1)})$. Then
	\begin{equation*}
		\|\bm{r}^{(k+1)}\|_2\le\beta\|\bm{x}^{(k+1)}-\bm{x}^{(k)}\|_2+\alpha^{(k)}\|\nabla s(\bm{x}^{(k+1)})-\nabla s(\bm{x}^{(k)})\|_2.
	\end{equation*}
	By Lemma~\ref{lem:ratio_interval} and the Lipschitz continuity of $\nabla s$, one has
$
		\|\bm{r}^{(k+1)}\|_2 \le (\beta+\overline{\alpha}L_s) \|\bm{x}^{(k+1)}-\bm{x}^{(k)}\|_2.
$
	It follows from $\bm{r}^{(k+1)} \in \partial h(\bm{x}^{(k+1)}) - \alpha^{(k)} \nabla s(\bm{x}^{(k+1)})$ that
$
	\operatorname{dist}\!\left(
	\bm{0},
	\partial h(\bm{x}^{(k+1)}) - \alpha^{(k)} \nabla s(\bm{x}^{(k+1)})
	\right)
	\le \|\bm{r}^{(k+1)}\|.
$
	Theorem~\ref{thm:outer_sufficient_decrease} implies
	$\|\bm{x}^{(k+1)}-\bm{x}^{(k)}\|_2\to 0$, it follows that
$$
		\operatorname{dist}\!\left(\bm{0},\partial h(\bm{x}^{(k+1)})-\alpha^{(k)}\nabla s(\bm{x}^{(k+1)})\right)\to 0.
$$
	Let $\bar{\bm{x}}$ be a cluster point of $\{\bm{x}^{(k)}\}$. Passing to a subsequence if necessary, we may assume $\bm{x}^{(k_j)}\to \bar{\bm{x}}$. Since
	$\|\bm{x}^{(k+1)}-\bm{x}^{(k)}\|_2\to 0$, we also have $\bm{x}^{(k_j+1)}\to \bar{\bm{x}}$. Because $\{\alpha^{(k)}\}$ converges by Theorem~\ref{thm:outer_sufficient_decrease}, there exists $\bar{\alpha}$ such that $\alpha^{(k_j)}\to \bar{\alpha}$. Moreover, from the optimality condition, one has
$
	\boldsymbol{\xi}^{(k_j+1)}
	=
	\alpha^{(k_j)}\nabla s(\bm{x}^{(k_j)})
	-\beta(\bm{x}^{(k_j+1)}-\bm{x}^{(k_j)})
	\to
	\bar{\alpha}\nabla s(\bar{\bm{x}}).
$
	Since $\boldsymbol{\xi}^{(k_j+1)}\in\partial h(\bm{x}^{(k_j+1)})$, $\bm{x}^{(k_j+1)}\to\bar{\bm{x}}$, and $h(\bm{x}^{(k_j+1)})=\|\bm{x}^{(k_j+1)}\|_p^p\to\|\bar{\bm{x}}\|_p^p=h(\bar{\bm{x}})$, the closedness of the limiting-subdifferential graph yields $\bar{\alpha}\nabla s(\bar{\bm{x}})\in\partial h(\bar{\bm{x}})$, that is, $\bm{0}\in \partial h(\bar{\bm{x}})-\bar{\alpha}\nabla s(\bar{\bm{x}})$. Lemma~\ref{lem:feasible_denominator_bound} ensures that $\|\bm{x}\|_q^p$ is uniformly bounded away from zero on $\mathcal{F}$. Consequently, the ratio map is continuous, which implies that 
	\begin{equation*}
		\bar{\alpha}=\lim_{j\to\infty}\alpha^{(k_j+1)}=\lim_{j\to\infty}\frac{\|\bm{x}^{(k_j+1)}\|_p^p}{\|\bm{x}^{(k_j+1)}\|_q^p}=\frac{\|\bar{\bm{x}}\|_p^p}{\|\bar{\bm{x}}\|_q^p}.
	\end{equation*}
	This proves \eqref{eq:critical_stationary_final}. Multiplying
	\eqref{eq:critical_stationary_final} by $\|\bar{\bm{x}}\|_q^p>0$ gives
	\eqref{eq:critical_ratio_final}.
\end{proof}

\section{Numerical Experiments}\label{sec:NumericalExperiments}

This section evaluates the proposed PLDM algorithm for solving the $\mathcal{Q}_{p,q}$-minimization problem in sparse signal recovery. All experiments were performed in MATLAB R2025b on a desktop computer equipped with an Intel Core i9-12900H processor at 2.50 GHz and 16 GB RAM. We consider the noiseless compressed sensing model $\bm{b}=\bm{A}\bm{x}^\star$, where $\bm{A}\in\mathbb{R}^{m\times n}$, $m=64$, and $n=1024$. The sparsity level is varied over $k\in\{2,4,\ldots,30\}$. For each trial, the support of the ground-truth signal $\bm{x}^\star$ is sampled uniformly at random, subject to the separation constraint described below when the oversampled DCT matrix is used. To test recovery under high dynamic range, the nonzero entries are generated with random signs and log-uniform magnitudes, using the interval $[1,10^3]$ for the correlated Gaussian matrices and $[1,10^5]$ for the oversampled DCT matrices.

Two sensing matrices are used. The first is a correlated Gaussian matrix, where the rows of $\bm{A}$ are independently sampled from $\mathcal{N}(\bm{0},\Sigma_r)$ with $\Sigma_r=(1-r)\bm{I}+r\bm{1}\bm{1}^{\top}$, and $r\in\{0.3,0.9\}$ controls the correlation level. The second is the oversampled DCT matrix $A_{ij}=\frac{1}{\sqrt{m}}\cos(\frac{2\pi w_i j}{F})$ for every $i=1,\ldots,m$ and $j=1,\ldots,n$, where $w_i\overset{\mathrm{i.i.d.}}{\sim}\mathcal{U}([0,1])$. Larger $F$ gives a more coherent sensing matrix. Following the common protocol for highly coherent dictionaries \cite{fannjiang2012coherence}, the support of $\bm{x}^\star$ in the DCT experiments is sampled with minimum separation $2F$. We report results for $F=10$ and $F=15$.

A run is declared successful if $\frac{\|\hat{\bm{x}}-\bm{x}^\star\|_2}{\|\bm{x}^\star\|_2}<10^{-3}$. For unsuccessful trials, we further distinguish between two failure modes. If the objective value at the returned point $\hat{\bm{x}}$ is smaller than that at $\bm{x}^\star$, the trial is counted as a model failure; in this case, the planted sparse vector is not favored by the model for that instance. Otherwise, the trial is counted as an algorithm failure  ; this indicates that the model may still be appropriate, but the numerical solver failed to reach a successful feasible point. We also report the reconstruction signal-to-noise ratio
$
    \operatorname{SNR}=20\log_{10}
    \left(\frac{\|\bm{x}^\star\|_2}{\|\hat{\bm{x}}-\bm{x}^\star\|_2}\right).
$
Unless otherwise stated, the constrained $\ell_1$ solution computed by Gurobi is used as the initialization. If this initialization fails, the minimum-norm feasible point $\bm{v}=\bm{A}^T(\bm{A}\bm{A}^T)^\dagger\bm{b}$ is used as a fallback.

\subsection{Parameter Sensitivity of $\mathcal{Q}_{p,q}$}

We first study how the proposed model depends on the parameter pair $(p,q)$. The parameter grid is given by $p\in\{0.1,0.3,0.5,0.7,0.9,1.0\}$ and $q\in\{1.2,1.4,\ldots,3.0\}$. For each triple $(p,q,k)$, we conduct $50$ Monte Carlo trials and average the results over trials and sparsity levels when drawing the heat maps.
\begin{figure}[htbp]
	\centering
	\begin{subfigure}[b]{0.96\textwidth}
		\centering
		\includegraphics[scale=0.2]{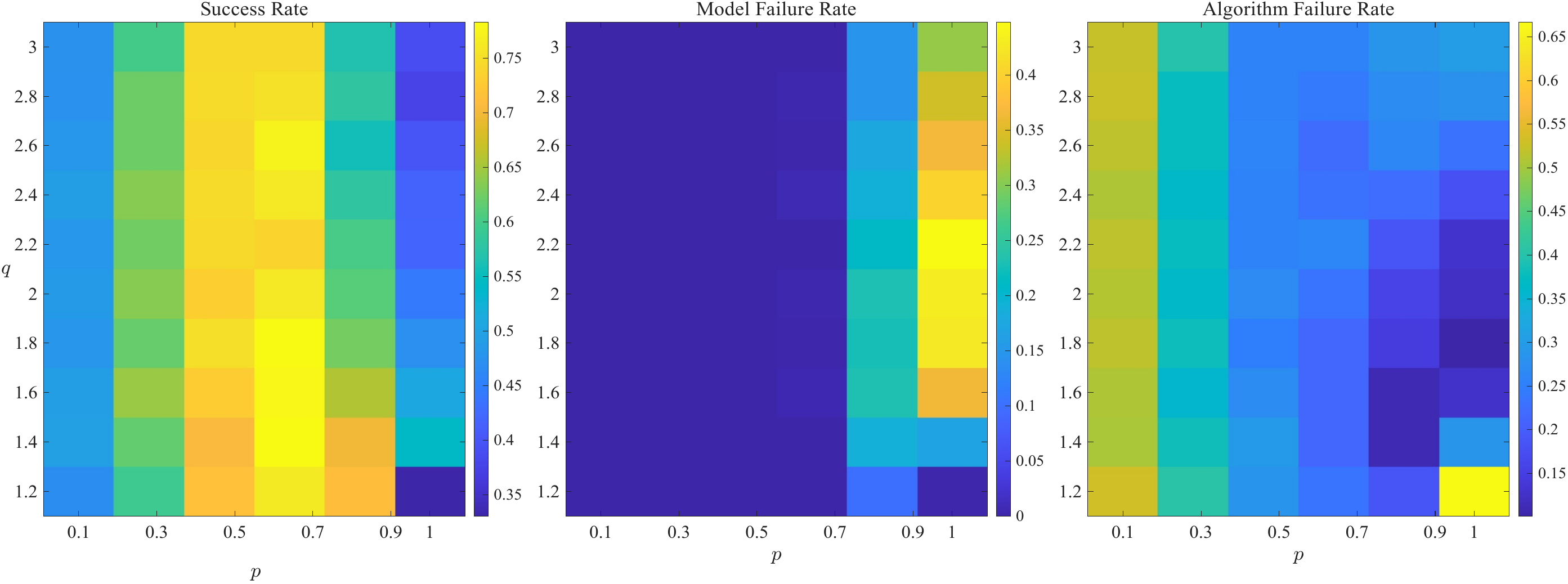}
		\caption{Gaussian matrix: $r=0.3$.}
		\label{fig:gauss-r03-heatmap}
	\end{subfigure}
	\begin{subfigure}[b]{0.96\textwidth}
		\centering
		\includegraphics[scale=0.2]{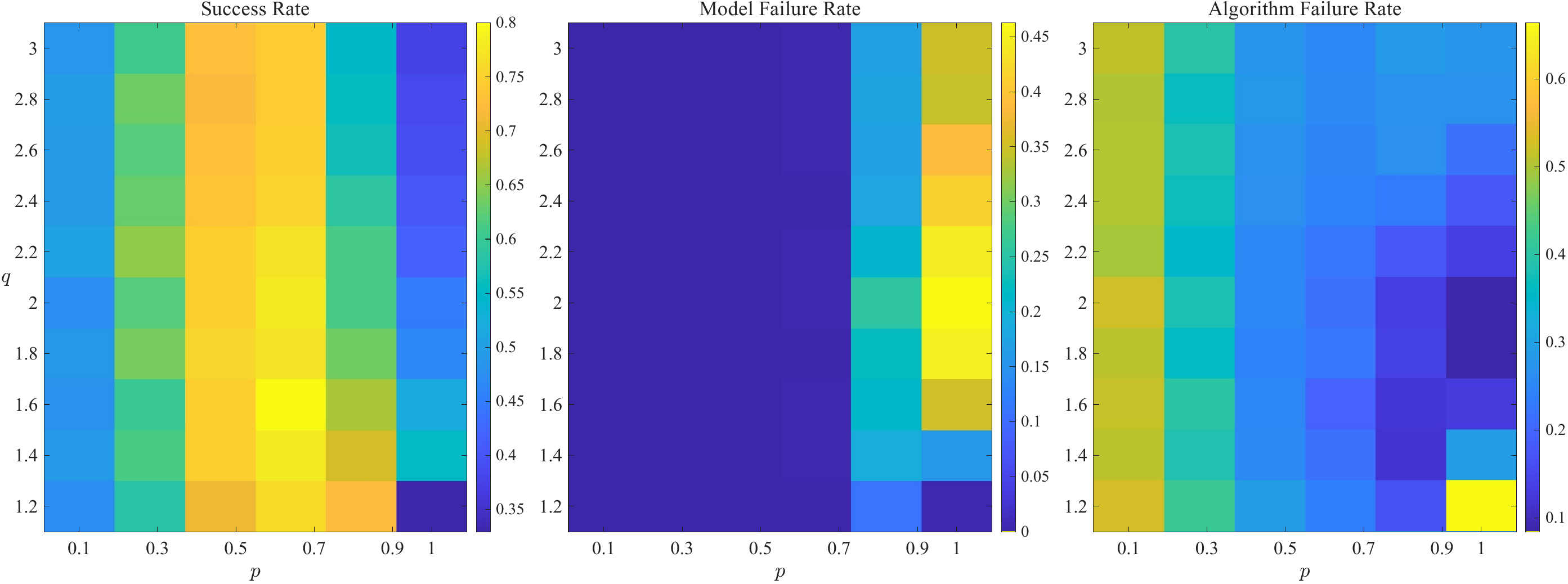}
		\caption{Gaussian matrix: $r=0.9$.}
		\label{fig:gauss-r09-heatmap}
	\end{subfigure}
	\caption{Parameter sensitivity of $\mathcal{Q}_{p,q}$ under correlated Gaussian sensing matrices. Each subfigure reports the average success rate, model-failure rate, and algorithm-failure rate over the tested sparsity levels.}
	\label{fig:gauss-parameter-heatmaps}
\end{figure}

Figure~\ref{fig:gauss-parameter-heatmaps} shows a clear and stable parameter pattern. Very small $p$ values, especially $p=0.1$, lead to relatively low success rates because the optimization landscape becomes difficult, as indicated by the high algorithm-failure rates. On the other hand, $p$ close to $1$ weakens the sparsity-promoting effect of the numerator and produces visible model failures, particularly at $p=1$. The most reliable region is the middle range $p=0.5$--$0.7$, where the success rate is the highest and both failure components are comparatively small. The choice of $q$ is less fragile than the choice of $p$, but values in the range $q\approx1.6$--$2.4$ usually provide the best trade-off. This behavior remains consistent when the Gaussian correlation increases from $r=0.3$ to $r=0.9$, suggesting that the favorable parameter region is not an artifact of a single sensing geometry.

\begin{figure}[htbp]
	\centering
	\includegraphics[scale=0.2]{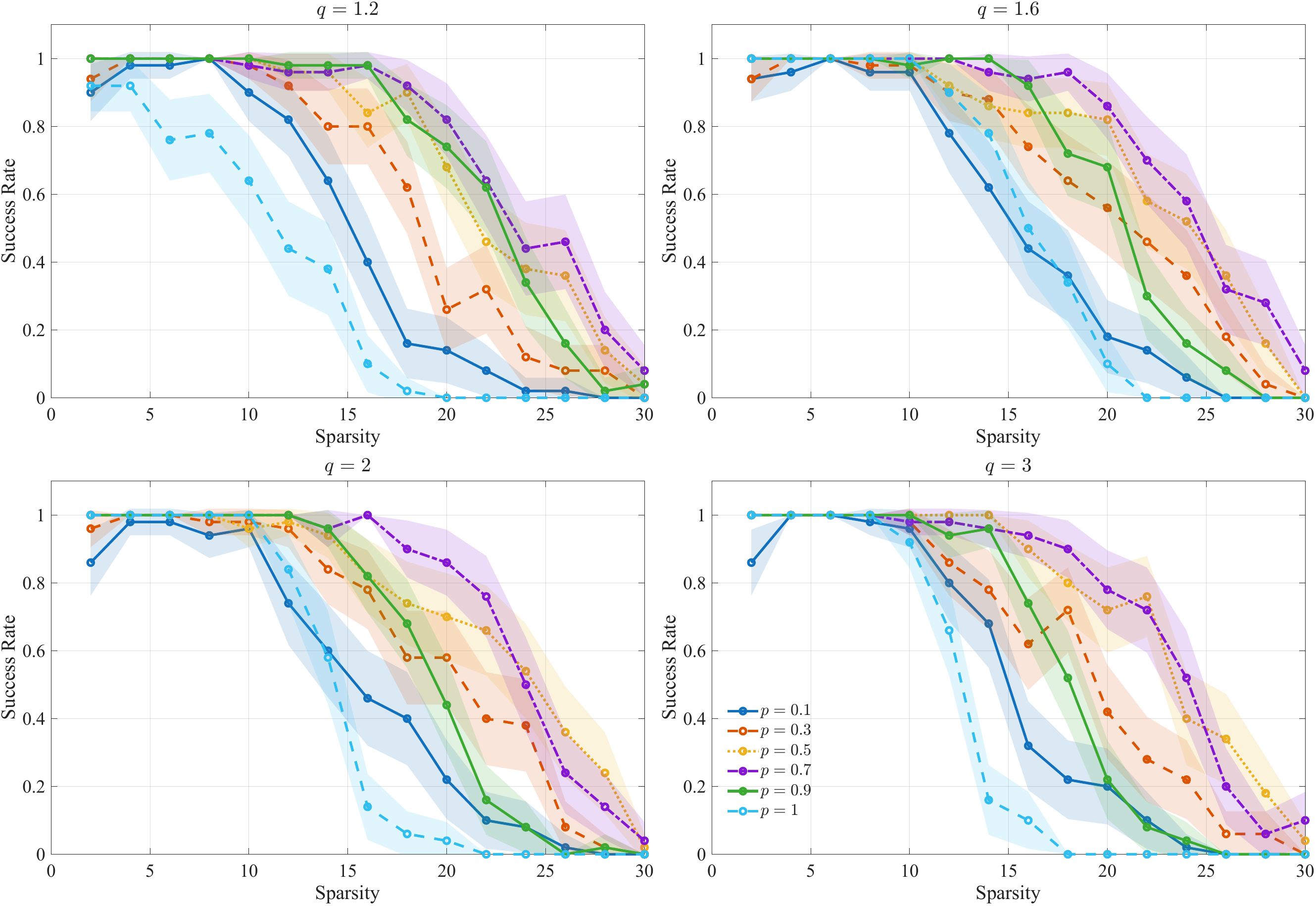}
	\caption{Success rate curves of $\mathcal{Q}_{p,q}$ under the Gaussian matrix with $r=0.3$. Each panel fixes $q$ and compares different $p$ values as the sparsity level increases.}
	\label{fig:gauss-r03-phase}
\end{figure}

\begin{figure}[htbp]
	\centering
	\includegraphics[scale=0.2]{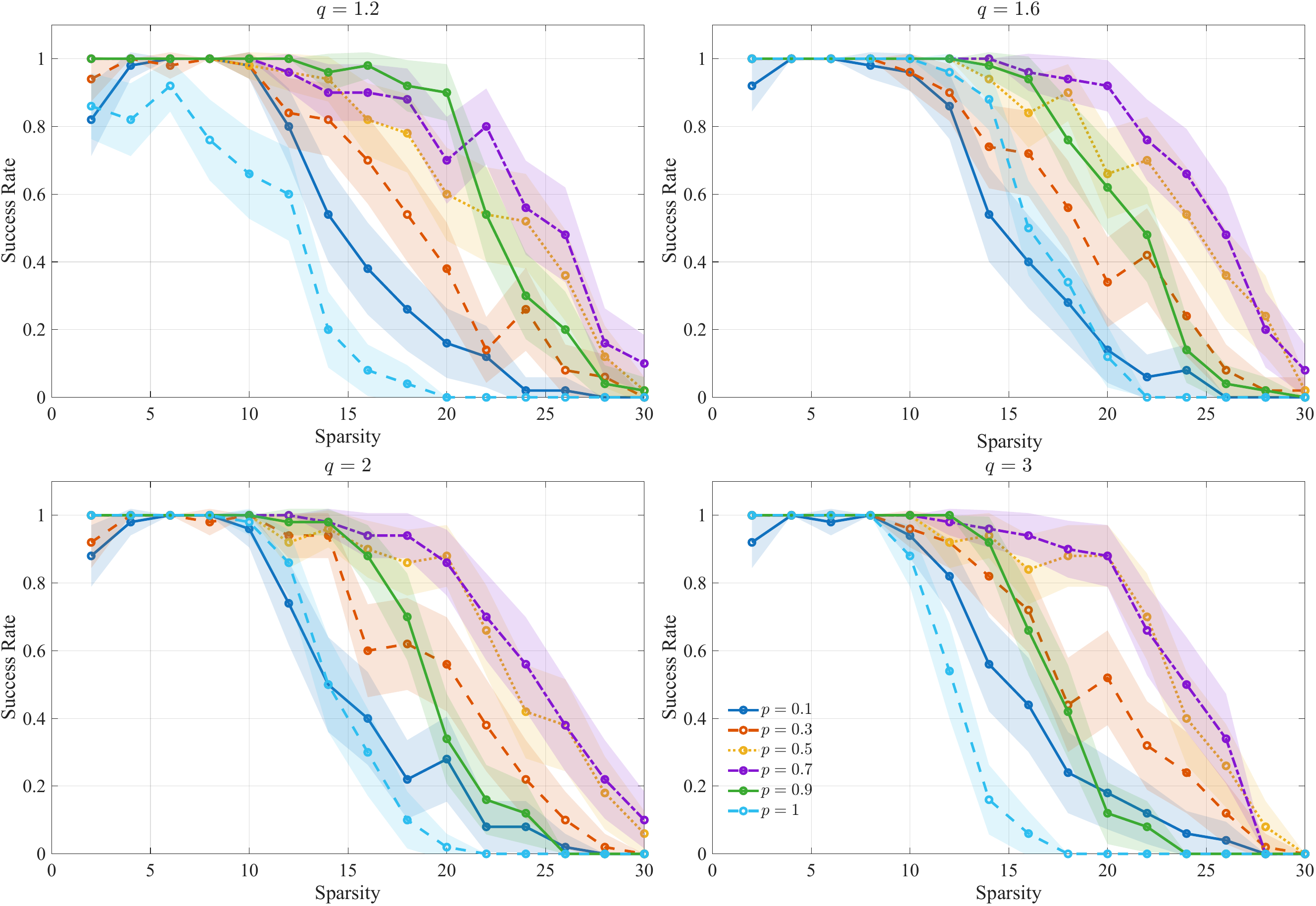}
	\caption{Success rate curves of $\mathcal{Q}_{p,q}$ under the Gaussian matrix with $r=0.9$.}
	\label{fig:gauss-r09-phase}
\end{figure}

Figures~\ref{fig:gauss-r03-phase} and~\ref{fig:gauss-r09-phase} provide a more detailed view of the success rate with respect to $k$. For small sparsity levels, almost all parameter choices recover the signal successfully. As $k$ increases, the curves separate. The choices $p=0.5$ and $p=0.7$ preserve high success rates over a wider sparsity range than both the more aggressive choice $p=0.1$ and the nearly convex choice $p=1$. Among them, $p=0.7$ is often the most robust when $q$ is moderate or large, while $p=0.5$ remains competitive for smaller $q$. The shaded confidence bands also show that the transition region becomes wider as the recovery problem becomes harder, but the ranking of the parameter choices is largely stable for both $r=0.3$ and $r=0.9$.

\begin{figure}[htbp]
	\centering
	\begin{subfigure}[b]{0.96\textwidth}
		\centering
		\includegraphics[scale=0.2]{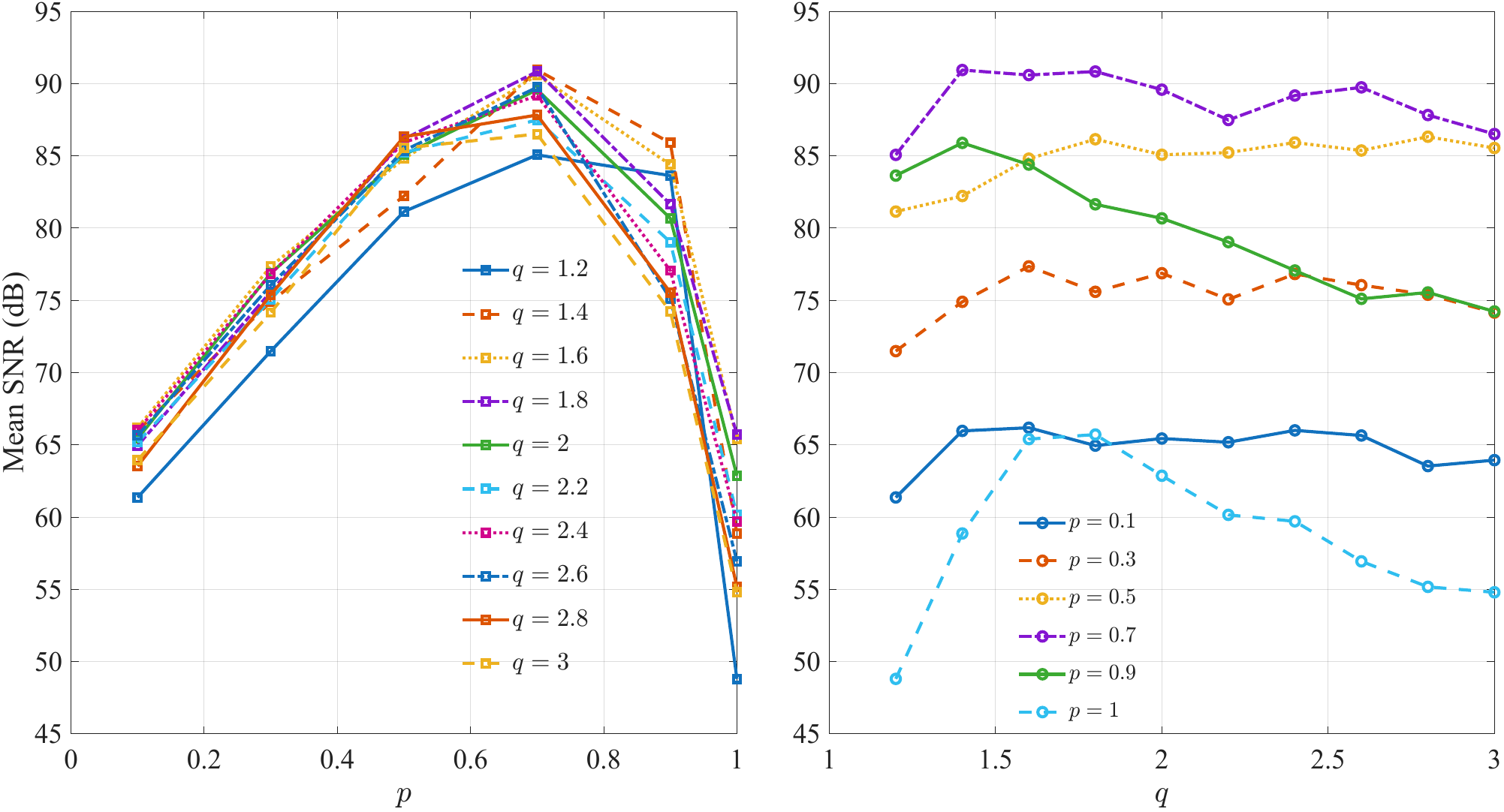}
		\caption{Gaussian matrix: $r=0.3$.}
		\label{fig:gauss-r03-snr}
	\end{subfigure}
	\begin{subfigure}[b]{0.96\textwidth}
		\centering
		\includegraphics[scale=0.2]{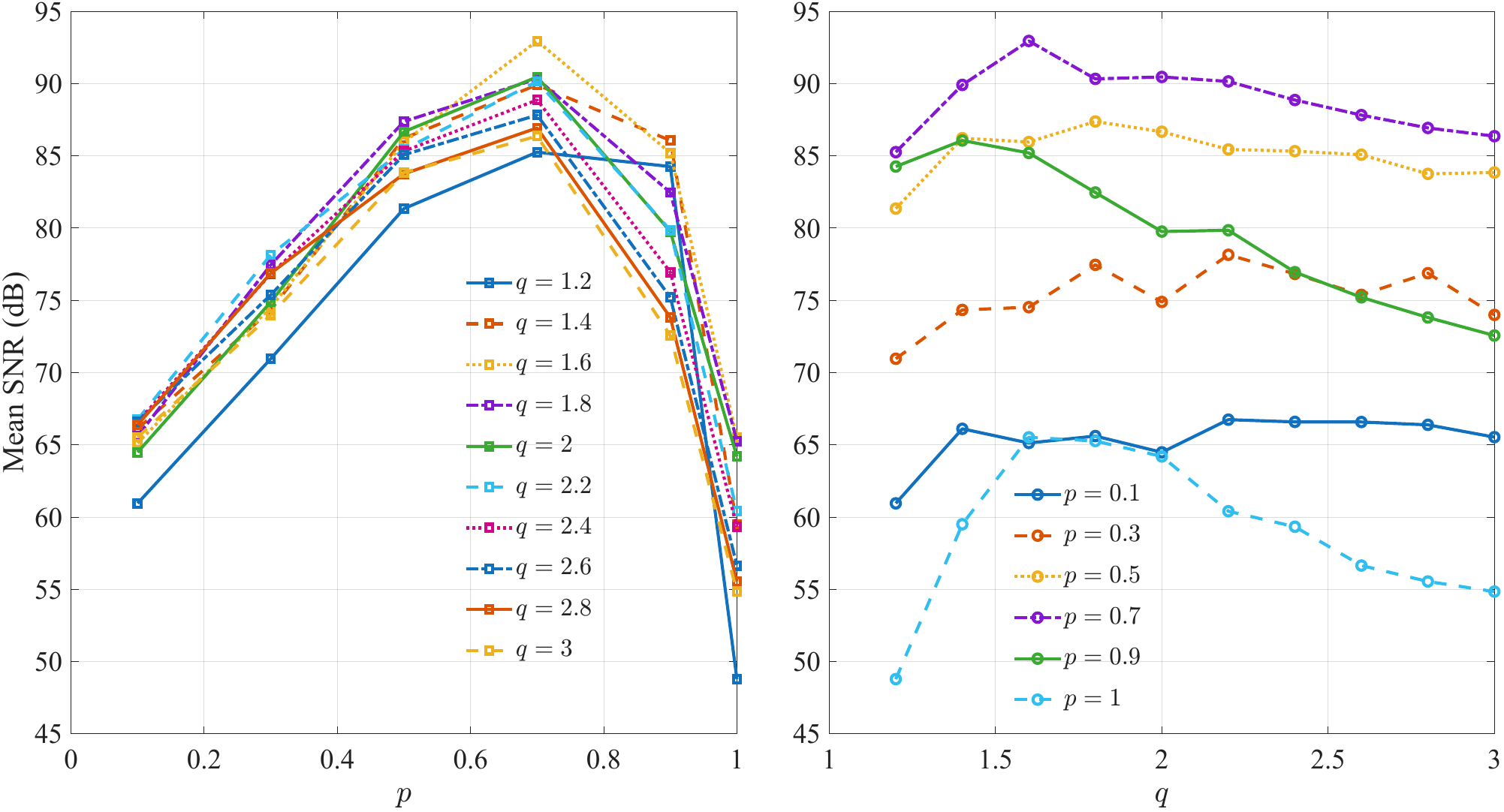}
		\caption{Gaussian matrix: $r=0.9$.}
		\label{fig:gauss-r09-snr}
	\end{subfigure}
	\caption{Mean reconstruction SNR of $\mathcal{Q}_{p,q}$ under correlated Gaussian sensing matrices.}
	\label{fig:gauss-snr}
\end{figure}

The SNR results in Figure~\ref{fig:gauss-snr} are consistent with the success rate analysis. Averaged over the tested sparsity levels, the best reconstruction fidelity is attained near $p=0.7$, with $p=0.5$ giving the second-best and very stable performance. In contrast, $p=1$ produces substantially lower SNR, while $p=0.1$ is also inferior despite its stronger sparsity bias. The dependence on $q$ is smoother: for $p=0.7$, the SNR stays high over a broad interval of $q$, whereas the performance for $p=0.9$ tends to decrease as $q$ grows. Overall, the empirical evidence supports the use of moderately nonconvex numerator exponents, especially $p=0.5$ or $p=0.7$, together with a denominator exponent around $q=1.5$--$2$.

\subsection{Comparison with Other Ratio-Type Models}

We next compare the proposed model with representative ratio-type sparse recovery models. The tested methods are $\ell_{0.5}/\ell_{1.5}$-DLPA, $\ell_{0.7}/\ell_{1.5}$-DLPA, $\ell_1/\ell_{1.5}$-CCP, $\ell_1/\ell_{\infty}$-ADMM, $\ell_1/\ell_{\infty}$-FISTA, $\ell_1/\ell_2$-A2, and $\ell_{0.5}/\ell_1$-FISTA. Here DLPA denotes the proposed Dinkelbach-type linearized proximal ADMM solver for $\mathcal{Q}_{p,q}$. For each method and each sparsity level, $100$ Monte Carlo trials are performed. We use the same success and failure classification as above.

\begin{figure}[htbp]
	\centering
	\begin{subfigure}[b]{0.96\textwidth}
		\centering
		\includegraphics[scale=0.3]{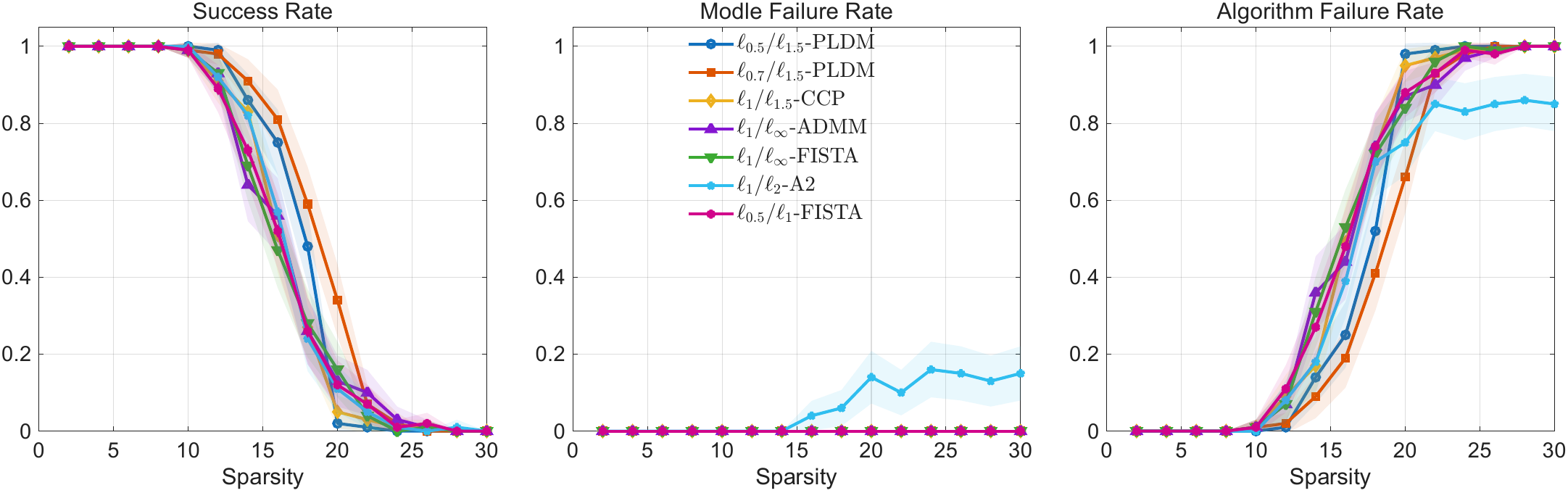}
		\caption{Oversampled DCT matrix: $F=10$.}
		\label{fig:comparison-dct-f10}
	\end{subfigure}
	\begin{subfigure}[b]{0.96\textwidth}
		\centering
		\includegraphics[scale=0.3]{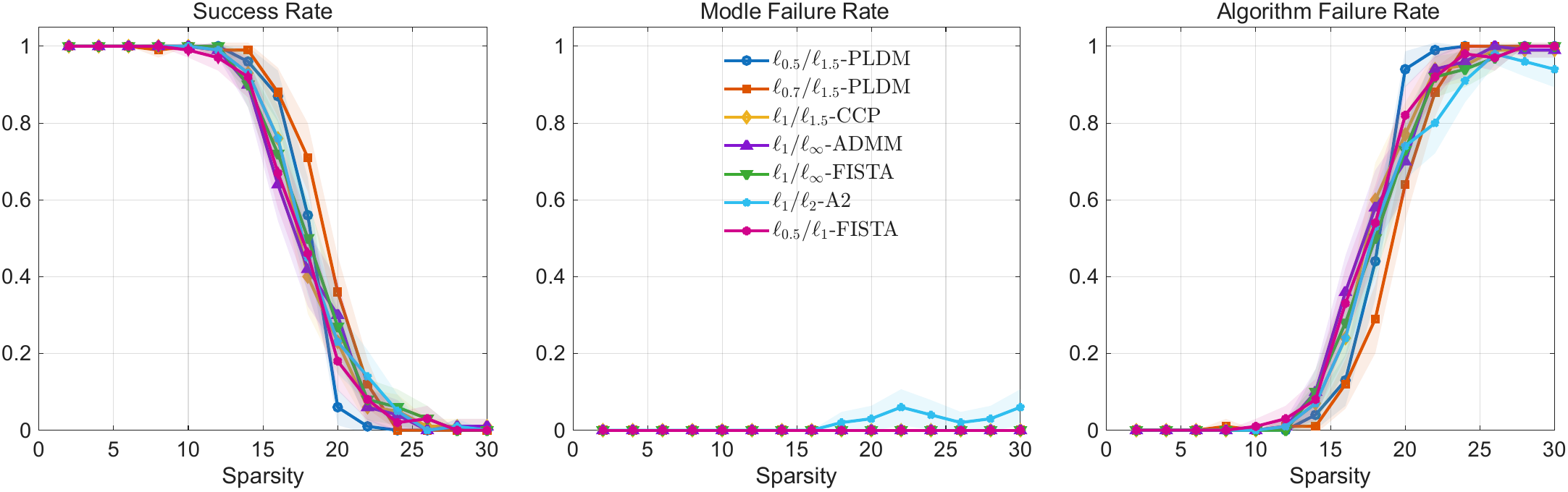}
		\caption{Oversampled DCT matrix: $F=15$.}
		\label{fig:comparison-dct-f15}
	\end{subfigure}
	\caption{Comparison of success rate, model-failure rate, and algorithm-failure rate on oversampled DCT matrices.}
	\label{fig:comparison-dct}
\end{figure}

Figure~\ref{fig:comparison-dct} reports the results for the oversampled DCT matrices, which are the most challenging tests because of strong column coherence. Both DLPA variants perform competitively, and $\ell_{0.7}/\ell_{1.5}$-DLPA gives the best overall experiments. For $F=10$, it retains a higher success rate than the baselines in the intermediate sparsity range, where most competing methods already deteriorate rapidly. This advantage becomes even more visible when the coherence is increased to $F=15$: $\ell_{0.7}/\ell_{1.5}$-DLPA maintains a high success probability for larger $k$, while most other methods enter the failure regime earlier. The middle panels show that model failures are nearly absent for most methods, except for a small contribution from $\ell_1/\ell_2$-A2 at large sparsity levels. Thus, the performance loss in the DCT experiments is mainly caused by algorithmic difficulty under high coherence rather than by a systematic mismatch between the model and the planted sparse signal.

\begin{figure}[htbp]
	\centering
	\begin{subfigure}[b]{0.96\textwidth}
		\centering
		\includegraphics[scale=0.3]{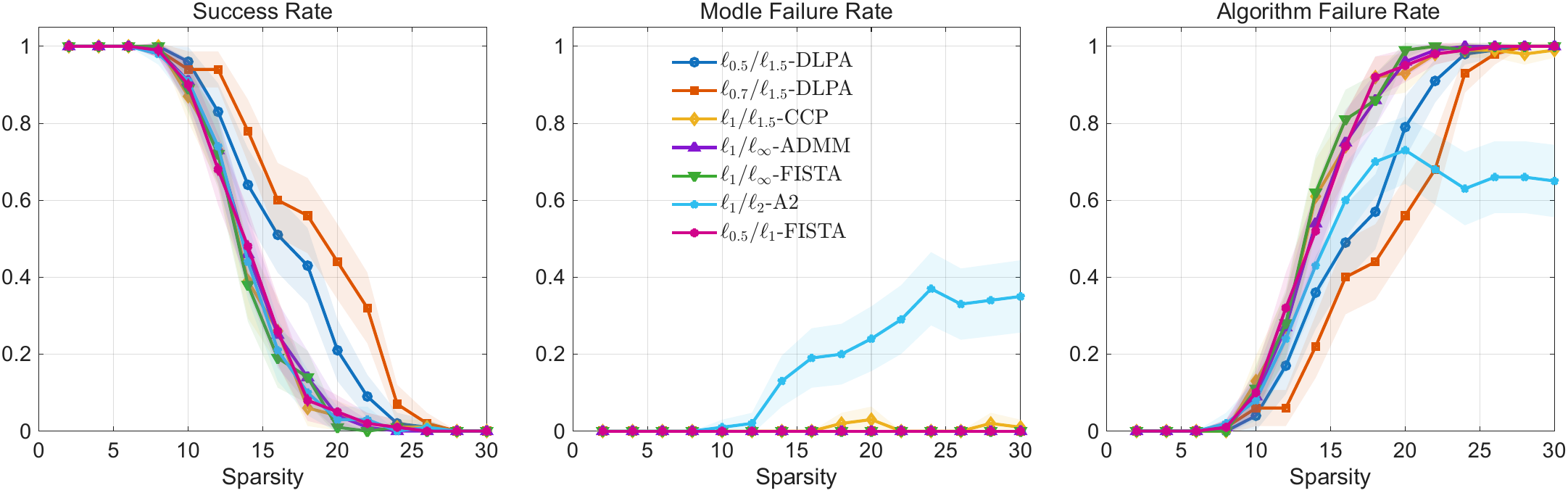}
		\caption{Correlated Gaussian matrix: $r=0.3$.}
		\label{fig:comparison-gauss-r03}
	\end{subfigure}
	\begin{subfigure}[b]{0.96\textwidth}
		\centering
		\includegraphics[scale=0.3]{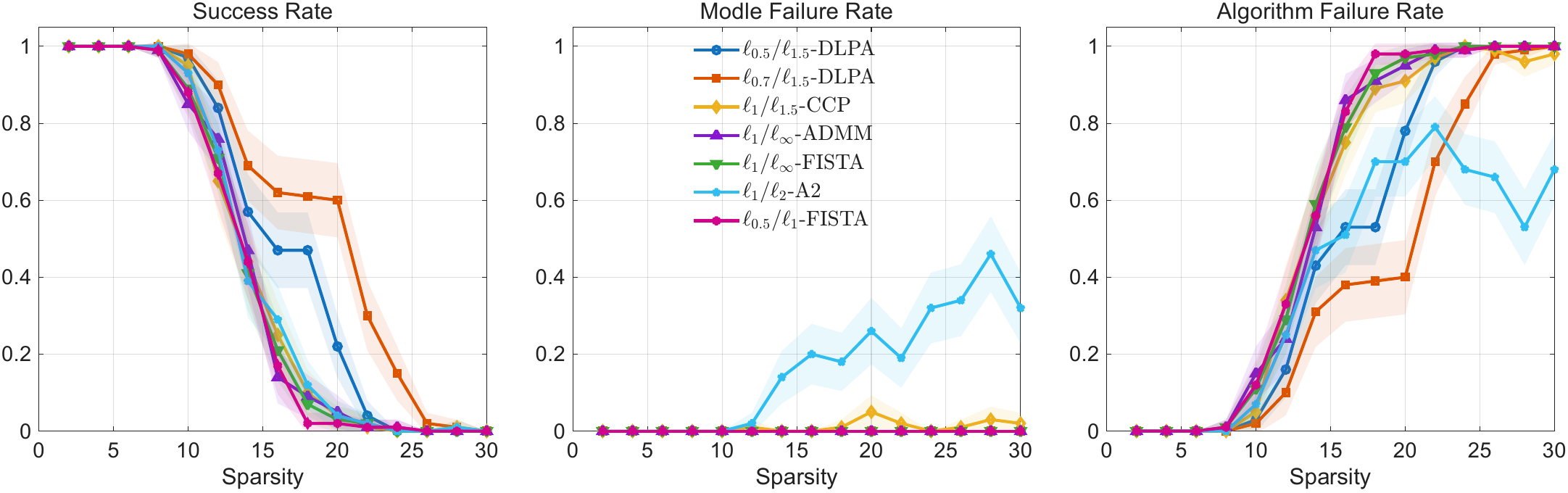}
		\caption{Correlated Gaussian matrix: $r=0.9$.}
		\label{fig:comparison-gauss-r09}
	\end{subfigure}
	\caption{Comparison of success rate, model-failure rate, and algorithm-failure rate on correlated Gaussian matrices.}
	\label{fig:comparison-gaussian}
\end{figure}

The Gaussian results in Figure~\ref{fig:comparison-gaussian} lead to the same conclusion. When $r=0.3$, $\ell_{0.7}/\ell_{1.5}$-DLPA has the slowest decay of success rate as $k$ grows, followed by $\ell_{0.5}/\ell_{1.5}$-DLPA. Several existing ratio-type methods perform well for very sparse signals but lose accuracy earlier in the transition region. When the correlation is increased to $r=0.9$, the advantage of the proposed model becomes more pronounced: $\ell_{0.7}/\ell_{1.5}$-DLPA keeps a noticeably higher success rate over moderate sparsity levels, whereas the baseline methods transition more sharply to failure. As in the DCT case, the dominant failure source is algorithmic failure. The only substantial model-failure component is again associated with the $\ell_1/\ell_2$-A2 baseline, indicating that the proposed $\ell_p^p/\ell_q^p$ formulation with $p<1$ is less prone to selecting a lower-objective but incorrect solution in these experiments.

Taken together, the experiments indicate that the proposed PLDM algorithm for solving the $\mathcal{Q}_{p,q}$-minimization problem is particularly effective in a moderately nonconvex regime. The parameter choices $(p,q)=(0.5,1.5)$ and $(0.7,1.5)$ are both robust, with $(0.7,1.5)$ delivering the strongest overall performance in the cross-model comparisons. The results also confirm the benefit of the generalized ratio structure under high dynamic range and coherent sensing matrices, where classical ratio models such as $\ell_1/\ell_2$, $\ell_1/\ell_\infty$, and $\ell_p/\ell_1$ tend to suffer earlier phase transitions or larger failure rates.

\section{Conclusion and future work}\label{sec:Conclusion}
This paper studied the generalized norm-ratio minimization model $\ell_p^p/\ell_q^p$ for compressed sensing. We established a local optimality criterion and a uniform exact recovery condition based on a null-space norm-ratio lower bound, and further derived high-probability sampling guarantees for the $\ell_1/\ell_q$ case under isotropic sub-Gaussian measurements. For stable recovery, we sharpened the RIP-based analysis of \eqref{eq:lp_lq_eps}: in the exact $k$-sparse case, the proposed bound gives a weaker RIP requirement and a smaller reconstruction error estimate, and further, in the compressible case it removes the explicit ambient-dimension dependence in earlier RIP--ROP estimates and also yields an RIP-only guarantee. These results provide theoretical support for $\ell_p^p/\ell_q^p$ minimization as both a sparsity-promoting and stable recovery model.

We also proposed a prox-linear Dinkelbach scheme for the constrained ratio problem and proved its convergence. Numerical experiments on Gaussian and oversampled DCT sensing matrices show that suitable choices of $(p,q)$ improve empirical recovery performance, especially for high-dynamic-range sparse signals and coherent measurements. Future work includes extending the sampling-complexity theory to the nonconvex range $0<p<1$, developing convergence-rate results for the Dinkelbach updates, and automatic selection of $(p,q)$.
%
\bibliographystyle{elsarticle-num}   
\bibliography{manuscriptbib}   
 \end{document}